\documentclass[11pt]{amsart}
\usepackage{amsmath,amsthm,amsfonts,amssymb,latexsym,enumerate,url,hyperref} 
\usepackage[dvipsnames]{xcolor}

\usepackage[top=30mm,right=30mm,bottom=30mm,left=20mm]{geometry}

\begin{document}

\theoremstyle{plain}

\newtheorem{thm}{Theorem}[section]
\newtheorem{lem}[thm]{Lemma}
\newtheorem{conj}[thm]{Conjecture}
\newtheorem{pro}[thm]{Proposition}
\newtheorem{cor}[thm]{Corollary}
\newtheorem{que}[thm]{Question}
\newtheorem{rem}[thm]{Remark}
\newtheorem{defi}[thm]{Definition}
\newtheorem{hyp}[thm]{Hypothesis}

\newtheorem*{thmA}{THEOREM A}
\newtheorem*{thmB}{THEOREM B}
\newtheorem*{corB}{COROLLARY B}
\newtheorem*{conjA}{CONJECTURE A}
\newtheorem*{conjB}{CONJECTURE B}
\newtheorem*{conjE}{CONJECTURE E}
\newtheorem*{thmC}{THEOREM C}
\newtheorem*{thmD}{THEOREM D}
\newtheorem*{conjC}{CONJECTURE C}
\newtheorem*{conj0}{CONJECTURE}

\newtheorem*{thmAcl}{Main Theorem$^{*}$}
\newtheorem*{thmBcl}{Theorem B$^{*}$}

\numberwithin{equation}{section}

\newcommand{\Maxn}{\operatorname{Max_{\textbf{N}}}}
\newcommand{\Syl}{\operatorname{Syl}}
\newcommand{\dl}{\operatorname{\mathfrak{d}}}
\newcommand{\Con}{\operatorname{Con}}
\newcommand{\cl}{\operatorname{cl}}
\newcommand{\Stab}{\operatorname{Stab}}
\newcommand{\Aut}{\operatorname{Aut}}
\newcommand{\Ker}{\operatorname{Ker}}
\newcommand{\IBr}{\operatorname{IBr}}
\newcommand{\Irr}{\operatorname{Irr}}
\newcommand{\SL}{\operatorname{SL}}
\newcommand{\FF}{\mathbb{F}}
\newcommand{\NN}{\mathbb{N}}
\newcommand{\N}{\mathbf{N}}
\newcommand{\C}{\mathbf{C}}
\newcommand{\OO}{\mathbf{O}}
\newcommand{\F}{\mathbf{F}}

\renewcommand{\labelenumi}{\upshape (\roman{enumi})}

\newcommand{\GL}{\operatorname{GL}}
\newcommand{\Sp}{\operatorname{Sp}}
\newcommand{\PGL}{\operatorname{PGL}}
\newcommand{\PSL}{\operatorname{PSL}}
\newcommand{\GU}{\operatorname{GU}}
\newcommand{\PGU}{\operatorname{PGU}}
\newcommand{\SU}{\operatorname{SU}}
\newcommand{\PSU}{\operatorname{PSU}}
\newcommand{\PSp}{\operatorname{PSp}}

\providecommand{\V}{\mathrm{V}}
\providecommand{\E}{\mathrm{E}}
\providecommand{\ir}{\mathrm{Irr_{rv}}}
\providecommand{\Irrr}{\mathrm{Irr_{rv}}}
\providecommand{\re}{\mathrm{Re}}

\def\irrp#1{{\rm Irr}_{p'}(#1)}

\def\Z{{\mathbb Z}}
\def\C{{\mathbb C}}
\def\Q{{\mathbb Q}}
\def\irr#1{{\rm Irr}(#1)}
\def\ibr#1{{\rm IBr}(#1)}
\def\irrv#1{{\rm Irr}_{\rm rv}(#1)}
\def \c#1{{\mathcal #1}}
\def\cent#1#2{{\bf C}_{#1}(#2)}
\def\syl#1#2{{\rm Syl}_#1(#2)}
\def\nor{\trianglelefteq\,}
\def\oh#1#2{{\bf O}_{#1}(#2)}
\def\Oh#1#2{{\bf O}^{#1}(#2)}
\def\zent#1{{\bf Z}(#1)}
\def\det#1{{\rm det}(#1)}
\def\aut#1{{\rm Aut}(#1)}
\def\ker#1{{\rm ker}(#1)}
\def\norm#1#2{{\bf N}_{#1}(#2)}
\def\alt#1{{\rm Alt}(#1)}
\def\iitem#1{\goodbreak\par\noindent{\bf #1}}
   \def \mod#1{\, {\rm mod} \, #1 \, }
\def\sbs{\subseteq}

\def\gc{{\bf GC}}
\def\ngc{{non-{\bf GC}}}
\def\ngcs{{non-{\bf GC}$^*$}}
\newcommand{\notd}{{\!\not{|}}}
\newcommand{\Out}{{\mathrm {Out}}}
\newcommand{\Mult}{{\mathrm {Mult}}}
\newcommand{\Inn}{{\mathrm {Inn}}}
\newcommand{\IBR}{{\mathrm {IBr}}}
\newcommand{\IBRL}{{\mathrm {IBr}}_{\ell}}
\newcommand{\IBRP}{{\mathrm {IBr}}_{p}}
\newcommand{\ord}{{\mathrm {ord}}}
\def\id{\mathop{\mathrm{ id}}\nolimits}
\renewcommand{\Im}{{\mathrm {Im}}}
\newcommand{\Ind}{{\mathrm {Ind}}}
\newcommand{\diag}{{\mathrm {diag}}}
\newcommand{\soc}{{\mathrm {soc}}}
\newcommand{\End}{{\mathrm {End}}}
\newcommand{\sol}{{\mathrm {sol}}}
\newcommand{\Hom}{{\mathrm {Hom}}}
\newcommand{\Mor}{{\mathrm {Mor}}}
\newcommand{\St}{{\sf {St}}}
\def\rank{\mathop{\mathrm{ rank}}\nolimits}
\newcommand{\Tr}{{\mathrm {Tr}}}
\newcommand{\tr}{{\mathrm {tr}}}
\newcommand{\Gal}{{\it Gal}}
\newcommand{\Spec}{{\mathrm {Spec}}}
\newcommand{\ad}{{\mathrm {ad}}}
\newcommand{\Sym}{{\mathrm {Sym}}}
\newcommand{\Char}{{\mathrm {char}}}
\newcommand{\pr}{{\mathrm {pr}}}
\newcommand{\rad}{{\mathrm {rad}}}
\newcommand{\abel}{{\mathrm {abel}}}
\newcommand{\codim}{{\mathrm {codim}}}
\newcommand{\ind}{{\mathrm {ind}}}
\newcommand{\Res}{{\mathrm {Res}}}
\newcommand{\Ann}{{\mathrm {Ann}}}
\newcommand{\Ext}{{\mathrm {Ext}}}
\newcommand{\Alt}{{\mathrm {Alt}}}
\newcommand{\AAA}{{\sf A}}
\newcommand{\SSS}{{\sf S}}
\newcommand{\CC}{{\mathbb C}}
\newcommand{\CB}{{\mathbf C}}
\newcommand{\RR}{{\mathbb R}}
\newcommand{\QQ}{{\mathbb Q}}
\newcommand{\ZZ}{{\mathbb Z}}
\newcommand{\KK}{{\mathbb K}}
\newcommand{\NB}{{\mathbf N}}
\newcommand{\ZB}{{\mathbf Z}}
\newcommand{\OB}{{\mathbf O}}
\newcommand{\EE}{{\mathbb E}}
\newcommand{\PP}{{\mathbb P}}
\newcommand{\GC}{{\mathbf G}}
\newcommand{\HC}{{\mathcal H}}
\newcommand{\AC}{{\mathcal A}}
\newcommand{\BC}{{\mathcal B}}
\newcommand{\GA}{{\mathfrak G}}
\newcommand{\SC}{{\mathcal S}}
\newcommand{\TC}{{\mathbf T}}
\newcommand{\DC}{{\mathcal D}}
\newcommand{\LC}{{\mathcal L}}
\newcommand{\RC}{{\mathcal R}}
\newcommand{\CL}{{\mathcal C}}
\newcommand{\EC}{{\mathcal E}}
\newcommand{\GCD}{\GC^{*}}
\newcommand{\TCD}{\TC^{*}}
\newcommand{\FD}{F^{*}}
\newcommand{\GD}{G^{*}}
\newcommand{\HD}{H^{*}}
\newcommand{\hG}{\hat{G}}
\newcommand{\hP}{\hat{P}}
\newcommand{\hQ}{\hat{Q}}
\newcommand{\hR}{\hat{R}}
\newcommand{\GCF}{\GC^{F}}
\newcommand{\TCF}{\TC^{F}}
\newcommand{\PCF}{\PC^{F}}
\newcommand{\GCDF}{(\GC^{*})^{F^{*}}}
\newcommand{\RGTT}{R^{\GC}_{\TC}(\theta)}
\newcommand{\RGTA}{R^{\GC}_{\TC}(1)}
\newcommand{\Om}{\Omega}
\newcommand{\eps}{\epsilon}
\newcommand{\varep}{\varepsilon}
\newcommand{\al}{\alpha}
\newcommand{\chis}{\chi_{s}}
\newcommand{\sigmad}{\sigma^{*}}
\newcommand{\PA}{\boldsymbol{\alpha}}
\newcommand{\gam}{\gamma}
\newcommand{\lam}{\lambda}
\newcommand{\la}{\langle}
\newcommand{\ra}{\rangle}
\newcommand{\hs}{\hat{s}}
\newcommand{\htt}{\hat{t}}
\newcommand{\sgn}{\mathsf{sgn}}
\newcommand{\SR}{^*R}
\newcommand{\tsig}{\boldsymbol{\sigma}}
\newcommand{\trho}{\boldsymbol{\rho}}
\newcommand{\teps}{\boldsymbol{\epsilon}}
\newcommand{\tvarep}{\boldsymbol{\varepsilon}}
\newcommand{\tL}{\tilde{L}}
\newcommand{\tM}{\tilde{M}}
\newcommand{\tb}{\tilde{b}}
\newcommand{\tUb}{\tilde{U}}
\newcommand{\ta}{\tilde{a}}
\newcommand{\tchi}{\tilde\chi}
\newcommand{\tw}[1]{{}^#1\!}
\renewcommand{\mod}{\bmod \,}
\newcommand{\relc}{\,{\boldsymbol{\succeq_{\sf c}}}\,}
\newcommand{\edit}[1]{{\color{red} #1}}

\def\Par{{\mathcal P}}
\def\OP{{\mathcal O}}
\def\DP{{\mathcal D}}
\newcommand{\0}{{\bar 0}}
\newcommand{\1}{{\bar 1}}
\newcommand{\La}{\Lambda}
\newcommand{\sfS}{{\sf S}}
\newcommand{\sfs}{{\sf s}}
\newcommand{\ts}{\tilde{\sf s}}
\newcommand{\tildet}{\tilde{\sf t}}
\newcommand{\tz}{{\sf z}}

\newcommand{\sfC}{{\sf C}}
\newcommand{\sfA}{{\sf A}}
\newcommand{\TS}{\tilde{\sf S}}
\newcommand{\TA}{\tilde{\sf A}}
\newcommand{\TiC}{\tilde{\sf C}}

\marginparsep-0.5cm

\renewcommand{\thefootnote}{\fnsymbol{footnote}}
\footnotesep6.5pt

\title{A reduction theorem for  the Feit conjecture}

\author[R. Boltje]{Robert Boltje}
\address[R. Boltje]{Department of Mathematics, University of California Santa Cruz, CA 95064, USA}
\email{boltje@ucsc.edu}

\author[A. Kleshchev]{Alexander Kleshchev}
\address[A. Kleshchev]{Department of Mathematics, University of Oregon, Eugene, OR~97403, USA}
\email{klesh@uoregon.edu}

\author[G. Navarro]{Gabriel Navarro}
\address[G. Navarro]{Departament de Matem\`atiques, Universitat de Val\`encia, 46100 Burjassot,
Val\`encia, Spain}
\email{gabriel@uv.es}

\author[P. H. Tiep]{Pham Huu Tiep}
\address[P. H. Tiep]{Department of Mathematics, Rutgers University, Piscataway,
  NJ 08854, USA}
\email{tiep@math.rutgers.edu}

\date{}

\thanks{The research of the first author is supported by Grant  PID2022-137612NB-I00 funded by MCIN/AEI/ 10.13039/501100011033 and ERDF ``A way of making Europe" and CDEIGENT grant CIDEIG/2022/29 funded by Generalitat Valenciana.}

\thanks{The second author is supported by the NSF grant DMS-2101791.}

\thanks{The fourth author gratefully acknowledges the support of the NSF (grant
DMS-2200850) and the Joshua Barlaz Chair in Mathematics.}

\thanks{The first author would like to thank the Mathematics Department of the University of Valencia
for their hospitality during a four-month visit.}

\thanks{Part of this work was done while the third author visited the Department of Mathematics, Rutgers University. 
It is a pleasure to thank Rutgers University for generous hospitality and stimulating environment.
We thank Carlos Tapp-Monfort for a careful reading of the manuscript. We also thank the referee for numerous helpful comments which have improved the readability of this paper.}



\begin{abstract} 
We prove that if all the simple groups involved in a finite group $G$ satisfy the 
{\sl inductive Feit condition}, then Walter Feit's conjecture from 1980 holds for $G$.
In particular, this would solve Brauer's Problem 41 from 1963 in the affirmative.
This inductive Feit condition implies that some features of all 
the irreducible characters of finite groups can be found locally.
\end{abstract}

\maketitle

\tableofcontents 
 
\section{Introduction} 
Already in 1904 H. Blichfeldt \cite{Bl}, and W. Burnside in 1911 \cite{Bu},   found a relationship between the irrationalities of
an irreducible complex character $\chi \in \irr G$ of a finite group $G$ and the existence of some elements in $G$ of a certain order: 
if the conductor of $\chi(g_i)$ equals $p_i^{a_i}$, where $p_1, \ldots, p_n$ are different primes and $g_i \in G$, then
$G$ possesses elements of order $p_1^{a_1} \cdots p_n^{a_n}$. 
(Recall that if $K$ is an abelian number field, then the conductor of $K$ is the smallest integer $n>0$ such that $K$ is contained in the
$n^{\mathrm{th}}$ cyclotomic field $\Q_n$; the conductor of $\alpha \in K$ is the conductor of $\Q(\alpha)$, and the conductor $c(\chi)$ 
of $\chi$ is the conductor of the smallest field $\Q(\chi)$ containing all of its values.)

The situation is more difficult if conductors of individual character values are not just prime powers. 
In 1962, R. Brauer provided a new method which allowed him to generalize the previous result and to obtain some others, cf.~\cite{Br64}. 
However this was not powerful enough to answer, for instance, his (still open) Problem 41 in the famous list of problems \cite{Br63}.

These results and Problem 41 were unified in a single question in 1980 by W.~Feit in his well-known paper 
that gave a list of the main problems that remained open after the Classification of the Finite Simple Groups \cite{F80}. 
  Later this became known as Feit's conjecture.
\medskip
\begin{conj0}~{\rm (Feit)}
Let $G$ be a finite group and let $\chi \in \irr G$.
Then there exists  $g \in G$ such that the order $o(g)$ is the conductor of $\chi$.
\end{conj0}
 
Feit's conjecture was proven for solvable groups independently in \cite{AC86} and 
\cite{FT86}. Further results for other families of groups were proven in \cite{FT87}, but since the end of the 1980's no significant progress on Feit's conjecture was made.
One of the main problems of the conjecture is that inductive methods do not  seem to work properly.

\smallskip
Our new strategy here is to confront the irreducible character $\chi$ of $G$ not against order of elements
of $G$, but
against other characters of proper subgroups of $G$. This approach, together with an ad hoc use of the recent methods
that have successfully led to the proof of the McKay conjecture and a reduction of its Galois version to simple groups
(see \cite{IMN}, \cite{NS}, \cite{NSV}, \cite{N}), allows us to prove the following.

\begin{thmA}\label{thmA} 
Suppose that $G$ is a finite group, $\chi \in \irr G$ with $\chi(1)>1$. 
Assume that the non-abelian composition factors of $G$ satisfy the inductive Feit condition. 
If $\chi$ is primitive, then there exists $H=\norm GH<G$ and $\psi \in \irr H$ such that $\Q(\chi)=\Q(\psi)$.
\end{thmA}

We remark that the character $\psi$ in Theorem A needs not to be a constituent of the restriction $\chi_H$, as shown, for instance, by $G={\rm SL}_2(3)$. At the time of this writing, we do not know of any example
where $H$ is not a maximal subgroup of $G$.

\medskip
We easily derive the following.

\begin{corB}
Suppose that $G$ is a finite group. If all the simple groups involved in $G$ satisfy the inductive Feit condition, then
Feit's conjecture holds for $G$.
\end{corB}

The precise definition of the inductive Feit condition is given in Definition \ref{inductive}, and it builds on the
inductive McKay condition and its Galois version. Roughly speaking, if $X$ is a quasi-simple group
and $\chi \in \irr X$, we want to find a proper self-normalizing intravariant subgroup $U$ of $X$ with a character $\mu \in \irr U$ 
such that $\chi$ and $\mu$ have the same cohomology and the same behavior under the
action of $\aut X_U\times \c G$, where $\c G$ is the absolute Galois group
and $\aut X_U$ is the group of automorphisms of $X$ that fix $U$. This may sound innocent, but in reality lies quite deep. It implies that relevant features of the irreducible characters of every finite group can be found in convenient proper self-normalizing 
subgroups of $G$.  (See for instance Corollary \ref{key3}.) Feit's conjecture would merely be   one application of  this theory.

As we are implying, the inductive Feit condition shares some similarities with the inductive  McKay condition \cite{IMN} and the inductive Galois--McKay condition \cite{NSV}, but there are also significant differences. First,
the inductive Feit condition applies to every irreducible character of $G$, not only to characters 
of degree coprime to a given prime $p$. Secondly, the target ``corresponding subgroup'' is  less specific (suitable self-normalizing subgroups instead of   $p$-Sylow normalizers). Thirdly, the fields in consideration are the smallest possible (rationals instead of the $p$-adics).

    
\smallskip
Of course, this leaves the stimulating task of proving that the inductive Feit condition holds for all simple groups. The aforementioned differences of the
inductive Feit condition make this task even more challenging.
In Section 7,  we accomplish this for the alternating groups.

\begin{thmC}
For any $n \geq 5$, the simple group $\AAA_n$ satisfies the inductive Feit condition.
\end{thmC}
Also, we prove that the sporadic groups and certain families of simple groups of Lie type
satisfy the inductive Feit condition  (see Theorems \ref{spo},   \ref{sl22}, \ref{sl21}, \ref{ree}).
(The inductive Feit condition for further families of simple groups of Lie type will be studied elsewhere.
For instance, see \cite{CT}.)  
Some of the new ideas in the proofs of these results, in particular various parts of Theorem \ref{easysituations},
should be useful for the verification of the inductive conditions for other global-local conjectures as well. 
As a consequence of this, and to show the power of our method, we   offer here the following.

\begin{thmD}
Suppose that $G$ is a finite group with abelian Sylow $2$-subgroups. Then Feit's conjecture holds for $G$.
\end{thmD}

In the final section of this paper, we consider a generalization of Feit's conjecture which we consider perhaps more natural.

\begin{conjE}
Suppose that $G$ is a finite group and $\chi \in \irr G$.
Then there is a subgroup $H$ of $G$ and a linear constituent 
$\lambda \in \irr H$ of the restriction $\chi_H$ such that the conductor of $\chi$ is the order of $\lambda$.
\end{conjE}

Conjecture E is stronger than Feit's conjecture.
Indeed,  it is not in general true that if $n$ is the order of an element of $G$
and $\chi \in \irr G$, then there exists a linear $\lambda$ constituent of $\chi_H$ 
of order divisible by $n$, even for faithful $\chi$. Examples are ${\sf S}_5$ with $\chi$ of degree 6, 
or for solvable groups, ${\tt SmallGroup}(324,160)$ with one of the characters of degree 6.
We will prove Conjecture E for solvable groups, and explore some other formulations which might have some interest.

In a forthcoming paper, \cite{BN}, it is shown that if, for given $\chi$, pairs $(H,\lambda)$ as in Conjecture~E exist at all then at least one such pair must occur in the canonical Brauer induction formula of $\chi$. Moreover, a natural indicator function on $\irr G$ with values in non-negative integers is constructed via Adams operations, whose value at $\chi\in\irr G$ is positive if and only if Conjecture~E holds for $\chi$.


\section{$\c G$-triples}

In this section, we use the ideas of the reduction of the Galois-McKay conjecture and borrow the notation from \cite{NSV}.  
For the reader's convenience we recall the definition of $\c G$-triples, related notions, properties, and results,
which will be used heavily in Sections~3 and 4. Some new results on projective representations and $\c G$-triples are also obtained, which should be useful for checking other inductive conditions.
\medskip

We use the notation for characters in \cite{Is} and \cite{N}. 
Among other standard results, we frequently use Clifford's theorem, the Clifford correspondence or Gallagher's theorem (Theorems 6.2, 6.11, and Corollary 6.17 of \cite{Is}, respectively) without further reference.

\medskip
Let $\Q^{\mathrm{ab}}$ denote the field generated by all roots of unity in $\C$. Throughout this section, 
we fix an intermediate field $\Q\sbs F \sbs \Q^{\mathrm{ab}}$ and set $\c G:=\mathrm{Gal}(\Q^{\mathrm{ab}}/F)$ which is known to be 
an abelian group. From Section~3 on we will assume that $F=\Q$ and $\c G=\mathrm{Gal}(\Q^{\mathrm{ab}}/\Q)$.

Let $X$ be a finite group. Then $\c G$ acts on $\irr X$ via $(\chi^\sigma)(x):= \chi(x)^\sigma:=\sigma(\chi(x))$,
for $x\in X$, $\sigma\in\c G$ and $\chi\in\Irr(X)$.

Note that we view $\aut X$ as a group with multiplication defined by $\alpha\beta:=
\beta\circ \alpha$. Thus, $x^\alpha:=\alpha(x)$ defines a right action of $\aut X$ on $X$, via group automorphisms.

More generally, if a group $\Gamma$ acts on $X$ from the right via automorphisms $x\mapsto x^\gamma$,
we consider the induced action of $\Gamma$ on $\irr X$ via $\theta^\gamma(x^\gamma):=
\theta(x)$ for $x\in X$, $\gamma\in\Gamma$. Note that this action commutes with the action of $\c G$ on $\irr X$, so that
$\Gamma\times \c G$ acts on $\irr X$, where $\chi^{(\gamma,\sigma)}(x)=\sigma(\chi^\gamma(x))$, for $x\in X$, 
$\chi\in \irr X$ and $\gamma\in\Gamma$. We will often just write $\chi^{\gamma\sigma}$ for $\chi^{(\gamma,\sigma)}$. 
For $\theta\in\irr X$ we may consider the stabilizer $(\Gamma\times \c G)_\theta$ of $\theta$ in $\Gamma\times \c G$.
We denote its first and second projection by $\Gamma_{\theta^{\c G}}$ and ${\c G}_{\theta^\Gamma}$, respectively, since these
are the respective stabilizers of the orbit of $\theta$ under the action of the other groups. 
More explicitly, $\Gamma_{\theta^{\c G}}$ is the set of all elements $\gamma\in \Gamma$ 
such there exists $\sigma\in\c G$ with $\theta^{\gamma\sigma}=\theta$, or equivalently, 
such there exists $\sigma\in \c G$ with $\theta^\gamma =\theta^\sigma$; 
and similarly for $\c G_{\theta^\Gamma}$. It is clear that the stabilizer $\Gamma_\theta$ 
is normal in $\Gamma_{\theta^{\c G}}$, the stabilizer $\c G_\theta$ is normal in $\c G_{\theta^\Gamma}$, 
and the condition $(\gamma,\sigma)\in (\Gamma\times \c G)_\theta$ defines an isomorphism 
$\gamma \Gamma_\theta\mapsto \sigma\c G_\theta$ between the groups $\Gamma_{\theta^{\c G}}/\Gamma_\theta$ and $\c G_{\theta^{\Gamma}}/\c G_\theta$ by using the projections from the
 direct product $\Gamma \times \c G$.

\medskip
Now let $G$ be a finite group and $N$ a normal subgroup of $G$. 
Then $G$ acts by conjugation on $N$ via $n\mapsto n^g=g^{-1}ng$ for $n \in N, g \in G$,
and by the the previous paragraph we obtain 
the usual action of $G$ on $\irr N$. The above considerations hold in particular in this case that $\Gamma=G$
and $G$ acts via conjugation on $N$ and $\irr N$.

Let $\theta\in \irr N$. Recall that if $\theta\in\irr N$ is $G$-fixed, i.e., $G=G_\theta$, then 
$(G,N,\theta)$ is called a {\bf character triple}. More generally, if $G=G_{\theta^{\c G}}$, we call 
$(G,N,\theta)$ a {\bf $\c G$-triple} and indicate this by writing $(G,N,\theta)_{\c G}$. 
This is equivalent to $\{ \theta^g \mid g \in G\}  \sbs \{ \theta^\sigma \mid \sigma \in \c G\}$.
Note that $(G_{\theta^{\c G}},N,\theta)_{\c G}$ is always a $\c G$-triple.
We denote by $\irr{G|\theta}$ the set of  $\chi \in \irr G$
such that $\theta$ is a constituent of the restriction $\chi_N$, and by $\irr{G|\theta^{\c G}}$ the set
of irreducible characters $\chi\in\irr G$ such that $\chi\in\irr{G|\theta^\sigma}$ for some $\sigma\in\c G$.

\medskip
We refer the reader to Chapter 11 of \cite{Is} or  Section 10.4 of \cite{N} for background on 
complex projective representations $\c P$ of a finite group $G$, i.e., functions $\c P\colon G\to \GL_n(\C)$
for some $n$ such that $\c P(x)\c P(y)=\alpha(x,y)\c P(xy)$ for $x,y\in G$ and some $\alpha(x,y)\in \C^\times$.
The function $\alpha\colon G\times G\to \C^\times$ is determined by $\c P$ and called the {\bf factor set}
associated with $\c P$. It satisfies the $2$-cocycle relation $\alpha(x,y)\alpha(xy,z)=\alpha(x,yz)\alpha(y,z)$ for all $x,y,z$ in $G$, 
and it is therefore, by definition, an element of the abelian group $Z^2(G,\C^\times)$ of {\em $2$-cocycles of $G$ with values in $\C^\times$}. 
For any function $\mu\colon G\to \C^\times$ the function $\mu\c P$ which maps $g$ to
$\mu(g)\c P(g)$ is again a projective representation. 
Its factor set is $\delta(\mu)\cdot\alpha$, where $\delta(\mu)(x,y)=\mu(x)\mu(y)\mu(xy)^{-1}$ for $x,y\in G$. 
The functions of the form $\delta(\mu)$ form the subgroup $B^2(G,\C^\times)\le Z^2(G,\C^\times)$ of {\em $2$-coboundaries of $G$ 
with values in $\C^\times$}. In fact,   all these can be defined if instead of $\C^\times$ we use any multiplicative subgroup $U$ of $\C^\times$, and we will need to do so.
\medskip

Recall that two complex (projective) representations $\c P$ and $\c P'$ of $G$ are called {\bf similar} if
there exists an invertible complex matrix $M$ such that $\c P(g) =M^{-1} \c P'(g)M$ for every $g \in G$. 
In this case we write $\c P\sim \c P'$ and observe that $\c P$ and $\c P'$ have the same factor set. 
Moreover, if $\c P$ and $\c P'$ are similar and if their entries and also the values of their factor set 
belong to a subfield $K$ of $\C$, then the matrix $M$ from above can be chosen to have entries in $K$ as well. 
This follows from the 
Noether-Deuring Theorem (see for instance \cite[Lemma~1.14.8]{Linckelmann1})
applied to the corresponding twisted group algebra of $G$ over $K$.

For any character triple $(G,N,\theta)$, there exists a projective representation $\c P$ of $G$
{\bf associated} with $\theta$ in the sense of Definition~5.2 in \cite{N}. That is, $\c P_N$ is a representation of $N$
affording the character $\theta$, $\c P(gn)=\c P(g)\c P(n)$, and $\c P(ng)=\c P(n)\c P(g)$ for all $g\in G$ and $n\in N$.
Note that, since $\c P_N$ is a representation, one has $\alpha(1,1)=1$. Also, $\alpha$ is constant on $N\times N$-cosets
of $G\times G$. Thus, we may interpret $\alpha$ also as element in $Z^2(G/N,\C^\times)$, and we do so frequently by abuse of notation.

\begin{thm}\label{projectivecyclotomic}
If $(G,N,\theta)$ is a character triple,
then there exists
a projective representation
$\c P$ of $G$ associated with $\theta$
which has entries in $\QQ^{\mathrm{ab}}$ and whose
factor set takes values in the group of $|G/N|$-th roots of unity.
\end{thm}

\begin{proof} 
Let $U$ denote the group of roots of unity in $\Q^{\mathrm{ab}}$. For any abelian group $A$ and positive integer $n$ let $A_n$ denote the $n$-torsion subgroup of $A$. By Corollary~1.2 in \cite{NSV} there exists a projective representation $\c P$ of $G$ associated with $\theta$ 
which has entries in $\Q^{\mathrm{ab}}$ and whose factor set $\alpha$ belongs to $Z^2(G/N,U)$. 
We follow an idea from the proof of \cite[Theorem~11.15]{Is}.  
For $x,y,z\in G/N$ we have $\alpha(x,yz)\alpha(y,z) = \alpha(x,y)\alpha(xy,z)$. 
Fixing $x$ and $y$ and taking the product over all $z\in G/N$ one obtains
$\mu(x)\mu(y)=\alpha(x,y)^{|G/N|}\mu(xy)$, where $\mu\colon G/N\to U$ is defined by $a\mapsto\prod_{b\in G/N} \alpha(a,b)$.

Next set $A:=\langle B^2(G/N,U), \alpha\rangle\le Z^2(G/N,U)$. 
The above equation relating $\mu$ and $\alpha$ shows that $|A:B^2(G/N,U)|$ is finite and divides $|G/N|$. 
Since the group $U$ is divisible, also $B^2(G/N,U)$ is divisible. 
Therefore, Lemma~11.14 in \cite{Is} implies that $B^2(G/N,U)$ has a complement $X$ in $A$. 
Since $A/B^2(G/N,U)$ has order dividing $|G/N|$, we have $X\le A_{|G/N|}\le Z^2(G/N,U)_{|G/N|}=Z^2(G/N,U_{|G/N|})$. 
Thus, $\alpha\in A\le B^2(G/N,U) \cdot Z^2(G/N,U_{|G/N|})$ and there exists a function 
$\nu\colon G/N\to U$ and $\beta\in Z^2(G/N,U_{|G/N|})$ such that $\alpha=\delta(\nu)\cdot \beta$. 
For $x\in G/N$ set $\lambda(x):=\nu(x)^{-1}\beta(1,1)^{-1}\in U_{|G/N|}$.
Then $\lambda\c P$ is a projective representation of $G$ associated with $\theta$ and with values in $\Q^{\mathrm{ab}}$. 
Its factor set equals $\alpha\cdot\delta(\lambda)=\beta(1,1)^{-1}\beta\in Z^2(G/N,U_{|G/N|})$.
\end{proof}

Suppose that $\c P$ is a projective representation of $G$ which has entries in $\Q^{\mathrm{ab}}$.
Then, for any $\sigma\in\c G$, the function $g\mapsto \c P(g)^\sigma$ defines again a projective
representation $\c P^\sigma$ of $G$. If $\alpha$ denotes the factor set of $\c P$ then the factor
set of $\c P^\sigma$ is given by $\alpha^\sigma$ which is defined by $\alpha^\sigma(x,y)=
\alpha(x,y)^\sigma$ for $x,y\in G$.

\begin{lem}\label{mu}
Suppose that $N \nor G$, $\theta \in \irr N$, and assume that
$\theta^{g\sigma }=\theta$ for some $g\in G$ and
$\sigma \in {\rm Gal}(\QQ^{\mathrm{ab}}/\QQ)$.
Let $\c P$ be a projective representation of $G_\theta$
associated with $\theta$ with entries in $\QQ^{\rm ab}$ and factor set $\alpha\in Z^2(G_\theta,\C^\times)$.
Then 
$$\c P^{g\sigma}(x)=\c P(gxg^{-1})^\sigma\, ,$$ 
where $\sigma$ is applied entry-wise, defines again a projective representation of $G_\theta$
associated with $\theta$ with factor set $\alpha^{g\sigma}$ given by 
$\alpha^{g\sigma}(x, y)=\alpha^g(x,y)^\sigma=\alpha(gxg^{-1},gyg^{-1})^\sigma$ for $x,y \in G_\theta$.
In particular, there is a unique function 
$$\mu_{g\sigma} \colon G_\theta \rightarrow \CC^\times$$
with $\mu_{g\sigma}(1)=1$ and constant on cosets of $N$
such that $\c P^{g\sigma} \sim \mu_{g\sigma} \c P$.
\end{lem}
   
\begin{proof}  
See Lemma~1.4 in \cite{NSV}.
\end{proof}

We are now ready to define a partial order relation between $\c G$-triples, which is exactly
the same relation given in \cite{NSV} except that the subgroup $\c H$ is now replaced by the more general
group $\c G$.

\begin{defi}\label{Hiso}
Suppose that $(G,N,\theta)_{\c G}$ and $(H,M,\varphi)_{\c G}$
are $\c G$-triples.
We write 
$$(G,N,\theta)_{\c G} \relc (H,M,\varphi)_{\c G}$$ 
if all the following conditions hold:
\begin{enumerate}[(i)]
\item[(i)]
$G=NH$, $N\cap H=M$, and $\cent {G}N \sbs H$. 
  
\smallskip
\item[(ii)]
$(H \times \c G)_\theta=(H \times \c G)_\varphi$. In particular, $H_\theta=H_\varphi$.

\smallskip  
\item[(iii)]
There are projective representations $\c P$ and $\c P'$ associated with 
$(G_\theta,N,\theta)$ and $(H_\varphi,M,\varphi)$
with entries in $\QQ^{\rm ab}$ and factor sets $\alpha$ and $\alpha'$, respectively, such that 
$\alpha$ and $\alpha'$ take roots of unity as values,
$\alpha_{H_\theta \times H_\theta}=\alpha'$, 
and, for each $c \in \cent GN$, the scalar matrices $\c P(c)$ and $\c P'(c)$ are 
associated with the same scalar $\zeta_c$.

\smallskip
\item[(iv)] For every $a \in (H\times \c G)_\theta$, the functions $\mu_a$
and $\mu'_a$  given by Lemma \ref{mu} agree on $H_\theta$.
\end{enumerate}
\end{defi}

In (iii), notice that if $c \in \cent GN$, then $c \in H_\theta$,
and $\c P(c)$ and $\c P'(c)$ are scalar matrices by Schur's Lemma
(applied to the irreducible representations $\c P_N$ and $\c P'_M$).

\begin{lem}\label{key1}
Let $N \nor G$, $H \le G$ such that $G=NH$ and set $M:=N\cap H$.
Suppose that $(G_{\theta^{\c G}}, N, \theta)_{\c G} \relc  (H_{\varphi^{\c G}}, M, \varphi)_{\c G}$.
Then there exists a $\c G$-equivariant bijection 
$^*:\irr{G|\theta^{\c G}} \rightarrow \irr{H|\varphi^{\c G}}$
which preserves ratios of characters in the sense that $\chi(1)/\theta(1)=\chi^*(1)/\varphi(1)$.
In particular, $F(\chi)=F(\chi^*)$ for all $\chi \in \irr{G|\theta^{\c G}}$.
Also, $[\chi_{\zent N}, \chi^*_{\zent N}] \ne 0$.
\end{lem}
 
\begin{proof}
From the hypotheses, we easily deduce that $G_{\theta^\c G} \cap H=H_{\varphi^\c G}$.
See Corollary 1.11 of \cite{NSV}.  For the last part use, for instance, Lemma 3.3 of \cite{NS}.
\end{proof}

\begin{lem}\label{canext}
Suppose that $(G,N,\theta)_{\c G}$ and $(H,M,\varphi)_{\c G}$
are $\c G$-triples such that $G=NH$, $N\cap H=M$, $\cent {G}N \sbs H$, 
and $(H \times \c G)_{ \theta}=    (H \times \c G)_{ \varphi}$.
Suppose further that $\theta$ has an extension $\hat\theta \in \irr{G_\theta}$
and $\varphi$ has an extension $\hat\varphi \in \irr{H_\varphi}$
such that $[\hat\theta_{\cent GN}, \hat\varphi_{\cent GN}] \ne 0$.
Finally suppose that, for every $(h,\sigma) \in (H \times \c G)_{ \theta}$, the unique linear characters
$\lambda$ and $\lambda'$ of $G_\theta$ and $H_\varphi$ with $\hat\theta^{h\sigma}=\lambda\hat\theta$ and
$\hat\varphi^{h\sigma}=\lambda'\hat\varphi$, respectively, satisfy $\lambda_{H_\varphi} = \lambda'$.
Then 
$$(G,N,\theta)_{\c G} \relc (H,M,\varphi)_{\c G} \, .$$
 \end{lem}
 
 \begin{proof}
 First note that the condition in the scalar product makes sense, since $\cent GN\le H_\varphi=H_\theta\le G_\theta$.
 
 The first and second condition in Definition~\ref{Hiso} follow from the assumptions of the lemma.
 
 Let $\c P$ and $\c P^\prime$ be representations of $G_\theta$ and $H_\varphi$ over $\Q^{\rm ab}$ affording $\hat\theta$ and
 $\hat\varphi$, respectively. Since $\hat\theta_N$ is irreducible,
 we have $\hat\theta_{\cent GN}=\theta(1)\gamma$, for some linear $\gamma \in \irr{\cent GN}$.
 By hypothesis, $\hat\varphi_{\cent GN}=\varphi(1)\gamma$. This shows that the third condition in Definition~\ref{Hiso} is satisfied.
 
To check the fourth condition in Definition~\ref{Hiso},
let $a \in (H\times \c G)_\theta$. Thus, $\theta^a=\theta$ and $\varphi^a=\varphi$, 
since $(H \times \c G)_{ \theta}=    (H \times \c G)_{ \varphi}$.
As $\c P^a$ is a representation of $G_\theta$ whose character is an extension of $\theta$ to $G_\theta$, it
is similar to the representation $\lambda \c P$ for a unique linear character $\lambda \in \irr{G_\theta/N}$. This implies that $\mu_a=\lambda$. Similarly, $\mu'_a=\lambda'$. But by hypothesis, $\lambda_{H_\varphi}=\lambda'$, so that $\mu_a$ and $\mu'_a$ coincide on $H_\varphi=H_\theta$ as desired. 
 \end{proof}

\begin{thm}\label{mult1}
Suppose that $(G,N,\theta)_{\c G}$ and $(H,M,\varphi)_{\c G}$
are $\c G$-triples with $H\le G$ such that $G=NH$, $N\cap H=M$ and  $\cent {G}N \sbs H$.
Suppose further that $(H \times \c G)_\theta=(H \times \c G)_\varphi$.
Assume that  $[\theta_M,\varphi]=1$, or that $G_\theta/N$ is a $p$-group and $[\theta_M,\varphi]$
is coprime to $p$. 
Then   
$$(G,N,\theta)_{\c G} \relc (H,M,\varphi)_{\c G} \, .$$
\end{thm}
 
\begin{proof}
(a) Suppose first that $[\theta_M,\varphi]=1$. By Theorem~\ref{projectivecyclotomic} there exists a projective representation $\c P$ of $G_\theta$ associated with $\theta$
with entries in $\Q^{\rm ab}$ and factor set $\alpha$ taking values in the group $Z$ of $|G_\theta/N|$-th roots of unity in $\Q^{\rm ab}$.

Let $\hat G_\theta:=G_\theta\times Z$ be the group with multiplication 
$(g,z)(g',z'):=(gg', \alpha(g,g')zz')$, cf.~Theorem~5.6 in \cite{N}. 
Then $N=N\times 1\nor \hat G_\theta$ and $Z=1\times Z\le Z(\hat G_\theta)$. 
Moreover, $\hat H_\theta:=H_\theta\times Z\le \hat G_\theta$ is a subgroup with $N\hat H_\theta=\hat G_\theta$,
since $(n,1)(h,z)=(nh, \alpha(n,h)z)=(nh,z)$.
Also, identifying $M$ with $M\times 1$, we have $N\cap \hat H_\theta=M \nor \hat H_\theta$.

Set $\hat{\c P}(g,z)=z\c P(g)$ for $(g,z)\in\hat G_\theta$. This defines an ordinary irreducible
representation of $\hat G_\theta$ whose character $\chi$ extends $\theta$. 
Since $[\chi_M,\varphi]=[\theta_M, \varphi]=1$,
there exists an extension $\xi\in\irr{\hat H_\theta}$ of $\varphi$ and a character 
$\delta$ of $\hat H_\theta$ (possibly $0$) with $[\delta_M,\varphi]=0$ 
such that $\chi_{\hat H_\theta}=\xi+\delta$.
Replacing $\hat{\c P}$ with a $\Q^{\rm ab}$-similar representation,
we may assume that 
\begin{equation}\label{eqn block matrix a}
  {\hat{\c P}}_{\hat H_\theta} = \left( \begin{array}{cc} {\hat{\c P}^\prime} & 0  \\ 0 & {\c Q} \\ \end{array} \right) \, ,
\end{equation}
where $\hat{\c P}^\prime$ is a representation of $\hat H_\theta$ that affords $\xi$,
and $\c Q$ is a representation of $\hat H_\theta$ that affords $\delta$. 
For $h\in H_\theta$ set $\c P'(h):=\hat{\c P}'(h,1)$.
Equation~(\ref{eqn block matrix a}) implies  
$\hat{\c P}^\prime(h,z)=z\hat{\c P}^\prime (h,1)=z\c P'(h)$ and ${\c Q}(h,z)=z{\c Q}(h,1)$ for $h \in H_\theta$ and $z \in Z$. 
Therefore, evaluating Equation~(\ref{eqn block matrix a}) at elements of the form $(h,1)$ with $h\in H_\theta$ implies that $\c P'$ is a
a projective representation associated to $\varphi$ with factor set $\alpha_{H_\theta \times H_\theta}$. And evaluating 
Equation~(\ref{eqn block matrix a}) at elements of the form $(c,1)$ with $c\in \cent GN$ implies that $\c P(c)$ and $\c P'(c)$ involve the same scalar.

Suppose that $a \in (H \times \c G)_\theta$. 
Then there exists functions $\mu_a\colon G_\theta/N\to \C^\times$ and $\mu'_a\colon H_\theta/M\to \C^\times$  
and invertible matrices $S$ and $T$ such that 
\begin{equation*}
   S^{-1}\c P^a(g) S=\mu_a(g) \c P(g) \quad\text{and}\quad T^{-1} (\c P')^a(h)T = \mu'_a(h) \c P'(h)
\end{equation*}
for all $g\in G_\theta$ and $h\in H_\theta$. Here we view $\mu_a$ and $\mu'_a$ as functions on $G_\theta$ and $H_\theta$, respectively.
Recall that we need to show that $\mu_a$ and $\mu'_a$ agree on $H_\theta$.
The above equations together with Equation~(\ref{eqn block matrix a}) imply
\begin{equation}\label{eqn block matrix a1}
   \mu_a(h)\begin{pmatrix} \c P'(h) & 0 \\ 0 & \c Q(h) \end{pmatrix} = S^{-1}\c P^a(h) S = 
   S^{-1} \begin{pmatrix} \mu'_a(h)T \c P'(h) T^{-1} & 0 \\ 0 & \c Q^a(h)\end{pmatrix} S
\end{equation}
for all $h\in H_\theta$. Since $\delta^a_M=\chi^a_M-\xi^a_M = \theta^a_M-\varphi^a=\theta_M-\varphi=\chi_M-\xi_M=\delta_M$, 
there also exists an invertible matrix $U$ such that $U^{-1} \c Q^a(m) U = \c Q(m)$ for all $m\in M$. 
Thus, Equation~(\ref{eqn block matrix a1}) for $h\in M$ implies that the invertible matrix 
$S^{-1}\begin{pmatrix}T&0\\0&U\end{pmatrix}$ is an intertwining matrix for the representation 
$\begin{pmatrix} \c P'_M & 0 \\ 0 & \c Q_M \end{pmatrix}$ of $M$. Since $[\varphi,\delta_M]=0$, this matrix must be block diagonal.
This implies that also $S=\begin{pmatrix}A&0\\0&B\end{pmatrix}$ is block diagonal with invertible matrices $A$ and $B$.
Now Equation~(\ref{eqn block matrix a1}) implies $\mu_a(h)\c P'(h) = \mu'_a(h) A^{-1}T \c P'(h) T^{-1}A$ and taking traces implies
\begin{equation}\label{eqn traces a}
   \mu_a(h)\tr(\c P'(h)) = \mu'_a(h) \tr(\c P'(h))
\end{equation}
for all $h\in H_\theta$. Since $\xi\in\irr{\hat H_\theta}$ restricts irreducibly to $\varphi\in\irr M$, \cite[Theorem~1.25]{N} implies that every coset of $M$ in $\hat H_\theta$ contains an element on which $\xi$ does not vanish. Therefore, for every $h\in H_\theta$ there exists $m\in M$ such that $0\neq \xi(hm,1)=\tr(\c P'(hm))$. Since $\mu_a$ and $\mu'_a$ are constant on cosets of $M$ in $H_\theta$, Equation~(\ref{eqn traces a}) implies $\mu_a(h)=\mu'_a(h)$ for all $h\in H_\theta$ as desired.

(b) Suppose now that $G_\theta/N$ is a $p$-group and that $n:=[\theta_M,\varphi]$ is not divisible by $p$. 
By Theorem~\ref{projectivecyclotomic} there exists a projective representation $\c P'$ of $H_\theta=H_\varphi$ associated with $\varphi$
with entries in $\Q^{\rm ab}$ and factor set $\alpha'$ taking values in the group $Z$ of $|G_\theta/N|$-th roots of unity in $\Q^{\rm ab}$.
In particular, $Z$ is a $p$-group. Let $\alpha\in Z^2(G_\theta/N, Z)$ be the image of $\alpha'$ under the isomorphism 
$H_\theta/M\cong G_\theta/N$, $hM\mapsto hN$, and form the group $\hat G_\theta$ from $\alpha$ as in Part~(a). Then $\alpha_{H_\theta\times H_\theta}=\alpha'$. Again we have $N:=N\times 1\nor \hat G_\theta$, $M:=M\times 1\nor \hat H_\theta$ and $\hat G_\theta= \hat H_\theta N$ with $M=\hat H_\theta\cap N$. Moreover, we will consider the normal subgroups $\hat N:=N\times Z\nor \hat G_\theta$ and $\hat M:=M\times Z\nor \hat H_\theta$ which satisfy $\hat G_\theta= \hat H_\theta \hat N$ and $\hat M = \hat H_\theta\cap \hat N$.

Define $\hat{\c P}'(h,z):= z\c P'(h)$ for $(h,z)\in\hat H_\theta$. Then $\hat{\c P}'$ is an irreducible representation of $\hat H_\theta$ 
whose character $\xi$ extends $\varphi$. 
Note that $\xi_{\hat M}= \varphi\times \iota=:\hat \varphi$, where $\iota\in \irr Z$ is the linear character given by $\iota(z)=z$ for all $z\in Z$. 
We also set $\hat \theta:=\theta \times \iota \in \irr{\hat N}$. 
Then $[\hat \theta_{\hat M}, \hat \varphi] = [\theta_M,\varphi]  = n$ is not divisible by $p$.
By \cite[Theorem~6.9(a)]{N}, there exist an extension $\chi\in\irr{\hat G_\theta}$ of $\hat \theta$,
 a character $\psi$ of $\hat H_\theta/\hat M$ with determinant $1$, 
 and a character $\delta$ of $\hat H_\theta$ (possibly $0$) with $[\delta_{\hat M},\hat\varphi]=0$ 
 such that $\chi_{\hat H_\theta} = \psi \xi +\delta$. Note that $\chi$ lies over $\iota$, since $\xi$ does and that $\psi(1)=n$, since $(\chi_{\hat H_{\theta}})_{\hat M} = \chi_{\hat M}=\hat \theta_{\hat M}$.
 
 Choose a representation $\hat{\c P}$ of $\hat G_\theta$ affording $\chi$ such that
 \begin{equation}\label{eqn block matrix 2}
   {\hat{\c P}}_{\hat H_\theta} = \left( \begin{array}{cc} {\c R \otimes \hat{\c P}^\prime} & 0  \\ 0 & {\c Q} \\ \end{array} \right) \, ,
 \end{equation}
 where $\c R$ is a representation of $\hat H_\theta/\hat M$  affording $\psi$ 
 and $\c Q$ is a representation of $\hat H_\theta$ affording $\delta$. 
 Moreover, set $\c P(g):= \hat{\c P}(g,1)$ for $g\in G_\theta$. 
 Then $\c P$ is a projective representation of $G_\theta$ with factor set $\alpha$ and satisfying $\hat{\c P}(g,z)=z\c P(g)$ 
 for all $(g,z)\in\hat G_\theta$, since $\chi_Z=\chi(1)\iota$.
 We will show that $\c P$ and $\c P'$ satisfy the conditions (iii) and (iv) in Definition~\ref{Hiso}.
 
For $c\in \cent GN$ write $\c P(c)= \epsilon(c) I_{\theta(1)}$ and $\c P'(c)=\epsilon'(c) I_{\varphi(1)}$. 
Note that $\cent GN\times Z$ is a subgroup of $\hat H_\theta$ which is normal in $\hat G_\theta$. 
We have $\chi_{\cent GN\times Z}=\chi(1) \hat\epsilon$ and $\xi_{\cent GN\times Z}=\xi(1)\hat\epsilon'$ 
for the linear characters $\hat\epsilon$ and $\hat\epsilon'$ of $\cent GN \times Z$ defined by 
$\hat\epsilon(c,z)=z\epsilon(c)$ and $\hat\epsilon'(c,z)=z\epsilon'(c)$. 
It suffices to show that $\hat\epsilon=\hat\epsilon'$. 
By Equation~(\ref{eqn block matrix 2}) applied to $(c,z)\in\cent GN\times Z$, 
it suffices to show that $\c R(c,z)=I_{\psi(1)}$, or equivalently, that $\psi_{\cent GN\times Z}=\psi(1) 1_{\cent GN\times Z}$. 
Set $V:=M\cent GN\times Z$ which is a normal subgroup of $\hat H_\theta$. 
It suffices to show that $\psi_V=\psi(1)1_V$, which we will prove now.  
Write $\psi_V=\rho_1+\cdots+\rho_t$ with $\rho_i\in \irr{V/\hat M}$. 
We have $\chi_V= (\chi_{\hat H_\theta})_V=\sum_{i=1}^t \rho_i\xi_V +\delta_V$. 
Since $\chi_{\cent GN\times Z}=\chi(1) \hat\epsilon$, the characters $\rho_i\xi_V$ lie over $\hat \epsilon$, for all $i=1,\ldots,t$. 
Moreover, by Gallagher's Theorem, each $\rho_i \xi_V$ is irreducible, since $\xi_V$ extends $\hat \varphi\in \irr{\hat M}$. 
By \cite[Theorem~6.8(d)]{N}, restriction from $V$ to $\hat M$ defines a bijection
\begin{equation*}
   \irr{V|\hat \epsilon}\longrightarrow \irr{\hat M | \hat\epsilon_{\hat M\cap (\cent GN\times Z)}}\,.
\end{equation*}
Thus, $(\rho_i \xi_V)_{\hat M}=\rho_i(1)\xi_{\hat M} = \rho_i(1)\hat\varphi$ is irreducible, so that $\rho_i$ is a linear character for all $i=1,\ldots,t$. Moreover, Gallagher's Theorem, together with the injectivity of the above map implies that $\rho_i=\rho_1$ for all $i$. Thus, $t=n$ and, $\psi_V=n\rho_1$.
Taking determinants and using that $\rho_1$ is a linear character of a $p$-group, we conclude that 
$\rho_1$ is trivial and that $\hat \epsilon=\hat \epsilon'$ as claimed.

Next let $a\in(H\times \c G)_\theta=(H\times \c G)_\varphi$ and let $\mu_a\colon G_\theta/N\to \C^\times$ and $\mu'_a\colon H_\theta/M\to\C^\times$ be the unique functions with $\c P^a\sim \mu_a\c P$ and $(\c P')^a \sim \mu'_a \c P'$, respectively. Further, let $S$ and $T$ be invertible matrices such that $S^{-1} \c P^a(g) S = \mu_a(g) \c P(g)$ for all $g\in G_\theta$ and $T^{-1} \c (P')^a(h) T = \mu'_a(h) \c P'(h)$ for all $h\in H_\theta$. Using the first of these equations, we obtain $\mu_a(g) \mu_a(g') \alpha(g,g') = \mu_a(gg') \alpha^a(g,g')$ for all $g,g'\in G_\theta$.
And fixing $g$ and taking the product of these equations over elements $g'$ from a transversal of $N$ in $G_\theta$, we obtain that $\mu_a(g)^{|G_\theta/N|}\in Z$. Thus $\mu_a(g)$ is a root of unity of $p$-power order. Similarly one shows that $\mu'_a(h)$ has $p$-power order for all $h\in H_\theta$.

Using equation~(\ref{eqn block matrix 2}) we obtain in the same way as in Part~(a) that
\begin{equation}\label{eqn block matrix 3}
   \mu_a(h) \begin{pmatrix} \c R(h)\otimes \c P'(h) & 0 \\ 0 & \c Q(h)\end{pmatrix} =
   S^{-1} \begin{pmatrix} \c R^a(h) \otimes \mu'_a(h) T\c P'(h) T^{-1} & 0 \\ 0 & \c Q^a(h)\end{pmatrix} S\,,
\end{equation}
for all $h\in H_\theta$. With the same argument as in Part~(a), there also exist an invertible matrix $U$ 
such that $U^{-1} \c Q^a(m) U= \c Q(m)$ for all $m\in M$. 
Therefore, Equation~(\ref{eqn block matrix 3}) for $h\in M$ implies that $S^{-1} \begin{pmatrix}I_n\otimes T & 0\\ 0 & U\end{pmatrix}$ 
is an intertwining matrix for the representation $\begin{pmatrix} I_n\otimes \c P'_M & 0\\ 0 & \c Q_M \end{pmatrix}$ of $M$.
Since the character $n\varphi$ of $I_n\otimes \c P'_M$ has no common constituent 
with the character $\delta_M$ of $\c Q_M$,  this intertwining matrix must be block diagonal. 
This implies that also $S=\begin{pmatrix}A& 0\\ 0 & B\end{pmatrix}$ must be block diagonal with invertible matrices $A$ and $B$. Therefore, $A^{-1}(I_n\otimes T)$ is an intertwining matrix for the representation $I_n\otimes \c P'_M$. Since $\c P'_M$ is irreducible, Schur's Lemma implies that $A^{-1}(I_n\otimes T)$ must be of the form $C\otimes I_{\varphi(1)}$ for some invertible matrix $C$. Thus $A=C^{-1}\otimes T$ and Equation~(\ref{eqn block matrix 3}) implies $\c R(h)\otimes \mu_a(h) \c P'(h)) = C^{-1} \c R^a(h) C\otimes \mu'_a(h)\c P'(h)$ for all $h\in H_\theta$. Applying the linear map $\id\otimes \tr$ on both sides, we obtain $\mu_a(h)\tr(\c P'(h)) \cdot \c R(h) = \mu'_a(h) \tr(\c P'(h))\cdot C^{-1}\c R^a(h) C$ and taking determinants yields $\mu_a(h)^n\tr(\c P'(h))^n = \mu'_a(h)^n \tr(\c P'(h))^n$ for all $h\in H_\theta$, since $\det{\psi}=1$. As in Part (a) we can find for every $h\in H_\theta$ an element $m\in M$ such that $\tr(\c P'(hm))\neq 0$. Since $\mu_a$ and $\mu'_a$ are constant on cosets of $M$, we obtain $\mu_a(h)^n=\mu'_a(h)^n$ for all $h\in H_\theta$. Since $\mu_a(h)$ and $\mu'_a(h)$ have $p$-power orders and $n=[\theta_M,\varphi]$ is not divisible by $p$, we obtain $\mu_a(h)=\mu'_a(h)$ for all $h\in H_\theta$ as desired. This completes the proof of the theorem.
\end{proof}

\medskip
The remainder of this section recalls results from \cite{NSV} on how the relation 
$$(G,N,\theta)_{\c G} \relc  (H,M,\varphi)_{\c G}$$ between two $\c G$-triples induces
the same relation between naturally associated $\c G$-triples. These results will be
essential for the proofs in Section~4. We point out that the results in \cite{NSV}
were stated for a particular subgroup $\c H$ of $\mathrm{Gal}(\Q^{\mathrm{ab}}/Q)$, but the 
proofs hold for any subgroup $\c G$.

\smallskip
The following lemma transitions to subgroups. 
It is immediate from the definition. 
 
\begin{lem}\label{goingsubgroup}
Let $(G,N,\theta)_{\c G} \relc  (H,M,\varphi)_{\c G}$ and let $N\le J\le G$. Then
$$(J,N,\theta)_{\c G} \relc  (H\cap J,M,\varphi)_{\c G}\,.$$
\end{lem}

The next lemma describes the situation when the relation $\relc $ descends to quotient groups.

\begin{lem}\label{goingquotient}
Let $(G,N,\theta)_{\c G} \relc  (H,M,\varphi)_{\c G}$ and suppose that $L\nor G$ is contained in 
$\ker\theta\cap\ker\varphi\cap \cent G N$,
and $\cent{G/L}{N/L}=\cent GN/L$. Then
$$(G/L,N/L,\theta)_{\c G} \relc  (H/L,M/L,\varphi)_{\c G}\, ,$$
where $\theta$ and $\varphi$ are considered in $\irr{N/L}$ and $\irr{M/L}$, respectively.
\end{lem}
 
\begin{proof}
See Lemma~2.4 in \cite{NSV}.
\end{proof}

Recall that the wreath product $G\wr {\sf S}_n$ of $G$ with the symmetric group ${\sf S}_n$ 
is defined as the semidirect product $G^n\rtimes {\sf S}_n$, 
where ${\sf S}_n$ acts on $G^n$ via $(g_1,\ldots,g_n)^\pi:=(g_{\pi(1)},\ldots, g_{\pi(n)})$. 
We always view ${\sf S}_n$ as a group with multiplication $\pi\rho:=\rho\circ \pi$.
 Thus, 
 $$((g_1,\ldots,g_n),\pi)\cdot((g'_1,\ldots, g'_n),\pi') = ((g_1 g'_{\pi(1)},\ldots,g_n g'_{\pi(n)}), \pi\pi')\, .$$
 The next lemma shows that the relation $\relc $ passes to wreath products.

 \begin{lem}\label{goingwreath}
 Let $(G,N,\theta)_{\c G} \relc  (H,M,\varphi)_{\c G}$ and let $n$ be a positive integer. Then
 $$((G\wr {\sf S}_n)_{\tilde\theta^{\c G}},N^n,\tilde\theta)_{\c G} \relc  ((H\wr {\sf S}_n)_{\tilde\varphi^{\c G}},M^n, \tilde\varphi)_{\c G}\, ,$$
 where $\tilde\theta:=\theta^n\in \irr{N^n}$ and $\tilde\varphi:=\varphi^n\in\irr{M^n}$.
 \end{lem}
 
 \begin{proof}
 This is the special case of Lemma~2.7 in \cite{NSV} with $k=1$.
 \end{proof}

The next lemma indicates how the relation $\relc $ passes to central products in a special case. 
We refer the reader to Chapter 10 (Section 3) of \cite{N} for the definition 
of central products of groups and the description of characters of central products,
which we will use freely.

\smallskip
Suppose that $K$ is the product of two subgroups $N$ and $Z$ with $N\nor K$ and $Z \leq \cent K N$.
Then $K$ is the central product of $N$ and $Z$. 
In this case 
$$\irr K =\dot{ \bigcup_{\nu \in \irr{Z\cap N}}} \irr{K|\nu}\,.$$
Moreover, for $\nu\in\irr{Z\cap N}$, one has a bijection
$$\irr{N|\nu}\times\irr{Z|\nu} \rightarrow \irr{K|\nu}\,,\quad (\theta,\lambda)\mapsto \theta\cdot\lambda\,,$$
with $(\theta\cdot\lambda)(nz)=\theta(n)\lambda(z)$ for $n\in N$ and $z\in Z$.
Note that, whenever a group $A$ acts via automorphisms on $K$,
stabilizing $N$ and $Z$, the following holds: if $a \in A$ and 
$\theta\cdot \lambda \in \irr K$, then 
$(\theta \cdot \lambda)^a=\theta\cdot \lambda$ if, and only if,
$\theta^ a=\theta$ and $\lambda^a=\lambda$. The same happens if $A \le \c G={\rm Gal}(\Q^{\rm ab}/\Q)$.

\begin{lem}\label{dotproduct}  Let $(G, N, \theta)_{\c G}$ and 
$(H, M, \varphi)_{\c G}$ be $\c G$-triples such that 
$$(G, N, \theta)_{\c G}\relc  (H, M, \varphi)_{\c G}\, .$$
Suppose that $Z\nor G$ is abelian and satisfies  $Z\sbs \cent G N$.
Let $\nu \in \irr{Z\cap N}$ be under $\theta$ and $\lambda \in \irr {Z|\nu}$. Then
$$(G_{(\theta\cdot \lambda)^\c G}, NZ, \theta\cdot \lambda)_{\c G}\relc  (H_{(\varphi\cdot \lambda)^\c G}, MZ, \varphi\cdot \lambda)_{\c G}.$$
\end{lem}

\begin{proof}
See Theorem 2.8 of \cite{NSV}
 \end{proof}

We will also need the {\em butterfly} theorem.

\begin{thm}\label{butterfly}
Let $(G, N, \theta)_{\c G}$ and $(H, M, \varphi)_{\c G}$ be 
$\c G$-triples such that $(G, N, \theta)_{\c G}\relc  (H, M, \varphi)_{\c G}$. 
Let $\epsilon \colon G \rightarrow {\rm Aut}(N)$ be the group homomorphism 
defined by conjugation with elements of $G$. Suppose that $N\nor \hat G$ and 
$\hat \epsilon(\hat G)=\epsilon(G)$, where $\hat \epsilon \colon \hat G \rightarrow {\rm Aut}(N)$ 
is given by conjugation by $\hat G$. Let $M \cent{\hat G} N\sbs \hat H\leq \hat G$ 
be such that $\hat \epsilon(\hat H)=\epsilon(H)$. Then 
$$(\hat G, N, \theta)_{\c G}\relc  (\hat H, M, \varphi)_{\c G}.$$
\end{thm}
 
\begin{proof}   
See Theorem 2.9 of \cite{NSV}.
\end{proof}

 To finish this section we prove a lemma which will be useful later.

 \begin{lem}\label{useful}
 Suppose that $G=NH$, $M=N\cap H$, where $N\nor G$ and $\cent GN \sbs H\le G$.
 Let $\theta \in \irr N$ and $\varphi \in \irr M$ lie over the same $\lambda \in \irr{\zent N}$.
 Suppose that $G_\theta=N\cent GN$ and that $(H \times \c G)_\theta=(H \times \c G)_\varphi$. 
Then 
$$(G_{\theta^{\c G}},N,\theta)_{\c G} \relc  (H_{\varphi^\c G},M, \varphi)_{\c G} \, .$$
\end{lem}

\begin{proof}
Write $L=\cent GN$ and $Z=\zent N$, so that $Z=N \cap L$. Note that $Z\sbs \zent{G_\theta}$, since $G_\theta=NL$.
We define a projective representation $\hat\lambda$ of $L$ associated with the character triple $(L,Z,\lambda)$ in the following way:
let $T$ be a fixed transversal of $Z$ in $L$ containing $1$. Define $\hat\lambda(l)=\lambda(z)$ for $l \in L$, where $l=zt$
for unique $z \in Z$ and $t \in T$. 
It is straightforward to verify that $\hat\lambda$ is a projective representation of $L$ associated with $(L,Z,\lambda)$ and with factor set 
$$\alpha(l_1,l_2)=\hat\lambda(l_1l_2)^{-1} \lambda(z_1z_2) \, ,$$
where $l_i=z_it_i$.

Now, let $\c Y$ be a representation of $N$ affording $\theta$ with entries in $\Q^{\rm ab}$.
We define 
$$\c P(nl)=\c Y(n) \hat\lambda(l) \, $$
for $n \in N$ and $l\in L$. We show that that $\c P$ is a well-defined projective representation of $N$ associated to $\theta$.
Suppose that $nl=n_1l_1$. Then $n_1^{-1}n=l_1l^{-1}=z \in N \cap L$, so that
$\c Y(n)\hat\lambda(l)=\c Y(n_1z)\hat\lambda(z^{-1}l_1)=\c Y(n_1) \lambda(z)\lambda(z^{-1}) \hat\lambda(l_1)$
and $\c P$ is well defined. Clearly, $\c P(g)$ has entries in $\Q^{\rm ab}$ for all $g\in G_\theta$. 
Since $N$ and $L$ centralize each other, one obtains
$$\c P(g_1) \c P(g_2)=\alpha(l_1, l_2) \c P(g_1g_2) \, ,$$
where $g_i=n_il_i$ with $n_i\in N$ and $l_i\in L$. 
It is now straightforward to verify that the projective representation $\c P$ is associated with $(G_\theta,N,\theta)$. 
Note that its factor set is given by $(n_1l_1,n_2l_2)\mapsto \alpha(l_1,l_2)$ for $n_i\in N$ and $l_i\in L$.

Similarly, let $\c Z$ be a representation of $M$ affording $\varphi$ with entries in $\Q^{\rm ab}$. Note that
$H_\varphi=H_\theta=G_\theta\cap H = NL\cap H = (N\cap H)L = ML$ and $M\cap L=Z$.
We define 
$$\c P^\prime (ml)=\c Z(m) \hat\lambda(l) \, $$
for $m \in M$ and $l\in L$. 
As for $\c P$, it is straightforward to verify that $\c P'$ is a well-defined projective representation of $H_\varphi$, 
which has entries in $\Q^{\rm ab}$, is associated with $(H_\varphi,M,\varphi)$, 
and whose factor set is given by $(m_1l_1,m_2l_2)\mapsto \alpha(l_1,l_2)$ for $m_i\in M$ and $l_i\in L$.

Parts (i) and (ii) in Definition~\ref{Hiso} of $\relc $ follow by hypotheses. 
It is now also clear that Part~(iii) of this definition is satisfied with the scalar $\zeta_l$ associated to $l\in L$ equal to $\hat\lambda(l)$. 
Finally, we verify Part~(iv) of Definition~\ref{Hiso}. Let $a=(g,\sigma)\in (G\times \c G)_\theta$. 
Since $\c Y^a(n)=\c Y(gng^{-1})^\sigma$ also affords the character $\theta$, there exists $S\in\GL_{\theta(1)}(\Q^{\rm ab})$ with
$S\c Y^a(n) S^{-1}=\c Y(n)$ for all $n\in N$. With this, for $n\in N$ and $l\in L$, we obtain
$$ \c P^a(nl)=\c Y(gng^{-1})^\sigma \hat\lambda(glg^{-1})^\sigma = S\c Y(n)S^{-1}\hat\lambda^a(l) = 
S \c P(nl) \hat\lambda(l)^{-1}\hat\lambda^a(l)S^{-1}\,,$$
so that $\mu_a(nl)= \lambda(l)^{-1}\hat\lambda^a(l)$. 
Similarly, one shows that for $a\in (H\times \c G)_\varphi$, $m\in M$ and $l\in L$, one has $\mu'_a(ml)=\lambda(l)^{-1}\hat\lambda^a(l)$. 
Now it is immediate that Part~(iv) in Definition~\ref{Hiso} is satisfied, and the proof is complete.
\end{proof}


\section{The inductive Feit condition}\label{thedefinition}

From now on we assume $\c G=\mathrm{Gal}(\Q^{\mathrm{ab}}/\Q)$.

\begin{defi}\label{inductive}Let $S$ be a finite non-abelian simple group
and let $X$ be a universal covering group 
of $S$. We say that $S$ satisfies the {\bf inductive Feit condition}
if for every $\chi \in \irr X$, there is a pair $(U,\mu)$, where
$U<X$ and $\mu \in \irr U$,   satisfying the following conditions:
\begin{enumerate}

\item
$U=\norm XU$ and $U$ is {\bf intravariant} in $X$, i.e., 
if $w \in {\rm Aut}(X)$, then $U^w=U^x$ for some $x \in X$.

\item
$(\Gamma \times \c G)_\chi=(\Gamma \times \c G)_\mu$,
where $\Gamma={\rm Aut}(X)_U$.

\item
$(X \rtimes \Gamma_{\chi^{\c G}}, X, \chi)_{\c G} \relc 
 (U \rtimes \Gamma_{\chi^{\c G}}, U, \mu )_{\c G}$.
\end{enumerate}

In this case, we also say that $(\chi,\mu)$, or simply $\chi$, satisfies the {\bf inductive Feit condition}.
\end{defi}

Note that condition (i) implies $\zent X\le U$, since $\norm XU=U$, and that
condition (ii) implies that $\Gamma_\chi = \Gamma_\mu$ and $\Gamma_{\chi^{\c G}}=\Gamma_{\mu^{\c G}}$. 
Note also that in condition (iii) one has
$(X \rtimes \Gamma)_{\chi^{\c G}}= X\rtimes (\Gamma_{\chi^{\c G}})$ and 
$(U \rtimes \Gamma)_{\chi^{\c G}}= U\rtimes (\Gamma_{\chi^{\c G}}) = 
U\rtimes (\Gamma_{\mu^{\c G}}) = (U\rtimes \Gamma)_{\mu^{\c G}}$, so that 
$(X \rtimes \Gamma_{\chi^{\c G}}, X, \chi)_{\c G}$ and $(U \rtimes \Gamma_{\chi^{\c G}}, U, \mu )_{\c G}$ are
$\c G$-triples.

\medskip
We give a few situations in which the inductive Feit condition is automatically satisfied.

\begin{thm}\label{easysituations}
Let $S$ be a non-abelian finite simple group, let $X$ be a universal covering group of $S$, and
let $\chi \in \irr X$. Further, let $U<X$ be self-normalizing and intravariant in $X$. Let $\Gamma=\aut X_U$ and assume that $\mu \in \irr U$ is
such that  $(\Gamma \times \c G)_\chi=(\Gamma \times \c G)_\mu$. 
\begin{enumerate}[\rm(a)]
\item
If $[\chi_U, \mu]=1$, then $(\chi,\mu)$ satisfies the inductive Feit condition.
\item
If $\aut X_\chi={\rm Inn}(X)$ and if $\chi$ and $\mu$ lie over the same irreducible character of $\zent X$, 
then $(\chi, \mu)$ satisfies the inductive Feit condition.
\item
Suppose $\Delta$ is a finite group acting on $X$ via automorphisms, fixing $U$ and stabilizing the orbit $\chi^{\c G}$ in $\irr X$.
Set $\tilde X:=X\rtimes \Delta$ and $\tilde U:=U\rtimes \Delta$, and
suppose that $\chi$ has an extension $\tilde\chi \in \irr{\tilde X_\chi}$
and $\mu$ has an extension $\tilde\mu \in \irr{\tilde U_\mu}$ such that 
$[\tilde\chi_{\cent {{\tilde U}_\mu}{X}}, \tilde\mu_{\cent {{\tilde U}_\mu}{X}}]\ne 0$.
Assume further that, for each $(\delta,\sigma) \in (\Delta \times \c G)_\chi$, the unique linear character 
$\lambda \in \irr{\Delta_\chi}$  with $\tilde\chi^{\delta\sigma}=\lambda\tilde\chi$ 
also satisfies $\tilde\mu^{\delta\sigma}=\lambda\tilde\mu$, after identifying $\tilde X_\chi/X=\Delta_\chi=\Delta_\mu=\tilde U_\mu/U$. 
Finally, assume that $\tilde \epsilon(\tilde X)=\hat\epsilon(X \rtimes \Gamma_{\chi^\c G})$,   where
$\tilde\epsilon\colon \tilde X \rightarrow \aut X$ and $\hat\epsilon: X \rtimes \Gamma_{\chi^\c G} \rightarrow \aut X$ 
are defined through the conjugation actions of $\tilde X$ and $X \rtimes \Gamma_{\chi^\c G}$ on their common normal subgroup $X$. 
Then $(\chi,\mu)$ satisfies the inductive Feit condition.
\item
Suppose that $\Sigma$ is a finite group containing $X$ as a normal subgroup such that the conjugation action of 
$\Delta:= \NB_\Sigma(U)$ on $X$ induces the subgroup $\Gamma_{\chi^{\c G}}$. 
Suppose that $\chi$ has an extension $\chi^\sharp \in \irr{\Sigma_\chi}$
and $\mu$ has an extension $\mu^\sharp \in \irr{\Delta_\mu}$ such that  
$[\chi^\sharp_{\cent {\Delta_\mu}{X}}, \mu^\sharp_{\cent {\Delta_\mu}{X}}]\ne 0$.
Assume further that, for each $(\delta,\sigma) \in (\Delta \times \c G)_\chi$, the unique linear character 
$\gamma \in \irr{\Sigma_\chi}$  with $(\chi^\sharp)^{\delta\sigma}=\gamma\chi^\sharp$ 
also satisfies $(\mu^\sharp)^{\delta\sigma}=\gamma\mu^\sharp$, after identifying $\Sigma_\chi/X=\Delta_\mu/U$. 
Then $(\chi,\mu)$ satisfies the inductive Feit condition.
\item Suppose that $\Sigma$ is a finite group containing $X$ as a normal subgroup such that the conjugation action of 
$\Delta:= \NB_\Sigma(U)$ on $X$ induces the subgroup $\Gamma_{\chi^{\c G}}$. 
Suppose that $\chi$ has a $\Sigma_\chi$-invariant real-valued extension $\chi^\ast \in \irr{\OB^{2'}(\Sigma_\chi)}$
and $\mu$ has a $\Delta_\mu$-invariant real-valued extension $\mu^\ast$ to the full inverse image of 
$\OB^{2'}(\Delta_\mu/U)$ in $\Delta_\mu$ 
such that  
$$[\chi^\ast_{\cent {\OB^{2'}(\Delta_\mu)}{X}}, \mu^\ast_{\cent {\OB^{2'}(\Delta_\mu)}{X}}]\ne 0.$$
Assume further that, for each $(\delta,\sigma) \in (\Delta \times \c G)_\chi$, the unique linear character 
$\gamma \in \irr{\OB^{2'}(\Sigma_\chi)}$  with $(\chi^*)^{\delta\sigma}=\gamma\chi^\ast$ 
also satisfies $(\mu^\ast)^{\delta\sigma}=\gamma\mu^\ast$, after identifying $\Sigma_\chi/X=\Delta_\mu/U$. 
Then $(\chi,\mu)$ satisfies the inductive Feit condition.
\item
Suppose that $\Sigma$ is a finite group containing $X$ as a normal subgroup such that the conjugation action of 
$\Delta:= \NB_\Sigma(U)$ on $X$ induces the subgroup $\Gamma_{\chi^{\c G}}$. 
Assume in addition that both $\chi$ and $\mu$ are real-valued, 
$[\chi_{\ZB(X)}, \mu_{\ZB(X)}]\ne 0$, and $|\Delta_\mu/U|$ is odd.
Then $(\chi,\mu)$ satisfies the inductive Feit condition.
\item
Suppose that $\Aut(X)_{\chi^{\c G}}/\Inn(X)$ is a $p$-group and that $[\chi_U,\mu]$ is not divisible by $p$. 
Then $(\chi,\mu)$ satisfies the inductive Feit condition.
\end{enumerate}
\end{thm}
 
\begin{proof}
Write $\hat X=X \rtimes \Gamma_{\chi^\c G}$
and $\hat U=U \rtimes \Gamma_{\chi^\c G}$. 

Let $c\colon X\to \Aut(X)$ denote the homomorphism sending $x$ to $-^x$, i.e., conjugation with $x$. 
We first establish that in all Parts~(a)--(g) the following hold:

\begin{enumerate}[\rm(i)]
\item $\hat X= X\hat U$ and $X\cap \hat U = U$.

\item $\cent{\hat X}{X}=\{(u,c(u)^{-1})\mid u \in U\}$. In particular, $\cent{\hat X}{X} \sbs \hat U$.

\item $(\hat U\times \c G)_\chi = (\hat U \times \c G)_\mu$.
\end{enumerate}

\smallskip
(i) is obvious from the definitions of $\hat X$ and $\hat U$. 
To see (ii), let $(x,\alpha)\in \cent{\hat X}{X}$. Then $y^{x\alpha}=y$ for all $y\in X$. 
Thus, $\alpha=c(x)^{-1}$ and, since $\alpha\in\Gamma$, we have $U=U^{\alpha^{-1}}=U^x$, so that $x\in \norm X U=U$. 
This shows one inclusion.
Conversely, if $u\in U$, then clearly $(u,c(u)^{-1})\in\hat X$ centralizes $X$. This shows (ii).
To see that (iii) holds, note that 
$(\hat U\times \c G)_\chi = (U\Gamma_{\chi^{\c G}}\times \c G)_\chi = (U\times 1)(\Gamma\times \c G)_\chi$, 
since $U$ fixes $\chi$ and since $\Gamma_{\chi^{\c G}}$ is the image of the first projection of $(\Gamma\times \c G)_\chi$. 
Similarly, using $\Gamma_{\chi^{\c G}}=\Gamma_{\mu^{\c G}}$, we have 
$(\hat U\times \c G)_\mu = (U\Gamma_{\chi^{\c G}}\times \c G)_\mu= (U\Gamma_{\mu^{\c G}} \times \c G)_\mu = 
(U\times 1)(\Gamma\times \c G)_\mu$. 
Since $(\Gamma \times \c G)_\chi=(\Gamma \times \c G)_\mu$ by hypothesis, we obtain
$(\hat U\times \c G)_\chi=(\hat U\times \c G)_\mu$, so that (iii) holds.

\smallskip
Note that in all Parts~(a)--(g) we only need to show that condition~(iii) in Definition~\ref{inductive}, namely
\begin{equation}\label{eqn hat c}
(\hat X, X,\chi)_{\c G} \relc  (\hat U, U, \mu)_{\c G}
\end{equation}
is satisfied, since $\chi$ and $(U,\mu)$ satisfy conditions (i) and (ii) in Definition~\ref{inductive} by hypothesis. 

\smallskip
(a) We apply Theorem~\ref{mult1} with $(\hat X, X, \chi)$ and $(\hat U, U, \mu)$ 
in place of $(G,N,\theta)$ and $(H,M,\varphi)$, respectively. Clearly, these are $\c G$-triples.
Note that all the hypotheses of Theorem~\ref{mult1} are either hypotheses in this lemma or are covered by (i), (ii), or (iii) above.
Thus, Theorem~\ref{mult1} implies (\ref{eqn hat c}).

\smallskip
(b) We apply Lemma~\ref{useful} to the $\c G$-triples $(\hat X, X,\chi)_{\c G}$ and $(\hat U, U,\mu)_{\c G}$.
At the beginning of the proof we already established that $\hat X= X\hat U$, $U=X\cap \hat U$, $\cent{\hat X}{X}\sbs \hat U$, and
$(\hat U\times \c G)_\chi = (\hat U \times \c G)_\mu$. 
Since, by hypothesis, $\chi$ and $\mu$ lie over the same irreducible character of $\zent X$,
the only remaining hypotheses of Lemma~\ref{useful} that needs to be checked is that $\hat X_\chi= X\cdot\cent{\hat X}{X}$.

But clearly $X\cent{\hat X}{X}\sbs \hat X_\chi$. 
Conversely, note that $\hat X_\chi=(X\rtimes\Gamma_{\chi^{\c G}})_\chi =X\cdot\Gamma_\chi$. 
But $\Aut(X)_\chi=\Inn(X)$ implies that $\Gamma_\chi= c(U)$. 
In fact, if $\alpha\in\Gamma_\chi$ then $\alpha = c(x)$ for some $x\in X$ and we obtain $U=U^\alpha=U^x$ so that $x\in \norm X U = U$.
Altogether we obtain $\hat X_\chi= X\cdot c(U) \le X \cdot \cent{\hat X}{X}$ by property (ii) at the beginning of the proof.
Now Lemma~\ref{useful} implies (\ref{eqn hat c}) and the proof of Part~(c) is complete.

\smallskip
(c) We first apply Lemma~\ref{canext} in order to show that $(\tilde X, X,\chi)_{\c G}\relc (\tilde U, U, \mu)_{\c G}$.
Let $\pi\colon \Delta\to\Aut(X)$ be the permutation representation of the action of $\Delta$ on $X$. 
Then $\pi(\Delta)\sbs \Gamma_{\chi^{\c G}}$. It follows that $(\tilde X, X,\chi)_{\c G}$ is a $\c G$-triple.
Moreover, since  $(\Gamma \times \c G)_\chi=(\Gamma \times \c G)_\mu$, we have $\Gamma_{\chi^{\c G}}=\Gamma_{\mu^{\c G}}$, 
so that also $(\tilde U, U, \mu)_{\c G}$ is a $\c G$-triple, since $\pi(\Delta)\sbs \Gamma_{\mu^{\c G}}$.
Some of the hypotheses of Lemma~\ref{canext} are already satisfied by the hypotheses in this Lemma.
The remaining ones that need to be verified are 
\begin{enumerate}
\item [\rm (i')] $X\tilde U = \tilde X$, 
\item [\rm (ii')] $X\cap \tilde U = U$, 
\item [\rm (iii')] $\cent{\tilde X}{X}\sbs \tilde U$, and 
\item [\rm (iv')] $(\tilde U \times \c G)_\chi = (\tilde U\times \c G)_\mu$.
\end{enumerate}

Again, (i') and (ii') follow from the definitions of $\tilde X$ and $\tilde U$. 
To verify (iii'), let $(x,\delta)\in\cent{\tilde X}{X}$. Then $y^{x\pi(\delta)}=y$ for all $y\in X$. 
Since $\pi(\delta)\in\Gamma$ fixes $U$, we obtain $U=U^{\pi(\delta)^{-1}}=U^x$, so that $x\in N_X(U) = U$ and $(x,\delta)\in \tilde U$.
This shows (iii').
Finally, since $U$ fixes $\chi$ and $\mu$, we have 
$(\tilde U\times \c G)_\chi = (U\times 1) (\Delta\times \c G)_\chi = 
(U\times 1)(\pi\times \id)^{-1}((\Gamma\times \c G)_\chi)$ and 
$(\tilde U\times \c G)_\mu = (U\times 1)(\Delta\times \c G)_\mu =
(U\times 1) (\pi\times \id)^{-1}((\Gamma\times \c G)_\mu)$, 
which shows (iv').
Note also that $(\Delta\times \c G)_\chi = (\Delta\times \c G)_\mu$ implies $\Delta_\chi=\Delta_\mu$ as claimed in the lemma.
Lemma~\ref{canext} now implies $(\tilde X, X,\chi)_{\c G}\relc (\tilde U, U, \mu)_{\c G}$.

Next we use the butterfly Theorem~\ref{butterfly} to show that 
$$(\hat X, X, \chi)_{\c G} \relc  (\hat U, U,\mu)_{\c G}$$
also holds.
The only remaining hypotheses of Theorem~\ref{butterfly} we need to verify are 
\begin{enumerate}
\item [($\alpha$)] $U\cent{\hat X}{X}\le \hat U$, and 
\item [($\beta$)] $\hat \epsilon (\hat U) = \tilde \epsilon (\tilde U)$. 
\end{enumerate}
But ($\alpha$) was already established in (ii') at the beginning of the proof.
And to see ($\beta$), note that $c(X) \cdot \Gamma_{\chi^{\c G}}  = \hat \epsilon(\hat X) = \tilde \epsilon (\tilde X) = c(X) \cdot \pi(\Delta)$, so that
$\Gamma_{\chi^{\c G}} \le c(X) \cdot \pi(\Delta)$. This implies $\Gamma_{\chi^{\c G}} \le c(U)\cdot \pi(\Delta)$. 
In fact, if $\alpha\in\Gamma_{\chi^{\c G}}$ then $\alpha=c(x)\pi(\beta)$ for some $x\in X$ and $\beta\in\Delta$. 
Since $\alpha,\pi(\beta)\in\Gamma=\Aut(X)_U$, we have $U= U^{\alpha\pi(\beta^{-1})} = U^x$, 
so that $x\in \norm X U = U$ and $\alpha\in c(U)\cdot \pi(\Delta)$ as claimed. 
But $\Gamma_{\chi^{\c G}} \le c(U)\cdot \pi(\Delta)$ implies
$\hat\epsilon(\hat U) = c(U)\cdot\Gamma_{\chi^{\c G}} = c(U)\cdot\pi(\Delta) = \tilde\epsilon(\tilde U)$ and ($\beta$) is proved.
This establishes (c).

\smallskip
(d) The assumptions on $U$, and the characters $\chi$ and $\mu$ imply that $\Sigma = X\Delta$ and $\Delta_\chi=\Delta_\mu$. 
Let 
$$\Phi: \Sigma_\chi \to \GL_n(\C), ~\Theta: \Delta_\mu \to \GL_m(\C)$$ 
be representations (over $\Q^{\mathrm {ab}}$) affording $\chi$, respectively $\mu$, and let $\Phi^\sharp$, $\Theta^\sharp$ be extensions of
them that afford $\chi^\sharp$, respectively $\mu^\sharp$. Now we form the semidirect product 
$$\tilde X = X \rtimes \Delta = \{(x,v) \mid x \in X, v \in \Delta\}$$ 
and define
$$\tilde\Phi: \tilde X_\chi \to \GL_n(\C)$$ 
via $\tilde\Phi((x,v)) = \Phi(x)\Phi^\sharp(v)$ for $x \in X$ and $v \in \Delta_\chi$. Since $\Phi^\sharp$ extends $\Phi$, $\tilde \Phi$ is a
representation of $\tilde X_\chi$ with character $\tilde \chi$ which extends $\chi$. Similarly, we form the semidirect product 
$$\tilde U = U \rtimes \Delta = \{(u,v) \mid u \in U, v \in \Delta\}$$  
and define
$$\tilde\Theta: \tilde U_\mu \to \GL_m(\C)$$ 
via $\tilde\Theta((u,v)) = \Theta(u)\Theta^\sharp(v)$ for $u \in U$ and $v \in \Delta_\mu$. Since $\Theta^\sharp$ extends $\Theta$, $\tilde \Theta$ is a
representation of $\tilde U_\mu$ with character $\tilde \mu$ which extends $\mu$. 

Note that $\CB_{{\tilde U}_\mu}(X) = \{(u,v) \mid u \in U,~v \in \Delta_\mu,~uv \in \CB_{\Sigma}(X)\}$. For any $(u,v) \in \CB_{{\tilde U}_\mu}(X)$,
$z:=uv$  fixes $\mu$ and $U$,  hence belongs to $\Delta_\mu=\Delta_\chi$, and thus $z \in \CB_{\Delta_\mu}(X)$. By Schur's lemma, there is a 
linear character $\nu$ of $\CB_{\Delta_\mu}(X)$ such that $\Phi^\sharp(z) = \nu(z)I_n$. By assumption and by Schur's lemma, 
$\Theta^\sharp(z) = \nu(z)I_m$. Now we have 
$$\tilde\Phi((u,v)) = \Phi(u)\Phi^\sharp(v) = \Phi^\sharp(z) = \nu(z)I_n,$$
and 
$$\tilde\Theta((u,v)) = \Theta(u)\Theta^\sharp(v) = \Theta^\sharp(z) = \nu(z)I_m.$$
It follows that 
$$[\tilde\chi_{\cent {{\tilde U}_\mu}{X}}, \tilde\mu_{\cent {{\tilde U}_\mu}{X}}]\ne 0.$$

Next, for each $(\delta,\sigma) \in (\Delta \times \c G)_\chi$, let $\lambda$ be the unique linear character 
of $\Delta_\chi$ such that $\tilde\chi^{\delta\sigma}=\lambda\tilde\chi$, and let $\gamma$ be the unique linear character 
of $\Delta_\chi$ such that $(\chi^\sharp)^{\delta\sigma}=\gamma\chi^\sharp$. Viewing $\lambda$ as a representation (of degree $1$)
of $\tilde X_\chi$, we see that the representation $\lambda\tilde\Phi$ is equivalent to $\tilde\Phi^{\delta\sigma}$; let the conjugation by 
$M \in \GL_n(\C)$ sends the former to the latter. In particular, this conjugation sends $\lambda(x)\tilde\Phi(x) = \Phi(x)=\lambda(x)\Phi^\sharp(x)$ to 
$\tilde\Phi^{\delta\sigma}(x) = (\Phi^{\sharp})^{\delta\sigma}(x)$ for all $x \in X$, and 
$\lambda(v)\tilde\Phi(v) = \lambda(v)\Phi^\sharp(v)$ to $(\Phi^\sharp)^{\delta\sigma}(v)$ for all $v \in \Delta_\chi$. 
Since $\Sigma_\chi=X\Delta_\chi$, it follows that 
$\lambda\Phi^\sharp$ is equivalent to $(\Phi^\sharp)^{\delta\sigma}$. Taking traces and using uniqueness, we have $\gamma=\lambda$.
By the assumption on $\gamma$, $(\mu^\sharp)^{\delta\sigma}=\lambda\mu^\sharp$. Repeating the preceding argument for $\Theta^\sharp$ and 
$\tilde\Theta$, we obtain that $\tilde\mu^{\delta\sigma}=\lambda\tilde\mu$.

We have verified all the hypotheses in (c), and hence (d) follows.

\smallskip
(e) Since $\chi^*$ is real-valued and $\Sigma_\chi$-invariant, it has a unique real extension $\chi^\sharp$ to $\Sigma_\chi$ by \cite[Lemma 2.1]{NT}. Similarly,
$\mu^*$ has a unique real extension $\mu^\sharp$ to $\Delta_\mu$. Let $Y$ denote $\CB_{\Delta_\mu}(X)$. By assumption, the restrictions of $\chi^\sharp$ and 
$\mu^\sharp$ to $Y \cap \OB^{2'}(\Delta_\mu)$ are both multiples of a single linear character. Of course, every Sylow $2$-subgroups of $Y$ 
are contained in $\OB^{2'}(\Delta_\mu)$, and hence $\OB^{2'}(Y) \leq Y \cap \OB^{2'}(\Delta_\mu)$.  Working with the abelian images of $Y$ in corresponding representations, we see that $\chi^\sharp$ and $\mu^\sharp$ agree on the $2$-part of it, and are trivial on the $2'$-part because they are real. It follows 
that the restrictions of $\chi^\sharp$ and $\mu^\sharp$ to $Y$ are both multiples of a single linear character.

Since $\Delta$ induces $\Gamma_{\chi^{\c G}}$, $\Delta_\chi=\Delta_\mu$ is normal in $\Delta$, and hence $\Delta$ normalizes $\OB^{2'}(\Delta_\mu)$
and $\OB^{2'}(\Sigma_\chi)$. Consider any $(\delta,\sigma) \in (\Delta \times \c G)_\chi$, and let $\lambda$ be the unique linear character of 
$\Sigma_\chi/X$ such that $(\chi^\sharp)^{\delta\sigma} = \lambda\chi^\sharp$. Note that the action of $\sigma$ on characters of $\Sigma_\chi$ is induced by 
the action of a Galois automorphism of some cyclotomic field and hence commutes with the complex conjugation, which obviously also commutes with 
the conjugation action by $\delta$. It follows that both $\chi^\sharp$ and $(\chi^\sharp)^{\delta\sigma}$ are real extensions
of $\chi$, whence $\lambda=\bar\lambda$. Similarly, the unique linear character $\nu$ of $\Delta_\mu/U$ such that $(\mu^\sharp)^{\delta\sigma} = \nu\mu^\sharp$
is also real-valued. By assumption, $\lambda$ and $\mu$ agree on $\OB^{2'}(\Delta_\mu/U)$ (after identifying $\Sigma_\chi/X=\Delta_\mu/U$); denote this 
restriction by $\lambda_0$. Now $\lambda$ and $\nu$ are real extensions of $\lambda_0$ to $\Delta_\mu/U$ which has odd index over 
$\OB^{2'}(\Delta_\mu/U)$. Again using \cite[Lemma 2.1]{NT}, we obtain $\lambda=\nu$. Thus $(\mu^\sharp)^{\delta\sigma} = \lambda\mu^\sharp$,
and we are done by (d).

\smallskip
(f) This is a special case of (e), with $\OB^{2'}(\Sigma_\chi)=X$, $\chi^\ast=\chi$, $\OB^{2'}(\Delta_\mu)=U$, and $\mu^\ast=\mu$.

\smallskip
(g) Since $U$ is self-normalizing in $X$, one has $\Inn(X)\cap \Gamma_{\chi^{\c G}} = c(U)$, 
and since $U$ is intravariant in $X$, one has $\Inn(X)\cdot \Gamma_{\chi^{\c G}} = \aut{X}_{\chi^{\c G}}$.
Thus, the inclusion $\Gamma_{\chi^{\c G}}\le \aut{X}_{\chi^{\c G}}$ induces an isomorphism 
$\Gamma_{\chi^{\c G}}/c(U) \cong \aut{X}_{\chi^{\c G}}/\Inn(X)$.
Let $P$ be a Sylow $p$-subgroup of $\Gamma_{\chi^{\c G}}$. Then $P c(U) = \Gamma_{\chi^{\c G}}$, since $\aut{X}_{\chi^{\c G}}/\Inn(X)$ is a $p$-group. 

We claim that $(X\rtimes P, X,\chi)_{\c G}\relc (U\rtimes P, U,\mu)_{\c G}$. 
It suffices to show that the hypotheses of Theorem~\ref{mult1} are satisfied. 
First, both triples are $\c G$-triples, since both $X$ and $P$ stabilize $\chi^{\c G}$ and since both $U$ and 
$P\le \Gamma_{\chi^{\c G}}= \Gamma_{\mu^{\c G}}$
stabilize $\mu^{\c G}$. 
Clearly, $X\rtimes P= X\cdot(U\rtimes P)$ and $X\cap (U\rtimes P)=U$. 
Moreover, $C_{X\rtimes P}(X)\le U\rtimes P$, since if $(x,\alpha)\in X\rtimes P$ centralizes $X$ 
then $c(x)=\alpha^{-1}\in \Gamma$ which implies $x\in \norm XU=U$. 
Using (iii) from the beginning of the proof and intersecting with $(U\rtimes P)\times {\c G}$, 
we also obtain $((U\rtimes P)\times \c G)_\chi=((U\rtimes P)\times \c G)_\mu$. 
Finally $(X\rtimes P)_\chi=X\rtimes(P_\chi)$, so that $(X\rtimes P)_\chi/X\cong P_\chi$ is a $p$-group. 
Together with $[\chi_U,\mu]\not\equiv 0\mod p$, this establishes the hypotheses of Theorem~\ref{mult1} and proves the claim.

Next we apply the butterfly Theorem~\ref{butterfly}, to $(X\rtimes P, X,\chi)_{\c G}\relc (U\rtimes P, U,\mu)_{\c G}$ 
and the $\c G$-triples $(\hat X, X,\chi)$ and $(\hat U, U,\mu)$ and check all the hypotheses. 
Considering the homomorphisms $\hat \epsilon \colon \hat X\to \aut{X}$ and $\epsilon\colon X\rtimes P\to \aut{X}$, 
we have $\hat\epsilon(\hat X) = \Inn(X)\cdot\Gamma_{\chi^{\c G}} = \Inn(X)\cdot c(U)\cdot P = \Inn(X)\cdot P= \epsilon(X\rtimes P)$, 
since $c(U)\le \Inn(X)$. 
Moreover, we have $\hat\epsilon(\hat U) = c(U)\cdot \Gamma_{\chi^{\c G}} = \Gamma_{\chi^{\c G}} = c(U)\cdot P = \epsilon(U\rtimes P)$ 
and $U\cdot C_{\hat{X}}(X)\le \hat U$ by (ii) from the beginning of the proof. 
Therefore, all the hypotheses are of Theorem~\ref{butterfly} are satisfied and we conclude that 
$(\hat X, X, \chi)_{\c G} \relc  (\hat U, U, \mu)_{\c G}$ as desired.
\end{proof}

To help check some of the conditions listed in Theorem \ref{easysituations}, we record the following.

\begin{lem}\label{easy2}
Let $X$ be a finite group, and let $U<X$ be self-normalizing and intravariant in $X$. Let $\chi \in \Irr(X)$ and $\mu \in \irr U$ be
such that  $\Q(\chi)=\Q(\mu)$. Let $\Gamma=\aut X_U$, and assume furthermore that $\Gamma_\chi=\Gamma_\mu$. Then 
$(\Gamma \times \c G)_\chi=(\Gamma \times \c G)_\mu$, if at least one of the following conditions holds.
\begin{enumerate}[\rm(i)]
\item $\chi$ or $\mu$ is rational-valued.
\item $\chi$ or $\mu$ is $\Gamma$-stable.
\item $[\Q(\chi):\Q]=2$, $\chi^{\c G}=\chi^\Gamma$, and $\mu^{\c G}=\mu^\Gamma$. 
\item $[\Q(\chi):\Q]=2$, and $\Gamma$ induces an odd-order subgroup of $\Out(X)$.
\end{enumerate}
\end{lem} 

\begin{proof}
Consider any $(\delta,\sigma) \in \Gamma \times \c G$. 
In the case of (i), $\sigma$ fixes both $\chi$ and $\mu$, hence 
$(\delta,\sigma)$ fixes $\chi$ if and only if $\delta \in \Gamma_\chi = \Gamma_\mu$ if and only if $(\delta,\sigma)$ fixes $\mu$. 
In the case of (ii), $\delta$ fixes both $\chi$ and $\mu$, and hence 
$(\delta,\sigma)$ fixes $\chi$ if and only if $\sigma \in \c G_\chi = \c G_\mu$ if and only if $(\delta,\sigma)$ fixes $\mu$. 
Assume we are in the case of (iv). Then there is an odd integer $m$ such that $\delta^m$ acts as conjugation by some element 
$x \in X$, and $x \in \NB_X(U)=U$. Thus $\delta^m$ fixes both $\chi$ and $\mu$. On the other hand, $\chi^{\c G}$ has size $2$, 
so $\sigma^m$ fixes $\chi$ if and only if $\sigma$ fixes $\chi$. Now if $(\delta,\sigma)$ fixes $\chi$,
then $(\delta^m,\sigma^m)$ fixes $\chi$, so $\sigma^m$ fixes $\chi$, whence $\sigma$ fixes $\chi$ and so $\delta$ fixes $\chi$. By assumption,
this ensures that $\mu$ is fixed by both $\delta$ and $\sigma$. Similarly, if $(\delta,\sigma)$ fixes $\mu$ then it also fixes $\chi$.

In the case of (iii), by the assumption, we can find $\gamma \in \c G$ such that the restriction of 
$\gamma$ to $\Q(\chi)$ generates ${\rm Gal}(\Q(\chi)/\Q)$, and $\chi^\Gamma=\chi^{\c G} = \{\chi,\chi^\gamma\}$, whence
$\Gamma = \langle \Gamma_\chi,g \rangle$ with $\chi^g = \chi^\gamma$ and $g^2 \in \Gamma_\chi$. Similarly, $\mu^\Gamma=\mu^{\c G} = \{\mu,\mu^\gamma\}$ and 
$\mu^g = \mu^\gamma$.
It follows that $(\delta,\sigma)$ fixes $\chi$ if and only if $(\delta,\sigma) \in \Gamma_\chi \times \c G_\chi$ or 
$(\delta,\sigma) \in g\Gamma_\chi \times \gamma \c G_\chi$. Similarly,  $(\delta,\sigma)$ fixes $\mu$ if and only if $(\delta,\sigma) \in \Gamma_\chi \times \c G_\chi$ or 
$(\delta,\sigma) \in g\Gamma_\chi \times \gamma \c G_\chi$, and so we are done.
\end{proof}


\section{Proof of Theorem A}
 
The goal of this section is to prove Theorem~A, see Theorem~\ref{main}. As an easy consequence we will obtain
Corollary~B, see Corollary~\ref{cormain}. Theorem~\ref{main} will be proved by reducing it to situations 
dealt with in Theorem~\ref{key2}, Lemma~\ref{res}, and Theorem~\ref{isaacs}.

\medskip
If a group $\Gamma$ acts on a group $X$ then $\Gamma\wr {\sf S}_n= \Gamma^n\rtimes{\sf S}_n$ 
acts on $X^n$ via
$$(x_1,\ldots,x_n)^{((\gamma_1,\ldots,\gamma_n),\pi)}= 
(x_{\pi^{-1}(1)}^{\gamma_{\pi^{-1}(1)}},\ldots, x_{\pi^{-1}(n)}^{\gamma_{\pi^{-1}(n)}})\, .$$

\begin{lem}\label{wreath}
Let $\Gamma$ be a group acting via automorphisms on a group $X$. 

\smallskip
{\rm (a)} The mapping $\bigl(((x_1,\gamma_1),\ldots,(x_n,\gamma_n)), \pi\bigr)\mapsto
\bigl((x_1,\ldots,x_n),((\gamma_1,\ldots,\gamma_n),\pi)\bigr)$
defines a group isomorphism
$$ (X\rtimes \Gamma) \wr {\sf S}_n \cong X^n\rtimes (\Gamma \wr {\sf S}_n)$$
which preserves the conjugation action on the normal subgroup $X^n$ of both sides. 

\smallskip
{\rm (b)} If $U$ is a subgroup of $X$ and $\Gamma=\aut X_U$ then the map
 \begin{align*}
\Gamma\wr {\sf S}_n & \to \aut{X^n}_{U^n}\,, \\
((f_1,\ldots,f_n),\pi) & \mapsto \Bigl((x_1,\ldots,x_n)\mapsto 
\bigl(f_{\pi^{-1}(1)}(x_{\pi^{-1}(1)}),\ldots,f_{\pi^{-1}(n)}(x_{\pi^{-1}(n)})\bigr) \Bigr)\, ,
\end{align*}
is an injective group homomorphism, preserving the action of both sides on $X^n$.
Moreover, if $X$ is quasisimple then the map is an isomorphism.
\end{lem}
 
\begin{proof}
For a proof of the last statement in Part~(b) see Lemma~10.24 in \cite{N}. All other parts are straightforward verifications.
\end{proof}

\begin{thm}\label{direct}
Suppose that $S$ is simple, that $X$ is a universal covering group of
$S$, and that $\chi\in\irr X$. Suppose further that $U<X$ and $\mu\in\irr U$ are such that
$(\chi,\mu)$ satisfies the inductive Feit condition.
Set $\Gamma={\rm Aut}(X)_U$,
$\tilde \chi=\chi^n \in \irr{X^n}$, $\tilde\mu=\mu^n \in \irr{U^n}$, 
and $\tilde\Gamma=\Gamma \wr {\sf S}_n$. 
Then $U^n=\norm {X^n}{U^n}$,
$(\tilde\Gamma \times \c G)_{\tilde \chi}=(\tilde\Gamma \times \c G)_{\tilde\mu}$,
and
$$(X^n \rtimes \tilde\Gamma_{\tilde\chi^{\c G}}, X^n, \tilde\chi)_{\c G} \relc 
(U^n \rtimes \tilde\Gamma_{\tilde\mu^{\c G}}, U^n, \tilde\mu )_{\c G} \, .$$
\end{thm}
 
\begin{proof}
Since $\norm X U = U$, we also have $\norm{X^n}{U^n}= U^n$ and the first statement holds.
  
An element $((f_1,\ldots,f_n),\pi,\sigma)\in \tilde{\Gamma}\times \c G$
fixes $\tilde\chi$ if and only if $\chi^{f_i\sigma}=\chi$ for all $i=1,\ldots,n$, i.e., if and only if
$(f_i,\sigma)\in (\Gamma\times \c G)_\chi$ for all $i=1,\ldots,n$. 
Similarly, $((f_1,\ldots,f_n),\pi,\sigma)$ fixes $\mu$ if and only if
$(f_i,\sigma)\in (\Gamma\times \c G)_\mu$ for all $i=1,\ldots,n$. 
Since $(\Gamma\times \c G)_\chi = (\Gamma\times \c G)_\mu$ by the inductive Feit condition,
we obtain $(\tilde\Gamma \times \c G)_{\tilde \chi}=(\tilde\Gamma \times \c G)_{\tilde\mu}$ 
and the second statement holds. 
 
In order to prove the last statement set
$$(G,N,\theta)=((X\rtimes\Gamma)_{\chi^{\c G}}, X, \chi)\quad \text{and}\quad 
(H,M,\varphi)=((U\rtimes\Gamma)_{\mu^{\c G}}, U, \mu)\,.$$
By the inductive Feit condition we have $(G,N,\theta)_{\c G} \relc  (H,N,\varphi)_{\c G}$.
Applying Lemma~\ref{goingwreath} we obtain
$$\bigl(((X\rtimes \Gamma)_{\chi^{\c G}}\wr {\sf S}_n)_{\tilde\chi^{\c G}}, X^n,\tilde\chi \bigr)_{\c G} \relc 
\bigl(((U\rtimes \Gamma)_{\mu^{\c G}}\wr {\sf S}_n)_{\tilde\mu^{\c G}}, U^n, \tilde\mu \bigr)_{\c G}\, .$$
Moreover, one has $((X\rtimes \Gamma)_{\chi^{\c G}}\wr {\sf S}_n)_{\tilde\chi^{\c G}} = 
((X\rtimes \Gamma)\wr {\sf S}_n)_{\tilde\chi^{\c G}}$ and 
$((U\rtimes \Gamma)_{\mu^{\c G}}\wr {\sf S}_n)_{\tilde\mu^{\c G}} =
((U\rtimes \Gamma)\wr {\sf S}_n)_{\tilde\mu^{\c G}}$. Using the isomorphism in Lemma~\ref{wreath}(a) we obtain
$$\bigl((X^n\rtimes \tilde\Gamma)_{\tilde\chi^{\c G}}, X^n,\tilde\chi \bigr)_{\c G} \relc 
\bigl((U^n \rtimes \tilde\Gamma)_{\tilde\mu^{\c G}}, U^n, \tilde\mu \bigr)_{\c G}$$
and the proof is complete.
\end{proof}

 The following key theorem deals with a special situation to which the proof of the main theorem
 of this section, Theorem~\ref{main}, will be reduced.

\begin{thm}\label{key2}
Suppose that $K \nor G$ and $K/\zent K\cong S^n$, 
where $S$ is a non-abelian finite simple group which satisfies the inductive Feit condition.   
Moreover, suppose that $\cent G{\zent K}$ acts transitively by conjugation on the $n$ minimal
normal subgroups of $K/\zent K$.
Let $\tau \in \irr{K}$ be $\cent G{\zent K}$-invariant.
 Then there exists a pair $(U,\mu)$, where $\zent K \le U<K$ with $\norm KU=U$,
 $G=K\norm GU$, $\norm G{\norm GU}=\norm GU$, and $\mu \in \irr{U}$, such that
$$(G_{\tau^{\c G}},K, \tau)_{\c G} \relc  (\norm GU_{\mu^{\c G}},U, \mu)_{\c G} \, .$$ 
In particular (see  Lemma~\ref{key1})
there is a $\c G$-equivariant bijection
$$^*:\irr{G|\tau^{\c G} }\rightarrow  \irr{\norm G U | \mu^{\c G}}\,$$
preserving ratios of characters.
Therefore, if $\chi \in \irr{G|\tau^{\c G}}$ and $\xi=\chi^*$, then $\chi(1)/\tau(1)=\xi(1)/\mu(1)$, $\Q(\chi)=\Q(\xi)$,
and $[\chi_{\zent K}, \xi_{\zent K}] \ne 0$.
\end{thm}
 
\begin{proof}
{\sc Step 1:} We first have to establish some results concerning the group theoretic situation.
For ease of notation set $C:=\zent K \nor G$ and $V:=\cent G C \nor G$.
By hypothesis, there exists a minimal normal subgroup $T/C\cong S$ of $K/C$ and 
elements $v_1=1,v_2,\ldots,v_n\in V$ such that $K/C=T_1/C\times\cdots\times T_n/C$,
where $T_i:=T^{v_i}$ for $i=1,\ldots,n$. Then $T_i/C\cong S$ for all $i$.

Next, as it is well-known,  we show that for $i\neq j$ in $\{1,\ldots,n\}$ one has $[T_i,T_j]=1$. In fact, 
$[T_i,T_j]\le C=\zent K$ implies $[T_i,T_j,T_i]=[T_j,T_i,T_i]=1$. The Three-Subgroup
Lemma then implies that $[T'_i,T_j]=1$. But since $T_i/C\cong S$ is non-abelian simple, 
we have $T'_iC=T_i$ so that also $[T_i,T_j]=1$.

Since $K=T_1\cdots T_n$ and $[T_i,T_j]=1$ if $i\neq j$, we have $K'=T'_1\cdots T'_n$. This implies
$K'C=T'_1\cdots T'_nC=T_1\cdots T_n =K$. Now
let $i\in\{1,\ldots,n\}$. Then $C=\zent{T_i}$, since $C$ centralizes $T_i$ and $T_i/C\cong S$
has no nontrivial central subgroup. By \cite[31.1]{A00}, $T_i=T'_i C$ and $T'_i\nor K$ is a 
quasisimple (in particular perfect) group with $T'_i/\zent{T'_i}\cong S$, where 
$\zent{T'_i}=C\cap T'_i= C\cap T' =: D$ is independent of $i$, since $T_i=T^{v_i}$ and $v_i$ 
centralizes $C$.

Next we show that $\zent{K'}=D$. Clearly $D\le \zent{K'}$, since $D\le K'$ and $D\le C=Z(K)$
centralizes $K'$. Conversely, suppose that $z\in\zent{K'}$ and write $z=t_1\cdots t_n$ with
$t_i\in T'_i$ for $i=1,\ldots,n$. Then, for every $i\in\{1,\ldots,n\}$, we have
$[t_i,T'_i] = [z,T'_i] \subseteq [z,K'] = \{1\}$, so that $t_i\in\zent{T'_i}=D$. Thus, $x\in D$.

We now have $\zent{K'}=D\le C\cap T' \le C\cap K' \le \zent{K'}$ which implies equality everywhere.
Next we show that $T' = T\cap K'$. Clearly, $T'\le T\cap K'$. Moreover, 
$T'C\le (T\cap K')C\nor T$ with $T/(T\cap K')C$ is abelian. But $T/C\cong S$ is perfect, so that
$(T\cap K')C=T$. This implies $(T\cap K')/D=(T\cap K')/(C\cap K') \cong (T\cap K')C/C
\cong T/C\cong S$. But from above we have $T'/D\cong S$. Thus $T'\le T\cap K'$ have the same
order so that $T\cap K' = T'$ as desired.
Conjugating with $v_i$ we also obtain $T_i\cap K' =T'_i$ for all $i=1,\ldots, n$.

\smallskip
We now reset the notation and define $K_1:=K'$,  $C_1:=K_1\cap C=D$, and $X:=T'=T\cap K'$, 
a quasisimple group with $\zent{K_1}=\zent X=C_1$. While $K$ might not have been perfect, note that
$K'$ is perfect, since $K'=T'_1\cdots T'_n$ with pairwise commuting and 
perfect normal subgroups $T'_i$. As will become apparent, this is the reason we need to work with $K_1$ first.
Note that $K_1$ is the central product of $X_i:=X^{v_i}$, $i=1,\ldots,n$.
Write $\tau_C= e \nu$, where $\nu\in\irr C$ is linear and 
$e$ is a positive integer, and set $\tau_1:=\tau_{K_1}$. Note that $\tau_1\in\irr{K_1}$, 
since $K=K_1 \zent K$, see Corollary 10.8 in \cite{N}. Finally set $\nu_1:=\nu_{C_1}$, a linear
character of $C_1$. The following diagram depicts the situation we have established.

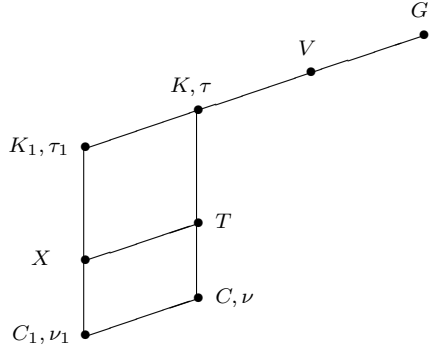
\begin{center}
\unitlength 5mm
{\scriptsize
\begin{picture}(12,10)
\put(0,0){$\bullet$}
\put(0,2){$\bullet$}
\put(0,5){$\bullet$}
\put(3,1){$\bullet$}
\put(3,3){$\bullet$}
\put(3,6){$\bullet$}
\put(6,7){$\bullet$}
\put(9,8){$\bullet$}

\put(0.1,0.1){\line(0,1){5}}
\put(3.1,1.1){\line(0,1){5}}
\put(0.1,0.1){\line(3,1){3}}
\put(0.1,2.1){\line(3,1){3}}
\put(0.1,5.1){\line(3,1){9}}

\put(-1.8,0){$C_1,\nu_1$}
\put(-1.3,2){$X$}
\put(-1.9,5){$K_1,\tau_1$}
\put(3.6,1){$C,\nu$}
\put(3.6,3){$T$}
\put(2.4,6.6){$K,\tau$}
\put(5.8,7.6){$V$}
\put(8.8, 8.6){$G$}
\end{picture}
}
\end{center}

{\sc Step 2:} Consider the homomorphism
$$\pi\colon X^n\to K_1\, , \quad  (x_1,\ldots, x_n)\mapsto x_1^{v_1}\cdots x_n^{v_n}\, ,$$
on the exterior direct product group $X^n$. It is surjective, since $\pi(X^n)=X_1\cdots X_n=K_1$.
Since $K_1$ is the central product of $X_1,\ldots,X_n$, we can write
$$\tau_1=\gamma_1\cdot \ldots \cdot\gamma_n$$
for some $\gamma_i\in\irr{X_i|\nu_1}$ for $i=1,\ldots,n$, see Lemma~10.7 in \cite{N} and 
observe that $\gamma_i|_{C_1}$ is a summand of the character $\tau_1|_{C_1}=e\nu_1$.
Since $X_i=X^{v_i}$, we have $\gamma_i=\rho_i^{v_i}$ for some $\rho_i\in\irr{X|\nu_1}$. Here, $\rho_i$
lies above $\nu_1$, since $V$ centralizes $C_1$. For each $x\in X$ and each $i=1,\ldots,n$, we have
$\bigl(\prod_{j\neq i}\gamma_j(1)\bigr)\rho_i(x)=\bigl(\prod_{j\neq i}\gamma_j(1)\bigr)\cdot \gamma_i(x^{v_i}) = \tau_1(x^{v_i}) = \tau_1(x)$,
since  $\tau$ and also $\tau_1=\tau_{K_1}$ are $V$-stable. Thus, for any two $i, i'$ in $\{1,\ldots,n\}$,
$\rho_i$ is a rational multiple of $\rho_{i'}$. Since they are all irreducible we obtain $\rho_1=\cdots = \rho_n$
which we denote by $\rho$. Thus, the inflation of $\tau_1$ under the epimorphism 
$\pi\colon X^n\to K_1$ is equal to $\tilde\rho:=\rho^n\in\irr{X^n}$.

Let $p\colon Y\to S$ be a universal central extension of $S$. We fix an isomorphism $\phi\colon X/C_1\to S$. 
Since its composition $t\colon X\to X/C_1\to S$ with the natural epimorphism is a central extension and since
$X$ is perfect, there exists a unique homomorphism 
$q\colon Y\to X$ such that $p=t\circ q$. Moreover, $q$ is then a universal central extension of $X$, 
see \cite[33.8]{A00}. We denote the inflation of $\rho\in\irr X$ via $q\colon Y\to X$ again by $\rho\in \irr Y$.
Since $S$ satisfies the inductive Feit condition, there exists $R<Y$ and $\delta\in\irr R$ such that 
$(\rho,\delta)$ satisfies the inductive Feit conditions (i)--(iii) in Definition~\ref{inductive}. 
In particular, $\zent Y\le R$.
Theorem~\ref{direct} then implies that 
\begin{equation}\label{eqn 0}
(Y^n\rtimes \tilde\Gamma_{\tilde\rho^{\c G}}, Y^n, \tilde\rho)_{\c G} \relc 
(R^n\rtimes \tilde\Gamma_{\tilde\delta^{\c G}}, R^n, \tilde\delta)_{\c G}\,. 
\end{equation}
Here, $\Gamma=\aut Y_R$, $\tilde\rho=\rho^n\in\irr{Y^n}$ and $\tilde\delta=\delta^n\in\irr{R^n}$.
By the inductive Feit condition, we have $(\Gamma\times \c G)_\rho=(\Gamma\times \c G)_\delta$,
which in turn implies $(\tilde\Gamma \times \c G)_{\tilde\rho}= (\tilde\Gamma \times \c G)_{\tilde\delta}$.
Applying Lemma~\ref{wreath}(b), (\ref{eqn 0}) becomes
\begin{equation}\label{eqn 1}
(Y^n\rtimes\aut{Y^n}_{R^n,\tilde\rho^{\c G}}, Y^n, \tilde\rho)_{\c G} \relc 
(R^n \rtimes\aut{Y^n}_{R^n,\tilde\delta^{\c G}}, R^n, \tilde\delta)_{\c G}\,.
\end{equation}
Note that $(\tilde\Gamma \times \c G)_{\tilde\rho}= (\tilde\Gamma \times \c G)_{\tilde\delta}$ immediately implies
\begin{equation}\label{eqn stab 1}
(\aut{Y^n}_{R^n}\times \c G)_{\tilde\rho}= (\aut{Y^n}_{R^n}\times \c G)_{\tilde\delta}
\end{equation}
and intersection with $\aut{Y^n}_Z\times \c G$ yields
\begin{equation}\label{eqn stab 2}
(\aut{Y^n}_{R^n,Z}\times \c G)_{\tilde\rho}= (\aut{Y^n}_{R^n,Z}\times \c G)_{\tilde\delta}\,.
\end{equation}

\medskip
{\sc Step 3:} Denote by $Z$ the kernel of the surjective homomorphism
$$\hat\pi:= \pi \circ q^n \colon Y^n\to X^n\to K_1\,,\quad 
(y_1,\ldots,y_n)\mapsto q(y_1)^{v_1}\cdots q(y_n)^{v_n}\,,$$
and denote by $\aut{Y^n}_{R^n,Z,\tilde\rho^{\c G}}$ the set of all automorphisms of $Y^n$ 
which stabilize $R^n$, $Z$, and $\tilde\rho^{\c G}$. Then, by Lemma~\ref{goingsubgroup}, the relation in
(\ref{eqn 1}) implies 
\begin{equation}\label{eqn 2}
(Y^n\rtimes\aut{Y^n}_{R^n,Z,\tilde\rho^{\c G}}, Y^n, \tilde\rho)_{\c G} \relc 
(R^n \rtimes\aut{Y^n}_{R^n,Z,\tilde\delta^{\c G}}, R^n, \tilde\delta)_{\c G}\,.
\end{equation}

\medskip
{\sc Step 4:} The first goal in this step is to show that $\hat\pi\colon Y^n\to K_1$ is a 
universal central extension of $K_1$. To this end consider the surjective homomorphism
$$K_1\to K_1/C_1\cong (X/C_1)^n\cong S^n\,,$$
where the first map is the canonical epimorphism, the second map sends
$(x_1C_1,\ldots, x_nC_1)$ with $x_i\in X_i=X^{v_i}$ to $(x_1^{v_1^{-1}}C_1,\ldots,x_n^{v_n^{-1}}C_1)$,
and the third map is $\phi^n$, where $\phi\colon X/C_1\to S$ is the fixed isomorphism from above.
It is straightforward to check that the composition of $\hat\pi$ and the above map $K_1\to S^n$ is equal
to $p^n$. Therefore, $Z=\ker{\hat\pi}$ is contained in $\ker{p^n}$ which is central in $Y^n$, and $\hat\pi$ is a
central extension of $K_1$.
Since the above map $K_1\to S^n$ has kernel $C_1$, it is a central extension. 
Moreover, since $K_1$ is perfect, by \cite[33.8]{A00}(b), $Y^n$ is a universal covering group of $K_1$.
Since $\hat\pi\colon Y^n\to K_1$ is a central extension of $K_1$ and $Y^n$ is a universal covering group of $K_1$, 
the map $\hat\pi\colon Y^n\to K_1$ is a universal central extension of $K_1$.

Next set $U_1:=\hat\pi(R^n)\le K_1$ and
consider the component-wise homomorphism
$$\tilde\pi\colon Y^n\rtimes\aut{Y^n}_{R^n,Z,\tilde\rho^{\c G}} \to K_1\rtimes \aut{K_1}_{U_1,\tau_1^{\c G}}\,,
\quad(\underline{y},\alpha)\mapsto (\hat\pi(\underline{y}),\bar\alpha)\,,$$
where, for any $\alpha\in\aut{Y^n}_Z$ we denote by $\bar\alpha\in\aut{K_1}$ the unique automorphism 
such that $\hat\pi\circ \alpha= \bar\alpha\circ\hat\pi$. Note that if $\alpha$ additionally stabilizes $R^n$ 
and $\tilde\rho^{\c G}$ then $\bar\alpha$ stabilizes $\hat\pi(R^n)=U_1$ and $\tau_1^{\c G}$, since 
$\tilde\rho$ is the inflation of $\tau_1$ with respect to $\hat\pi\colon Y^n\to K_1$.

For later use we record that $\norm{K_1}{U_1}=U_1<K_1$. In fact this follows immediately from the 
surjective map $\hat\pi\colon Y^n\to K_1$, noting that $Z=\ker{\hat\pi}\le R^n$, $\hat\pi(R^n)=U_1$,
and $\norm{Y^n}{R^n}=R^n<Y^n$.

The next goal is to show that $\tilde\pi$ is surjective. It suffices to show that 
$\aut{Y^n}_{R^n,Z,\tilde\rho^{\c G}} \to \aut{K_1}_{U_1,\tau_1^{\c G}}$, $\alpha\mapsto \bar\alpha$,
is surjective. So let $\beta\in \aut{K_1}_{U_1,\tau_1^{\c G}}$. By Corollary B.8(b) in \cite{N}, there exists
$\alpha\in\aut{Y^n}_Z$ with $\bar\alpha=\beta$. Moreover, since $\hat\pi(R^n)=U_1$ and 
$\ker{\hat\pi}=Z\le \zent{Y^n}\le R^n$,
one has $\alpha(R^n)=\alpha(\hat\pi^{-1}(U_1))=\hat\pi^{-1}(U_1) = R^n$. Also, 
$\tilde\rho^\alpha=\mathrm{inf}_{\hat\pi}(\tau_1)^\alpha = \mathrm{inf}_{\hat\pi}(\tau_1^\beta) =
\mathrm{inf}_{\hat\pi}(\tau_1^\sigma) = \mathrm{inf}_{\hat\pi}(\tau_1)^\sigma= \tilde\rho^\sigma$
for some $\sigma\in \c G$.

Let $L$ denote the kernel of $\tilde\pi$. Our final objective in this step is to show that $L$ satisfies the hypotheses 
in Lemma~\ref{goingquotient}, i.e., (i) $L\le\ker{\tilde\rho}$, (ii) $L\le\ker{\tilde\delta}$,
(iii) $L\le \cent{Y^n\rtimes\aut{Y^n}_{R^n,Z,\tilde\rho^{\c G}}}{Y^n}$, 
and (iv) $\cent{Y^n\rtimes\aut{Y^n}_{R^n,Z,\tilde\rho^{\c G}}/L}{Y^n/L} = 
\cent{Y^n\rtimes\aut{Y^n}_{R^n,Z,\tilde\rho^{\c G}}}{Y^n}/L$. 
First note that, since $\tilde\pi$ is a component-wise map, $\ker{\tilde\pi}$ consists of all pairs
$(\underline{y},\alpha)$ with $\underline{y}\in\ker{\hat\pi}=Z$ and $\bar\alpha=\id_{K_1}$. But by 
Corollary B.8(b) in \cite{N}, this implies $\alpha=\id_{K_1}$. Therefore, 
$L=\{(z,\id_{Y^n})\mid z\in Z\}$. But $Z=\ker{\hat\pi}$ clearly acts trivially on the inflation $\tilde\rho$ of 
$\tau_1$ via $\hat\pi\colon Y^n\to K_1$ which implies (i).
By the relation in (\ref{eqn 2}) and Part~(iii) in Definition~\ref{Hiso} of $\relc $, $\tilde\rho$ and $\tilde\delta$ 
lie above the same linear character of $\zent {Y^n}$. Since $Z=\ker{\hat\pi}\le \zent{Y^n}$, we obtain
$Z\le \ker{\tilde\rho}\cap \zent{Y^n} = \ker{\tilde\delta}\cap \zent{Y^n}\le \ker{\tilde\delta}$. This now implies (ii).
In order to see that (iii) holds, just note that $Z\le\zent{Y^n}$, so that $L$ centralizes $Y^n$.
Next we show (iv). The right hand side is always contained in the left hand side. So let $(\underline y,\alpha)\in
Y^n\rtimes\aut{Y^n}_{R^n,Z,\tilde\rho^{\c G}}$ such that $(\underline y, \alpha)L$ centralizes $Y^n/L$. Then, for all
$\underline y'\in Y^n$ we have $(\underline y')^{(\underline y,\alpha) } = \underline y' l_{\underline y'}$, for some element
$l_{\underline y'}\in L$. One verifies easily that the function $\underline y'\mapsto l_{\underline y'}$ defines a group 
homomorphism $Y^n\to L$. But $L$ is abelian and $Y^n$ is perfect, so that this homomorphism is trivial. 
Thus, $(\underline y, \alpha)$ centralizes $Y^n$ and (iv) is proved.

Let $\mu_1\in\irr{U_1$} be the unique character whose inflation via $\hat\pi\colon Y^n\to K_1$ is equal to $\tilde\delta$.
Since $L$ satisfies the hypotheses of Lemma~\ref{goingquotient}, we obtain, after applying the isomorphism 
$Y^n\rtimes\aut{Y^n}_{R^n,Z,\tilde\rho^{\c G}}/L \cong K_1\rtimes \aut{K_1}_{U_1,\tau_1^{\c G}}$
induced by $\tilde\pi$, that
\begin{equation}\label{eqn 3}
(K_1\rtimes \aut{K_1}_{U_1,\tau_1^{\c G}}, K_1,\tau_1)_{\c G} \relc 
(U_1\rtimes \aut{K_1}_{U_1,\mu_1^{\c G}}, U_1,\mu_1)_{\c G}\,.
\end{equation}

\medskip
{\sc Step 5:} Define 
$\epsilon\colon G_{\tau_1^{\c G}}\to\aut{K_1}_{\tau_1^{\c G}}$ and 
$\hat\epsilon\colon K_1\rtimes\aut{K_1}_{\tau_1^{\c G}} \to \aut{K_1}_{\tau_1^{\c G}}$ 
by $\epsilon(g)(y):=y^g$ and $\hat\epsilon(x,f)(y):=f(x^{-1}yx)$ for $g\in G_{\tau_1^{\c G}}$, $x, y\in K_1$ and 
$f\in \aut{K_1}_{\tau_1^{\c G}}$.
Set $A:=\epsilon(G_{\tau_1^{\c G}})$ and $\hat G:=\hat\epsilon^{-1}(A)$.
Note that $K_1\le \hat G$, since $\hat\epsilon(K_1)=\Inn(K_1) = \epsilon(K_1)\le \epsilon(G_{\tau_1^{\c G}})$.
Therefore we can apply Lemma~\ref{goingsubgroup} to $K_1\le \hat G\le K_1\rtimes\aut{K_1}_{\tau_1^{\c G}}$ 
and (\ref{eqn 3}) to obtain
\begin{equation}\label{eqn 4}
(\hat G,K_1,\tau_1)_{\c G} \relc  ((U_1\rtimes\aut{K_1}_{U_1,\mu_1^{\c G}}) \cap \hat G, U_1,\mu_1)_{\c G}\,.
\end{equation}

\medskip
{\sc Step 6:} We will next replace (\ref{eqn 4}) using the butterfly Theorem~\ref{butterfly} for the comparison
chain of subgroups $K_1\le \norm G{U_1}_{\tau_1^{\c G}}\le G_{\tau_1^{\c G}}$. We have to check that the hypotheses of
Theorem~\ref{butterfly} are satisfied, i.e., that (i) $\hat\epsilon(\hat G)=\epsilon(G_{\tau_1^{\c G}})$, 
(ii) $U_1\cent{G_{\tau_1^{\c G}}}{K_1} \le \norm{G_{\tau_1^{\c G}}}{U_1} \le G_{\tau_1^{\c G}}$, and
(iii) $\hat\epsilon((U_1\rtimes\aut{K_1}_{U_1,\tau_1^{\c G}})\cap \hat G) = \epsilon(\norm{G_{\tau_1^{\c G}}}{U_1})$.
Note first that (ii) is immediate, since 
$\cent{G_{\tau_1^{\c G}}}{K_1}\le \cent{G_{\tau_1^{\c G}}}{U_1}\le \norm{G_{\tau_1^{\c G}}}{U_1}$.

In order to prove (i) and (iii) we will need the equality $G=K_1 \cdot\norm G{U_1}$ which we show next.
Let $g\in G$. Then $U_1^g\le K_1=K_1$ and $c_g\in\aut{K_1}$.
Since $\hat\pi\colon Y^n\to K_1$ is a universal central extension (Step 4), 
\cite[Corollary B.8]{N} implies that there exists $\alpha\in\aut{Y^n}_Z$ such that
$\bar\alpha=c_g\in \aut{K_1}$, i.e., such that $c_g\circ \hat\pi = \hat\pi \circ \alpha$.
By Lemma~\ref{wreath}(b) and (c), $\alpha$ is of the form 
$(f_1,\ldots,f_n) \pi$ for some $f_1,\ldots,f_n\in\aut Y$ and $\pi\in{\sf S}_n$. Since $R$
satisfies the inductive Feit condition (i), for each $i=1,\ldots,n$, there exists
$y_i\in Y$ such that $R^{f_i}=R^{y_i}$. We claim that
$g':=\hat\pi(y_1^{-1},\ldots,y_n^{-1})g=q(y_1^{-1})^{v_1}\cdots q(y_n^{-1})^{v_n} g\in \norm G{U_1}$.
Since $U_1=q(R)^{v_1}\ldots q(R)^{v_n}$, it suffices to show that $q(R)^{v_ig'}\le U_1$ for all $i\in\{1,\ldots,n\}$.
Since $q(R)^{v_i}\le X_i$ commutes with all $q(R)^{v_j}\le X_j$ for $j\neq i$, we have 
$q(R)^{v_ig'}=q(R)^{v_i q(y_i^{-1})^{v_i}g}=q(R)^{q(y_i^{-1})v_ig}=
q(R^{y_i^{-1}})^{v_ig}=q(R^{f_i^{-1}})^{v_ig}=q(R)^{v_{\pi^{-1}(i)}}\le U_1$,
where the last equation follows from, $c_g\circ\hat\pi=\hat\pi\circ(f_1,\ldots,f_n)\pi\colon Y^n\to K_1$ applied to 
$1\times\cdots \times R^{f_i^{-1}}\times \cdots\times 1$ with $R^{f_i^{-1}}$ in the $i$-th component. 
Altogether, we have $g=q(y_1)^{v_1}\cdots q(y_n)^{v_n} g' \in K_1\cdot \norm G {U_1}$. Therefore,
$G=K_1 \cdot\norm G{U_1}$ .

Now we obtain $G_{\tau_1^{\c G}}= (K_1\cdot \norm G{U_1})_{\tau_1^{\c G}} = K_1\cdot \norm G{U_1}_{\tau_1^{\c G}}
= K_1\cdot \norm{G_{\tau_1^{\c G}}}{U_1}$, where the second equation holds, since $K_1$ stabilizes $\tau_1$ 
and in particular $\tau_1^{\c G}$. This implies $\epsilon(G_{\tau_1^{\c G}}) = \
\epsilon(K_1)\cdot\epsilon(\norm{G_{\tau_1^{\c G}}}{U_1}) \le
\hat\epsilon(K_1\rtimes \aut{K_1}_{U_1,\tau_1^{\c G}})$,
since $\epsilon(K_1)=\Inn(K_1)=\hat\epsilon(K_1)$ and 
$\epsilon(\norm{G_{\tau_1^{\c G}}}{U_1}) \le \aut{K_1}_{U_1,\tau_1^{\c G}} =
\hat\epsilon(\aut{K_1}_{U_1,\tau_1^{\c G}})$. Thus, $A\le \mathrm{im}(\hat\epsilon)$ which implies 
$\hat\epsilon(\hat\epsilon^{-1}(A))=A$ so that (i) holds.

In order to prove that the left hand side in (iii) is contained in the right hand side, let $u\in U_1$ and 
$\beta\in\aut{K_1}_{U_1,\tau_1^{\c G}}$ with $(u,\beta)\in \hat G$. Then $\hat\epsilon(u,\beta)=\epsilon(g)$
for some $g\in G_{\tau_1^G}$. This implies $\beta(x^u)=x^g$ for all $x\in K_1$. 
In particular, $U_1^g=\beta(U_1^u)=\beta(U_1)=U_1$, so that 
$g\in \norm{G_{\tau_1^{\c G}}}{U_1}$ and $\hat\epsilon(u,\beta)\in\epsilon(\norm{G_{\tau_1^{\c G}}}{U_1})$. 
Conversely, let $g\in \norm{G_{\tau_1^{\c G}}}{U_1}$. 
Since $\epsilon(g)\in\epsilon(G_{\tau_1^{\c G}})=\hat\epsilon(\hat G)$ by (i), there exists 
$(k,\beta)\in K_1\rtimes\aut{K_1}_{U_1,\tau_1^{\c G}}$ with 
$\epsilon(g)=\hat\epsilon(k,\beta)$. This implies $\beta(x^k)=x^g$ for all $x\in K_1$. In particular,
$\beta(U_1^k) = U_1^g= U_1=\beta(U_1)$ so that $U_1^k=U_1$ and $k\in \norm{K_1}{U_1} = U_1$ (see Step~4). 
This completes the proof of (iii).

Having verified all hypotheses, we can now apply Theorem~\ref{butterfly} to (\ref{eqn 4}) and obtain
\begin{equation}\label{eqn 5}
(G_{\tau_1^{\c G}},K_1,\tau_1)_{\c G} \relc  (\norm{G_{\tau_1^{\c G}}}{U_1}, U_1, \mu_1)_{\c G}\,.
\end{equation}
We claim that we can replace $\norm{G_{\tau_1^{\c G}}}{U_1}$ In (\ref{eqn 5}) with $\norm{G_{\mu_1^{\c G}}}{U_1}$.
Note that this does not follow automatically from the application of the butterfly theorem.
It follows immediately from the equality 
\begin{equation}\label{eqn stab 3}
(\norm G{U_1} \times \c G)_{\tau_1} = (\norm G{U_1} \times \c G)_{\mu_1}
\end{equation}
which we will show now. Let $(g,\sigma)\in \norm G{U_1} \times \c G$ and denote by $\beta\in\Aut(K_1)$ 
the conjugation map by $g$. Then $\beta\in\aut{K_1}_{U_1}$ and 
$\mu_1^g=\mu_1^\beta$ and $\tau_1^g= \tau_1^\beta$. 
Since $\hat\pi\colon Y^n\to K_1$ is a universal central extension with $Z=\ker{\hat\pi}\le R^n$ 
and since $\hat\pi(R^n)=U_1$, there exists $\alpha\in\aut{K_1}_{R^n,Z}$ such that 
$\beta(\hat\pi(\underline y))= \hat\pi(\alpha(\underline y))$ for all $\underline y \in Y^n$, see Corollary B.8(b) in \cite{N}.
Note that $\tilde\rho^{\alpha\sigma} = \mathrm{inf}_{\hat\pi}(\tau_1)^{\alpha\sigma} =
\mathrm{inf}_{\hat\pi}(\tau_1^\beta)^\sigma = \mathrm{\inf}_{\hat\pi}(\tau_1^{\beta\sigma})$ and 
$\tilde\rho=\mathrm{inf}_{\hat\pi}(\tau_1)$, so that 
$$ \tilde\rho^{\alpha\sigma}=\tilde\rho \iff 
\mathrm{inf}_{\hat\pi}(\tau_1^{\beta\sigma}) = \mathrm{inf}_{\hat\pi}(\tau_1) \iff
\tau_1^{\beta\sigma} = \tau_1\,,$$
where the last equivalence follows from the injectivity of $\mathrm{inf}_{\hat\pi}$. Similarly, one has
$\tilde\delta^{\alpha\sigma}=\tilde\delta \iff \mu_1^{\beta\sigma}=\mu_1$. But
$\tilde\rho^{\alpha\sigma}=\tilde\rho$ if and only if $\tilde\delta^{\alpha\sigma}=\tilde\delta$, 
by Equation~(\ref{eqn stab 2}). Altogether we obtain
$\tau_1^{\beta\sigma}=\tau_1$ if and only if $\mu_1^{\beta\sigma}=\mu_1$, as claimed.

Using $\norm{G_{\tau_1^{\c G}}}{U_1} = \norm{G_{\mu_1^{\c G}}}{U_1} = \norm G{U_1}_{\mu_1^{\c G}}$, (\ref{eqn 5}) now becomes
\begin{equation}\label{eqn 5a}
(G_{\tau_1^{\c G}},K_1,\tau_1)_{\c G} \relc  (\norm G{U_1}_{\mu_1^{\c G}}, U_1, \mu_1)_{\c G}\,.
\end{equation}

\medskip
{\sc Step 7:} Now we move back from $K_1$ to $K$. Set $U:=U_1 C$. Then $U_1=K_1\cap U$. We claim that $\norm G U = \norm G{U_1}$. 
In fact, since $C=\zent K$ is normal in $G$, $\norm G{U_1}\le \norm G U$. 
Conversely, let $g\in \norm G U$ and write $g= k n$ with $k\in K_1$ and $n\in \norm G{U_1}$, 
which is possible by the second paragraph in Step 6. Since $n\in \norm G{U_1}\le\norm GU$, we have $k\in\norm{K_1}{U}\le \norm{K_1}{U\cap K_1}= \norm{K_1}{U_1}$. Thus, $g\in \norm G{U_1}$.

We apply now Lemma~\ref{dotproduct} to (\ref{eqn 5a}) with the normal subgroup $C$ in the role of $Z$, 
$\nu\in\irr C$ in the role of $\lambda$, and $\nu_1\in\Irr(C_1)$ in the role of $\nu\in\Irr(Z\cap N)$. It is easy to verify that all the hypotheses in Lemma~\ref{dotproduct} 
are satisfied and we obtain
\begin{equation}\label{eqn 6}
(G_{\tau^{\c G}},K,\tau)_{\c G} \relc  (\norm G U_{\mu^{\c G}},U,\mu)_{\c G}\,,
\end{equation}
with $\mu:=\mu_1\cdot\nu\in\irr{U_1C}=\irr U$, since $\tau_1\cdot \nu=\tau$, as claimed in the statement of the theorem.
Here, we used the fact that if
$(g,\sigma)\in G\times\c G$ stabilizes $\tau$ then also
$\tau_1$, and similarly for $\mu$ and $\mu_1$. 

\medskip
{\sc Step 8:} We now prove the remaining parts of the theorem. First, recall from Step~4 that $\norm{K_1}{U_1}=U_1<K_1$. The natural
isomorphism $K_1/C_1\to K/C$ translates this into $\norm K U = U<K$, as claimed in the statement
of the theorem.

Next we apply Lemma~\ref{key1} to $K\nor G$, $\norm G U\le G$ and (\ref{eqn 6}) 
with $\norm G U_{\tau^{\c G}}$ replaced with $\norm G U_{\mu^{\c G}}$. Note that 
by the second paragraph of Step~6 and the first paragraph of this step, we have
$G=K_1\norm G(U_1) = K_1\norm GU$, so that
$G=K\norm G U$. Moreover, we have 
$K\cap \norm G U = \norm K U= U$. Thus, the hypotheses of Lemma~\ref{key1} are satisfied
and the conclusion of Lemma~\ref{key1} yields the last two sentences in the theorem.
Finally, we show that $N_G(U)$ is self-normalizing in $G$. In fact, if $g\in G$ normalizes $\norm GU$
then it also normalizes $K\cap \norm GU=U$.
This completes the proof of the theorem.  
\end{proof}

%
%
%
The following shows the power of the inductive  Feit condition.
In the presence of simple groups satisfying the inductive Feit condition,
 it implies that information about characters of a group $G$ can be found in a convenient proper
  self-normalizing subgroup of $G$.

\begin{cor}\label{key3}
Suppose that $K \nor G$ with $\zent K \sbs \zent G$.
Suppose that $K/\zent K$ is a minimal normal subgroup of $G/\zent K$
and is isomorphic to $S^n$, 
where $S$ is a non-abelian finite simple group which satisfies the inductive Feit condition.
Let $\tau \in \irr{K}$ be $G$-invariant.
 Then there exist a pair $(U,\mu)$, where $U<K$ and 
$\mu \in \irr U$, such that  $\norm KU=U$,
$G=K\norm GU$,  and a $\c G$-equivariant bijection
$$^*:\irr{G|\tau^{\c G} }\rightarrow  \irr{\norm G U | \mu^{\c G}}\, $$
which respects ratios of degrees.
Therefore, if $\chi \in \irr{G|\tau^{\c G}}$ and $\xi=\chi^*$, then 
$\chi(1)/\tau(1)= \xi(1)/\mu(1)$, $\Q(\chi)=\Q(\xi)$,
and $[\chi_{\zent K}, \xi_{\zent K}] \ne 0$.
\end{cor}

\begin{proof}
It easily follows from Theorem \ref{key2}.
\end{proof}

To what extent Corollary \ref{key3} is true for an arbitrary non-nilpotent normal
subgroup $K$ of $G$ would be studied elsewhere.

\medskip

Recall that if $G$ is solvable, then $G$  has up to conjugation a unique 
nilpotent self-normalizing subgroup, see \cite{C}. Any such subgroup is called a {\bf Carter subgroup} of $G$. 
If $N$ is a normal subgroup of $G$ with nilpotent quotient $G/N$ and $C$ is a Carter subgroup of $G$ 
then $NC=G$, see \cite{C}.

\begin{lem}\label{res}
Suppose that $X$ and $Y$ are normal subgroups of a finite group $G$ with $Y< X$, 
such that $G/X$ is nilpotent, $X/Y$ is an abelian chief factor of $G$, and $G/Y$ is not nilpotent. 
Let $\chi\in\irr G$, $\theta\in\irr X$ and $\varphi\in\irr Y$ be such that 
$\chi_X=e\theta$ for some integer $e$ and $\theta_Y=\varphi$. 
Then there exists $H<G$ such that $\chi_H$ is irreducible and $\Q(\chi)=\Q(\chi_H)$.
In fact, $H=\norm GH$ and $H$ is maximal in $G$.
\end{lem}

 \begin{proof}
Let $H/Y$ be a Carter subgroup of $G/Y$. 
Then $H<G$, since $G/Y$ is not nilpotent, 
and $XH=G$, since $G/X$ is nilpotent. Also $H=\norm GH$.
Moreover, $X/Y \cap H/Y$ is a normal subgroup of $G/Y=H/Y\cdot X/Y$, as it is normalized by $H/Y$ and $X/Y$.
Since $X/Y$ is a chief factor of $G$, this implies $H\cap X=Y$ or $H\cap X=X$. The latter case cannot happen, since
$G/Y$ is not nilpotent. So we have $H\cap X= Y$. In particular, $H$ is maximal in $G$.

By Lemma 6.8(d) of \cite{N}, restriction defines a bijection $\tau\colon \irr{G|\theta} \rightarrow \irr{H|\varphi}$.  
We claim that $\Q(\chi)=\Q(\chi_H)$. 
Clearly, $\Q(\chi_H) \sbs \Q(\chi)$ and $\Q(\varphi)\sbs \Q(\theta) \sbs \Q(\chi)$. 
Let $\sigma \in {\rm Gal}(\Q(\chi)/\Q(\chi_H))$ be arbitrary.
Then $(\chi_H)^\sigma=\chi_H$, and $\chi_Y=e\varphi$ implies $\varphi^\sigma=\varphi$. 
Therefore $\theta^\sigma$ is an extension of $\varphi$ and there exists $\lambda \in \irr {X/Y}$ 
with $\theta^\sigma=\lambda \theta$.
Since $\theta$ is $G$-invariant also $\theta^\sigma$ is $G$-invariant.
By the uniqueness in Gallagher's theorem, we conclude that $\lambda$ is $G$-invariant.
This implies $\ker{\lambda}$ is normal in $G$. Since $Y\le \ker{\lambda} \le X$ and $Y/X$ 
is a chief factor of $G$. We obtain $\ker{\lambda}=X$ or $\ker{\lambda}=Y$. 
If $\ker{\lambda}=X$ then $X/Y\le \zent{G/Y}$ and $G/Y$ is nilpotent, contradicting the
hypothesis of the lemma. Thus, $\lambda=1$ and $\theta^\sigma=\theta$.
But then $\chi^\sigma\in\irr{G|\theta}$ with $\tau(\chi^\sigma)=\chi^\sigma_H=\chi_H=\tau(\chi)$. 
The injectivity of $\tau$ implies $\chi^\sigma=\chi$. Since $\sigma$ was arbitrary, 
we obtain $\Q(\chi)=\Q(\chi_H)$, proving the claim, and completing the proof of the lemma.
\end{proof}

The following theorem summarizes known facts from Isaacs's work in \cite{Is73}, see also
\cite[Lemma~3(ii)]{AC86}.

\begin{thm}\label{isaacs}
 Suppose that $L, K$ are normal of $G$ with $L\le K$, where
 $K/L$ is an abelian $p$-group and $G/K$ is
a $p'$-group. Let $\varphi \in \irr L$ be $G$-invariant and let $\theta \in \irr K$ be such that
$\varphi^K=e\theta$, where $e^2=|K:L|$. Let $U/L$ be a complement of $K/L$ in $G/L$.
Then there exists a character $\psi$ of $U$ with $L\le \ker\psi$ and satisfying the following:
\begin{enumerate}[\rm(a)]
\item $|\psi(u)|^2=|\cent {K/L}u|$ for $u \in U$.

\item $\psi$ is rational valued.

\item The equation $\chi_U=\psi \chi^\prime$ defines a bijection 
$\irr{G|\theta} \rightarrow \irr{U|\varphi}$, $\chi \mapsto \chi^\prime$, 
satisfying $\Q(\chi)=\Q(\chi^\prime)$.
\end{enumerate}
\end{thm}
 
\begin{proof}
This is part of the content of Theorem 9.1 of \cite{Is73}, where the character $\psi$ is denoted by $\psi^{(K/L)}$.
Notice that in our situation, there is a unique complement of $K/L$ in $G$, up to $G$-conjugacy.
The values of $\psi$ are given in (a), (b) and (c) on page 619 of \cite{Is73}, where in that situation, $E=K/L$.
(See also Theorem 6.3.) The fact that $\psi$ is rational valued is given after Corollary 6.4 of \cite{Is73}.
By Theorem 9.1(d) of \cite{Is73}, we have that $\chi(g)=0$ if $g$ is not $G$-conjugate to an element $u \in U$.
Since $\psi$ does not vanish anywhere and is rational-valued, we have that $\Q(\chi)=\Q(\chi^\prime)$.
\end{proof}

We are now ready to prove Theorem~A. (Notice that the first two steps 
can be used to give a new shorter proof of the solvable case of Feit's conjecture.)
Recall that a character $\chi$ is {\bf primitive} if $\chi$ cannot be induced from a proper subgroup of $G$.
 
\begin{thm}\label{main}
Suppose that $G$ is a finite group, $\chi \in \irr G$ with $\chi(1)>1$. Suppose that $\chi$ is primitive.
Assume that the non-abelian composition factors of $G$ satisfy the inductive Feit condition. 
Then there exists $H<G$ and $\psi \in \irr H$ such that $\Q(\chi)= \Q(\psi)$ and $H=\norm GH$.  
\end{thm}

\begin{proof}
{\sc Step 1.}~~{\sl Let $G^\infty$ denote the smallest normal subgroup $N$ of $G$ such that $G/N$ is solvable.
We may assume that $G/G^\infty$ is nilpotent.}
Suppose that $G/G^\infty$ is not nilpotent. Let $X$ denote the smallest normal subgroup $N$ of $G$ 
such that $G/N$ is nilpotent.
Then $G^\infty<X$ and there exists $G^\infty \le Y <X$ such that $X/Y$ is a chief factor of $G$. 
Thus, $X/Y$ is an abelian $p$-group for some prime $p$.

Let $H/Y$ be a Carter subgroup of $G/Y$. Thus $H=\norm GH$. By the same arguments as in the proof of Lemma~\ref{res}, we have
$H<G$, $G=HX$ and $H\cap X=Y$.
Let $P/Y \in\syl p{G/Y}$. Since $G/X$ is nilpotent and $X\le P$, then $P$ is normal in $G$.
Moreover, since $X/Y\nor P/Y$, we have $X/Y\cap \zent{P/Y}>1$.
Since $X/Y$ is a chief factor of $G$, we conclude that $X/Y \sbs \zent{P/Y}$.
Next consider $Q:=P\cap H$ and note that $Q/Y\in\syl p{H/Y}$ is centralized by $X/Y$
and normalized by $H/Y$. Therefore, $Q$ is normal in $G$.
Since $X/Y$ is a chief factor of $G$ also $P/Q$ is a chief factor of $G$.
If $G/Q$ were nilpotent, then also $G/Y$ would be nilpotent since $G/X$
and $G/Q$ are nilpotent and $X\cap Q=Y$. Thus, $G/Q$ is not nilpotent.

Since $\chi$ is primitive, we can write $\chi_P=f\theta$ and $\theta_Q=e\varphi$
with $\theta\in \irr P$ and $\varphi\in \irr Q$. In this situation the Going Down Theorem,
see Theorem~6.18 in \cite{Is}, implies that either $e=1$ or $e^2=|P:Q|$. 
In the first case we are done by Lemma~\ref{res} and in the latter case
we are done by Theorem~\ref{isaacs}, using that $G/P$ is a $p'$-group.


\smallskip
{\sc Step 2.}~~We prove the theorem in the remaining case that   $G/G^\infty$ is nilpotent.
Set $L:=G^{\infty}$. If $L=1$, then $G$ is nilpotent, and $\chi$ is linear, 
since $\chi$ is primitive, contradicting $\chi(1)>1$.
So we may assume that $L>1$.
Let $L/Z$ be a chief factor of $G$ and write $\chi_Z=u\lambda$ with $\lambda\in\irr Z$. 
Note that $L/Z$ is isomorphic to a direct product $S^n$, where $S$ is a non-abelian simple group, 
and that the $n$ minimal normal subgroups of $L/Z$ are permuted transitively by conjugation with $G$.
Note also that, since $S$ is a composition factor in $G$, it satisfies the inductive Feit condition.

We will show that there exists $Z \sbs H<G$, with $H=\norm GH$, and $\psi \in \irr {H|\lambda}$ 
such that $\Q(\chi) = \Q(\psi)$.

To this end we apply Theorem~A and Corollary~B in \cite{L16} to the character triple $(G,Z,\lambda)$. As a consequence, there exists a character triple $(G_1,Z_1,\lambda_1)$ which is isomorphic to the character triple $(G,Z,\lambda)$ such that the bijections between irreducible characters of corresponding intermediate groups between $Z$ and $G$ and between $Z_1$ and $G_1$ preserve character fields. Therefore, we may assume that $(G,Z,\lambda)=(G_1,Z_1,\lambda_1)$. By making this transition we will lose the property that $\chi\in \Irr(G)$ is primitive. But we still retain the property that $\chi$ is not induced from any subgroup of $G$ containing $Z$ and that every composition factor of $G/Z$ satisfies the inductive Feit condition. By the above results from \cite{L16}, the character triple $(G_1,Z_1,\lambda_1)$ has the following additional properties which we may now assume for $(G,Z,\lambda)$, namely that 
there exists a normal cyclic subgroup $C$ of $G$ contained in $Z$
and a faithful character $\nu \in \irr C$ such that $(G, C, \nu)_{\c G}$ is a $\c G$-triple and
$$\nu^Z=\lambda\,,\quad G_\nu=\cent GC\,,\quad G_\nu Z=G\,,\quad \text{and} \quad G_\nu \cap Z=C\,.$$

 
Set $K:=L \cap G_\nu$. Since $K/C\cong L/Z\cong S^n$ has no non-trivial central subgroups, we obtain $C=\zent K$.
Moreover, since $[K,Z]\sbs C$, the group $K/C$ is a direct product of $\cent GC$-conjugate simple groups,
each of which is isomorphic to $S$. Write $\chi_L=v\eta$ with $\eta\in\irr{L|\lambda}$. Since $\eta\in\irr{C|\nu}$ 
and $K=L_\nu$, Clifford theory implies that there exists a unique character $\tau \in \irr{K|\nu}$ such that $\tau^L=\eta$. 
Notice that $\tau$ is $G_\nu$-invariant by the uniqueness of $\tau$, since $\eta$ is $G$-invariant.

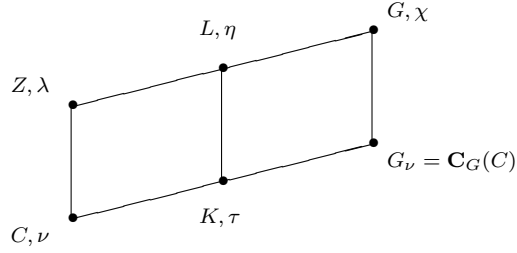
\begin{center}
\unitlength 5mm
{\scriptsize
\begin{picture}(9,7)
\put(0,0){$\bullet$}
\put(0,3){$\bullet$}
\put(4,1){$\bullet$}
\put(4,4){$\bullet$}
\put(8,2){$\bullet$}
\put(8,5){$\bullet$}

\put(0.1,0.1){\line(0,3){3}}
\put(4.1,1.1){\line(0,3){3}}
\put(8.1,2.1){\line(0,3){3}}
\put(0.1,0.1){\line(4,1){8}}
\put(0.1,3.1){\line(4,1){8}}

\put(-1.5,-0.5){$C,\nu$}
\put(-1.5,3.5){$Z,\lambda$}
\put(3.5,0){$K,\tau$}
\put(3.5,5){$L,\eta$}
\put(8.5,1.5){$G_\nu=\cent GC$}
\put(8.5,5.5){$G,\chi$}
\end{picture}
}
\end{center}

\bigskip
By Theorem \ref{key2} there exist a subgroup $C\le U<K$ with $\norm KU=U$
and $G=K\norm GU$, a character $\mu \in \irr U$, and
a $\c G$-equivariant bijection:
$$^*: \irr{G|\tau^{\c G} }\rightarrow  \irr{\norm GU | \mu^{\c G}}\, .$$
Since $K,Z \nor G$ and $K\cap Z=C$, we obtain $[Z,U]\le [Z,K]\le C\le U$, so that $Z \sbs \norm GU$. 
Moreover, $\norm KU=U<K$ implies that $\norm GU<G$. Notice that $\norm GU$ is self-normalizing.
Since $\tau^L=\eta$ and $\chi_L=v\eta$, $\chi$ lies over $\tau$. 
Set $\psi:=\chi^* \in \irr{\norm GU}$. Then $\Q(\psi)=\Q(\chi)$ and $\xi$ lies over $\nu$, 
by the last part in Theorem \ref{key2}. Since $Z\le \norm GU$ this implies that $\psi$ also lies over $\nu^Z=\lambda$.
Thus, the pair $(\norm GU,\psi)$ has the desired properties.
This concludes the proof of the Theorem.
\end{proof}

We obtain now Corollary B as a consequence of Theorem~\ref{main}.

\begin{cor}\label{cormain}
Suppose that $G$ is a finite group and $\chi \in \irr G$. 
Assume that the simple groups involved in $G$ satisfy the inductive Feit condition. 
Then there exists $g \in G$ such that $o(g)=c(\chi)$. 
\end{cor}

\begin{proof}
We argue by induction on $|G|$.
Suppose that $\chi$ is not primitive. Then $\chi=\tau^G$ for some $\tau \in \irr J$, 
where $J<G$. By the character induction formula, we have that
$\Q(\chi) \sbs \Q(\tau)$. Therefore $c(\chi)$ divides $c(\tau)$.
By induction, there exists $h \in J$ such that $o(h)=c(\tau)$.
Then $o(g)=c(\chi)$ for an appropriate power $g$ of $h$.
 If $\chi(1)=1$, then $c(\chi)$ divides $o(\chi)=|G:K|$, where $K=\ker\chi$.
Also, $G/K$ is cyclic and therefore there is $x \in G$ such that $c(\chi)$ divides $o(x)$. 
Now, $c(\chi)=o(g)$ for an appropriate power $g$ of $x$.
If $\chi(1)>1$, then by Theorem \ref{main}, there is $H<G$ and $\psi \in \irr H$ such that $\Q(\chi)= \Q(\psi)$. By induction, there is $h \in H$ such that $o(h)=c(\psi)=c(\chi)$, and we are done. 
\end{proof}

As we have remarked,  we do not necessarily have
in Theorem~\ref{main}  that $\psi$ is an irreducible constituent of $\chi_H$, if we wish to keep the hypothesis that $H=\norm GH$.
Without this hypothesis, we have not found yet a counterexample to that assertion. Also, we have not found an example of 
Theorem \ref{main} where $H$ cannot be chosen 
 a maximal subgroup of $G$. Notice that the proof of Theorem~\ref{main} does give a proof 
 of this for solvable groups.  Also,  it is even possible that $\chi$ and $\psi$
share  more properties besides $\Q(\chi)=\Q(\psi)$, but we leave this for another place. 
If we specifically ask  for $\psi$ to be linear, this constitutes the subject of Conjecture E, which we will deal with in the last Section.  Finally, with essentially the same proof, the statement of Theorem \ref{main} holds if $\chi$ is only assumed to be quasi-primitive (instead of primitive), that is, if $\chi_N$ is a multiple of a single irreducible character of $N$ for every normal subgroup $N$ of $G$. 

\section{Sporadic groups}

In this Section we prove that the sporadic groups satisfy the inductive Feit condition.

\begin{thm}\label{spo}
The sporadic groups satisfy the inductive Feit condition.
\end{thm}

\begin{proof}
We start with the sporadic groups
$\mathrm{Co1}$, $\mathrm{Co2}$, $\mathrm{Co3}$, $\mathrm{Fi23}$, $\mathrm{J1}$, $\mathrm{J4}$, $\mathrm{Ly}$, $\mathrm{M11}$, $\mathrm{M23}$, $\mathrm{M24}$,
$\mathrm{B}$, $\mathrm{Ru}$, $\mathrm{Th}$ and $\mathrm{M}$.
Note that these are the sporadic simple groups  with trivial outer automorphism group.
Let $S$ be any of them and let $X$ be the universal cover of it.
By Theorem~\ref{easysituations}(b), for each $\chi \in \irr X$, it suffices to find a self-normalizing subgroup
$U$ of $X$ and $\mu \in \irr U$ a constituent of $\chi_U$ such that $\Q(\chi)=\Q(\mu)$.
We use \cite{GAP} to find such a pair $(U,\mu)$ for each $\chi$, where $U$ is a maximal subgroup of $X$. 
This is easily done using the character tables of the maximal subgroups of the sporadic groups provided by \cite{GAP}. 
This is not however possible for the Monster. The following list of maximal subgroups covers
all the fields of the characters of the Monster: $\mathrm{2.B}, 
\mathrm{3.Fi24}, \mathrm{2^{1+24}_+ \cdot Co1}$,
$\mathrm{L_2(59)}, 
\mathrm{(D10xHN).2},
\mathrm{L_2(71)}$ and 
$\mathrm{(7:3xHe):2}$ (for each of these types, there is a unique conjugacy class of maximal subgroups, see \cite{Atlas}, \cite{HW1}, \cite{HW2}, ensuring that $U$ is intravariant in $X$). 

Next we consider the universal covering groups $X$ of the remaining sporadic groups $S$, namely $\mathrm{2.M12}, \mathrm{12.M22},  \mathrm{2.J2}, \mathrm{3.J3}, \mathrm{6.Fi22}, \mathrm{3.F3+}, \mathrm{2.HS}, \mathrm{3.McL}, \mathrm{He}, \mathrm{6.Suz}, \mathrm{3.ON}$, and $\mathrm{HN}$. Note that $|\Out(S)|=2$ and $X/\ZB(X)$ is cyclic in all these cases.
We use a combination of different strategies to find a pair $(U,\mu)$ for every $\chi\in\irr X$ such that $(\chi,\mu)$ satisfies the inductive Feit condition. The first strategy is to use the character tables of the maximal subgroups $U$ of $X$, provided by \cite{GAP}, except for $X=\mathrm{3.Fi3+}$.
GAP also provides the character table of a group of the form $X.2$ for each $X$.  For each $\chi \in \irr X$, we select 
all the pairs $(U, \mu)$ such that $\Q(\mu)=\Q(\chi)$, and such that $\mu$ is the only irreducible constituent of $\chi_U$ with degree $\mu(1)$ and multiplicity $[\chi_U, \mu]$ and $\chi$ is the only irreducible constituent of $\mu^G$ with field of values
$\Q(\chi)$, degree $\chi(1)$, and  multiplicity $[\chi_U, \mu]$.
This guarantees that $(\Gamma \times \c G)_\chi=(\Gamma \times \c G)_\mu$, if  $\Gamma=\aut X_U$. If one of these multiplicities is odd, then we apply Theorem \ref{easysituations}(g), and we are done. 
If $\chi^{X.2}$ is irreducible, then we are done by Theorem \ref{easysituations}(b). This way, one can see that a vast majority of the characters $\chi\in\irr X$ satisfy the inductive Feit conjecture.
For instance, it completely eliminates the groups $\mathrm{12.M22}$, $\mathrm{2.Fi22}$ or $\mathrm{McL}$.

The remaining strategies aim at applying Theorem~\ref{easysituations}(d).
For instance we compute Sylow $p$-subgroups $P$ of $X$ and $U:=\norm X P$ to find characters $\mu\in\irr U$ such that the conditions in Theorem~\ref{easysituations}(d) are satisfied. As a variation, if $\chi\in\irr X$ has non-trivial kernel $Z_1\le \ZB(X)$, we may also work with the group $X/Z_1$ instead of $X$. All of them are provided by \cite{GAP}. In combination with the first strategy this eliminates the groups $\mathrm{2.M12}$, $\mathrm{3.McL}$, $\mathrm{He}$, $\mathrm{Fi22}$.

For the smaller groups $\mathrm{2.J2}$ and $\mathrm{2.HS}$, one can directly compute the conjugacy classes of maximal subgroups of $X$ and find matches $(U,\mu)$ for all $\chi\in\irr X$, using again Theorem~\ref{easysituations}(d).

One is now left with very few cases of $\chi\in\irr X$, which have to be dealt with by ad hoc methods, going through the different maximal subgroups $U$ of $X$ and $X.2$, as far as provided by \cite{GAP}, and checking if there are matches $(U,\mu)$ which satisfy the hypotheses of Theorem~\ref{easysituations}(d). 
Let us compute one non-trivial example. Suppose that $X=\mathrm{3.J3}$. There is only one $\chi\in\irr X$ left to be resolved, namely the unique character of degree $816$. It is rational valued and has an extension to $X.2$ with field of values equal to $\mathrm{NF}(24,[ 1, 5, 19, 23 ])$. Unfortunately, \cite{GAP} does not provide the maximal subgroups of $X.2$. But fortunately, $\ker{\chi}=\ZB(X)$, so that we can work with $X/\ZB(X)$ and $X.2/\ZB(X)$, and \cite{GAP} provides the maximal subgroups of those groups. The maximal subgroup $U$ of $X/\ZB(X)$ with the description $\mathrm{"j3m6"}$ has a unique $\mu\in\irr U$ of degree $20$. This character extends to the normalizer of $U$ in $X.2/\ZB(X)$, namely the maximal subgroup with description $\mathrm{"(3\times M10):2"}$. The two extensions of $\mu$ have again the field of value $\mathrm{NF}(24,[ 1, 5, 19, 23 ])$. One can check now that the hypotheses of Theorem~\ref{easysituations} are satisfied so that $(\chi,\mu)$ satisfies the inductive Feit conjecture. The only remaining cases that cannot be handled by such a quick inspection are the ones treated now below in greater detail.

\smallskip
We handle the case $X = 3.S$ the universal cover of $S=\mathrm{Fi'_{24}=F3+}$ and $\chi$ any of the four characters labeled 139--142 in \cite{GAP}.
Any such character $\chi$ has degree 1,152,161,010, and field of values $\Q(\chi)=\Q(\sqrt{-3},\sqrt{5})$. None of them extends to $Y:=3.S.2$, and one can check that $\chi^Y = \chi+\chi'$, where $\chi'=\chi^\sigma$ for $\sigma \in {\rm Gal}(\Q(\chi)/\Q)$ precisely when
$$\sigma=\sigma_2: \sqrt{5} \mapsto -\sqrt{5},~\sqrt{-3} \mapsto -\sqrt{-3},$$
obtained by restricting the Galois automorphism $\zeta_{15} \mapsto \zeta_{15}^2$ of ${\rm Gal}(\Q_{15}/\Q)$.  

By \cite[Theorem 1.1]{LW}, $S$ has a unique conjugacy class of maximal subgroups of type 
$$\bar{U}= \AAA_6 \times \SL_2(8) \rtimes 3,$$ 
which extend to maximal subgroups $\SSS_6 \times \SL_2(8) \rtimes 3 = \langle \bar{U},\bar{v} \rangle$ of $\Aut(S) = S.2$, where 
$$\SSS_6 = \langle \AAA_6,\bar{v}\rangle \cong \AAA_6 \cdot 2_1,~~\SL_2(8) \rtimes 3 = \langle \SL_2(8),\bar{t} \rangle \cong \Aut(\SL_2(8))$$
for an involution $\bar{v}$ and an element $\bar{t}$ of order $3$, and 
$\SL_2(8) \rtimes 3$ centralizes $\AAA_6$. Furthermore,
by restricting the irreducible character of degree 8671 of $S$ to $\bar{U}$, we see that $\bar{t}$  
belongs to class $3c$ in $S$ (in the notation of \cite{GAP}), hence all its lifts to elements in $X$ have order $3$. Moreover, by \cite[Proposition 4.9]{LW}, if 
$U$ denotes the full inverse image of $\bar{U}$ in $X$, then the full inverse image $A$ of the subgroup $\AAA_6 \lhd \bar{U}$ in $U$ is perfect: 
$$A \cong 3\AAA_6.$$
Since $\SL_2(8)$ has trivial Schur multiplier, $\SL_2(8)$ lifts to a subgroup $B \cong \SL_2(8)$ in $U$. Now, $B$ centralizes $A$ modulo $\ZB(X) = \ZB(A)$,
so by perfectness of $A$ and the Three-Subgroup-Lemma,  $B$ also centralizes $A$ in $U$. As mentioned above, we can lift the complement 
$\langle \bar{t} \rangle$ to $\SL_2(8)$ in $\SL_2(8) \rtimes 3$ to a subgroup $\langle t \rangle \cong C_3$ in $U$. Note that $\ZB(\bar{U})=1$ and 
$\ZB(A) \leq \ZB(U)$, so $\ZB(A)=\ZB(U)$. As $t$ induces a non-inner automorphism of $B$ modulo
$\ZB(U)$, $tBt^{-1} \cong B$ is contained in $\ZB(U) \times B$, whence $t \in \NB_U(B)$ but $t \notin A \times B$. Thus $\NB_U(B) \geq \langle A, B,t \rangle = U$, i.e. 
$B \lhd U$. Again using perfectness of $A$ and the Three-Subgroup-Lemma, we see that $t$ generates a subgroup of order $3$ in $U/B$ which centralizes 
$A$. Since $t \notin AB$, we obtain $U/B = A \times C_3$, and thus $U$ has a quotient $U/K \cong A$, where $K = \langle B,t \rangle$. 

Taking an inverse image $v$ of order $2$ of $\bar{v}$ in $X$, we
see that $v \in \NB_Y(U)$ induces an outer automorphism of $\AAA_6$ coming from $\SSS_6 = \AAA_6 \cdot 2_1$. Now, using \cite{GAP} and 
taking $\mu_1$ to be any faithful irreducible character of $U/K \cong $ of degree $3$ (agreeing with $\chi$ on $\ZB(A)$), we can check that $\QQ(\mu_1)=\QQ(\chi)$ and 
$\mu_1^v = \mu_1^\sigma$ for $\sigma \in {\rm Gal}(\Q(\chi)/\Q)$ precisely when $\sigma$ is the above automorphism $\sigma_2$. 
Arguing as above, since $v$ centralizes $B$ modulo $\ZB(U)$, we have $[v,B]=1$ in $U$. Next, $v$ centralizes $t$ modulo $\ZB(U) = \langle z \rangle$, 
so we can write $vtv^{-1}=tz^i$ for some $i = 0,1,2$. Now $2 = |v|=|t^{-1}vt| = |z^iv|$, whence $i=0$ as $z$ is central of order $3$. Thus $[v,t]=1$; in particular,
$v$ normalizes $K$. Inflating $\mu_1$ from $U/K$ to the character $\mu$ of $U$, we have the same action of $v$ on $\mu$. Consequently,
the pair $(U,\mu)$ has all the desired properties to apply Theorem~\ref{easysituations}(d).
\end{proof} 

Essentially the same arguments are used to prove the following statement.

\begin{lem}\label{69}
The simple groups $\AAA_6$ and $\AAA_7$ satisfy the inductive Feit condition.
\end{lem}

\begin{proof}
(a) We start with $\AAA_7$. The center $Z$ of its universal covering group $X$ has order $6$ and $|\Out(\AAA_7)|=2$. \cite{GAP} provides the group $\mathrm{"6.A7.2"}$. Using computations with \cite{GAP}, we find for each $\chi\in\irr{X}$ a self-normalizing and intravariant subgroup $U<X$ ($U$ can be chosen to be a maximal subgroups of $X$ or the normalizer of the Sylow $7$-subgroup of $X$) and an irreducible character $\mu$ of $U$ such that $(\chi,\mu)$ satisfies the inductive Feit condition. 

\smallskip
(b) Now to $\AAA_6$. The center $Z$ of its universal covering group $X$ has again order $6$, but $\Out(\AAA_6)\cong C_2\times C_2$. 
There exists no group of type $6 \cdot \AAA_6 \cdot 2\,\hat{}\,2$, in fact no group of type $2 \cdot \AAA_6 \cdot 2_3$, see \cite{Atlas}, which makes the case of $\AAA_6$ more difficult than $\AAA_7$. 
Using \cite{GAP} for the group $X$ and the strategy described in the second paragraph of the proof of Theorem~\ref{spo} for the sporadic simple groups, we can eliminate 16 of the 31 irreducible characters of $X$. 
Also, \cite{GAP} provides groups of the types $\mathrm{"3.A6.2\,\hat{}\,2"}$ and $\mathrm{"A6.2\,\hat{}\,2"}$. Computations in \cite{GAP} with those groups eliminate all remaining $\chi\in\irr{X}$ with kernel of order $2$ or $6$. 
There are 2 Galois-conjugate irreducible characters of degree 10 and with kernel of order $3$, and 4 Galois-conjugate faithful irreducible characters of degree 12 left. 
Each of these characters $\chi$ of degree 10 (resp.~of degree $12$) satisfies the inductive Feit conjecture together with an irreducible character $\mu$ of the normalizer $U$ of the Sylow $2$-subgroup (resp.~Sylow $5$-subgroup) of $X$ satisfying $[\chi_U,\mu]=3$, by invoking Theorem~\ref{easysituations}(g). This completes the proof.
\end{proof}

\section{Alternating groups: Ordinary characters}
The goal of this section and the next is to prove Theorem C; that is, that the alternating groups 
satisfy the inductive Feit condition. In this section we focus on the ordinary characters of the alternating groups $\AAA_n$,
and in the next on the spin (i.e. faithful) characters of the covers $2\AAA_n$.
Recall that complex irreducible characters of the symmetric group $\SSS_n$ are labeled by partitions 
$$\lambda= (\lambda_1, \lambda_2, \ldots, \lambda_m), \mbox{ where }\lambda_i \in \Z_{\geq 1},~\lambda_1 \geq \lambda_2 \geq \ldots \lambda_m,~\sum^m_{i=1}\lambda_i =n$$ 
of $n$, see e.g. \cite[Theorem 2.1.11]{JK}. To each partition $\lambda \vdash n$ one can associate its Young diagram $Y(\lambda)$ as in \cite[\S1.4]{JK}, which 
carries lots of information about the character $\chi=\chi^\lambda$ labeled by $\lambda$, see e.g. \cite[Chapters 1, 2]{JK} or
\cite[Chapter 4]{FH}. In particular, by \cite[Theorem 2.5.7]{JK}, $\chi^\lam$ is not irreducible over $\AAA_n$ exactly when the partition
$\lam$ is {\it symmetric} (or {\it self-associated}, i.e. $Y(\lambda)$ is stable under the reflection through its main diagonal), in which case the values of each of the two irreducible constituents of $(\chi^\lambda)_{\AAA_n}$ 
are described in \cite[Theorem 2.5.13]{JK}.

\begin{pro}\label{res-an}
Suppose $G=\SSS_n$, $X= \AAA_n$ with $n \geq 2$ and the partition $\lam$ of $n$ is not symmetric. Then $\theta:=(\chi^\lam)_X$ upon restriction to
a subgroup $U:=\AAA_{n-1}$ contains an irreducible character $\mu$ of $U$ with multiplicity one which extends to $\SSS_{n-1}$. Furthermore
$\theta$ satisfies the inductive Feit condition if $n \neq 6$. 
\end{pro}

\begin{proof}
We will freely use the branching rule \cite[Theorem 2.4.3]{JK} from $\SSS_n$ to the stabilizer subgroup $\SSS_{n-1}$ of $n$ in $\SSS_n$, as well as the consequence
of \cite[Theorem 2.5.7]{JK} that, for any $\al,\beta \vdash n$, 
$$[(\chi^\al)_X,(\chi^\beta)_X]_X \neq 0 \mbox{\ \ if and only if\ \ }\beta \in \{\al,\al'\}.$$ 
To prove the first statement, we proceed by induction on $n \geq 5$, with the induction base 
$2 \leq n \leq 5$ easily checked. We embed $U < V < G$ with $V \cong \SSS_{n-1}$.

For the induction step $n \geq 6$, let $a=\lambda_1$ denote the longest (top) row of the Young diagram $Y(\lam)$ of $\lam$, and let
$b=\lambda'_1$ be the longest column of $Y(\lam)$. By symmetry, we may assume 
$$a \geq b.$$ 
Let $B_1, \ldots,B_t$ denote 
the removable boxes of $\lam$, counted from top to bottom of $Y(\lam)$, and for $1 \leq i \leq t$, let $\delta_i \vdash n-1$ 
be the partition whose Young diagram is obtained from $Y(\lam)$ by removing $B_i$, so that 
\begin{equation}\label{a10}
  (\chi^\lam)_V = \sum^t_{i=1}\chi^{\delta_i}.
\end{equation}  
By their construction, $\delta_i \neq \delta_j$ whenever $i \neq j$.

\smallskip
(a) Here we consider the case $a > b$. If $t=1$, then $Y(\lam)$ is a rectangle of size $a \times b$, in which case 
$\delta_1 \neq (\delta_1)'$ and take $\mu= \chi^{\delta_1}$. So assume $t \geq 2$. For any $2 \leq i \leq t-1$, $\delta_i$ has longest row $a$ and longest 
column $b$, whereas $\delta_t$ has longest row $a$ and longest column $\leq b < a$. In particular, $\delta_t$ is not symmetric, and 
$\delta_t \neq (\delta_i)'$ when $2 \leq i \leq t-1$. Also, $\delta_1$ has longest row $\geq a-1$ and longest column $b$, so $\delta_t \neq (\delta_1)'$.
Hence, taking $\mu=(\chi^{\delta_t})_U$, we see using \eqref{a10} that $\mu \in \Irr(U)$ and $[\theta_U,\mu]_U=1$.

\smallskip
(b) From now on we assume $a=b$. As $\lam \neq \lam'$, we have $t \geq 2$. Let $c \geq 1$ be the number of rows of length $a=\lam_1$ in $Y(\lam)$ 
and let $d \geq 1$ the number of columns of length $a=\lam'_1$ in $Y(\lam)$. Again by symmetry we may assume that
$$c \geq d.$$
Here we assume in addition that $c > d$. Then $\delta_t$ has $c$ rows of (largest possible) length $a$, and $d-1$ columns of length $a$, all other being shorter. Thus $(\delta_t)'$ has $c$ columns of (largest possible) length $a$, and $d-1$ rows of length $a$, all other being shorter. On the other hand,
for any $1 \leq i \leq t-1$, $\delta_i$ has $\geq c-1 \geq d$ rows of length $a$. It follows that $\delta_t \neq (\delta_i)'$ when $2 \leq i \leq t$, and 
$\delta_t$ is not symmetric. So we are again done by taking $\mu=(\chi^{\delta_t})_U$.

\smallskip
(c) Finally, assume that $a=b$ and $c=d$. Let $\gamma$ be the partition whose Young diagram is obtained from $Y(\lam)$ by removing the union of
the $c$ longest rows and the $c$ longest columns. Since $\lam \neq \lam'$, $\gam$ is non-empty and not symmetric, in particular, $Y(\gam)$ has 
$m$ boxes where $2 \leq m < n$. Applying the induction  hypothesis to $\gam \vdash m$ (and adding the removed boxes back), we see
that there is some $2 \leq j \leq t-1$ such that $\delta_j \neq (\delta_i)'$ for all $2 \leq i \leq t-1$. Furthermore, $\delta_j$ has 
$c$ rows of length $a$ and $c$ columns of length $a$. On the other hand, $\delta_1$ has $c-1$ rows of length $a$, all other being shorter, and 
$\delta_t$ has $c-1$ columns of length $a$, all other being shorter. Thus $\delta_j$ differs from $(\delta_1)'$ and $(\delta_t)'$ as well.  Hence we  
are done by taking $\mu=(\chi^{\delta_j})_U$.      

\smallskip
For the second statement, first note that $U$ satisfies condition \ref{inductive}(i). Condition \ref{inductive}(ii) holds because
$\theta$ and $\mu$ are both rational and $\SSS_{n-1}$-invariant. Finally, condition \ref{inductive}(iii) holds by Theorem \ref{easysituations}(a).
\end{proof}

We will now work with $\chi^\lam$ where $\lam\vdash n$ is symmetric. First we consider the symmetric hook 
$$\lam_m:=\bigl((m+1)/2,1^{(m-1)/2}\bigr) := \bigl((m+1)/2,\underbrace{1,1, \ldots,1}_{(m-1)/2~{\rm times}}\bigr)$$ 
for any odd $m \geq 1$.
If $n=ab$ where $a,b \in \Z_{\geq 2}$, we denote by $\SSS(a,b)$ the subgroup $\SSS_a \wr \SSS_b$, and $\AAA(a,b) = \SSS(a,b) \cap \AAA_n$.

\begin{pro}\label{res-wr}
Suppose $\chi=\chi^{\lam_n}$ and $n = kl$, where $k,l \geq 3$ are odd. Then the following statements hold.
\begin{enumerate}[\rm(a)]
\item The restriction $\chi_{\SSS(k,l)}$ contains a unique irreducible constituent 
$\varphi$, which lies above the character $\chi^{\lam_k} \boxtimes \chi^{\lam_k} \boxtimes \ldots \boxtimes \chi^{\lam_k}$ of the normal 
subgroup $N=\SSS_k^l$ of $M:=\SSS(k,l)$, with multiplicity $1$. 
\item Furthermore, $\varphi$ splits over $\AAA(k,l)$: $\varphi_{\AAA(k,l)}=\varphi^+ + \varphi^-$ with $\varphi^\pm \in \Irr(\AAA_{k,l})$.
\item $\chi_{\AAA_n} = \theta^++\theta^-$ with $[(\theta^\eps)_{\AAA(k,l)},\varphi^\eps]=1$, 
$\theta^\eps$ and $\varphi^\eps$ have 
the same stabilizer in $\c G$, whence $\QQ(\theta^\eps)=\QQ(\varphi^\eps)=\QQ(\sqrt{(-1)^{(n-1)/2}n})$, for each $\eps=\pm$.
\item $\theta^\eps$ satisfies the inductive Feit condition for each $\eps=\pm$.
\end{enumerate}
\end{pro}

\begin{proof}
(a) Let $V_n \oplus \CC \cong \CC^n$ denote the natural permutation module for $\SSS_n$, so that $V_n$ affords the character $\chi^{(n-1,1)}$.  Then
$\chi$ is afforded by the exterior power $\wedge^{(n-1)/2}(V_n)$, see \cite[\S3.2]{FH}. The restriction of $V_n$ to $\SSS(k,l)$ is 
$$V_{k,1} \oplus V_{k,2} \oplus \ldots \oplus V_{k,l} \oplus V_l,$$
where for each $1 \leq i,j \leq l$, the $j^{\mathrm {th}}$ factor $\SSS_{k|j}$ of $N= \SSS_k^l$ 
acts trivially on $V_{k,i}$ if $j \neq i$ and on $V_l$, and acts via its deleted
natural permutation action if $j=i$. Furthermore, the complement $\SSS_l$ to $N$ in $M$ permutes these $V_{k,i}$, $1 \leq i \leq l$, transitively and 
naturally,
and acts on $V_l$ via its deleted natural permutation action. It follows (from the formula \cite[(B.1), p. 473]{FH} for the exterior power of a direct sum of two modules)
that  $\wedge^{(k-1)l/2}(V_n)$ contains the $N$-submodule
$$A:= \wedge^{(k-1)/2}(V_{k,1}) \otimes \wedge^{(k-1)/2}(V_{k,2}) \otimes \ldots \otimes \wedge^{(k-1)/2}(V_{k,1}).$$
Since each $\wedge^{(k-1)/2}(V_{k|i})$ is acted on irreducibly by $\SSS_{k|i}$ and trivially by all the other 
factors $\SSS_{k|j}$, $A$ is irreducible over $N$. Next, $\wedge^{(n-1)/2}(V_n)$ contains the $M$-submodule 
 $$W= A \otimes B = \wedge^{(k-1)/2}(V_{k,1}) \otimes \wedge^{(k-1)/2}(V_{k,2}) \otimes \ldots \otimes \wedge^{(k-1)/2}(V_{k,1}) \otimes  \wedge^{(l-1)/2}(V_l),$$
where $N$ acts trivially on $B$, but $\SSS_l$ acts irreducibly on $B$, with character $\chi^{\lam_l}$. By Gallagher's theorem \cite[(6.17)]{Is},
$W$ is an irreducible $M$-module, with character say $\varphi$. Also note that $W$ is the $N$-isotypic component of $\wedge^{(n-1)/2}(V_n)$ that 
corresponds to the character $\chi^{\lam_k} \boxtimes \chi^{\lam_k} \boxtimes \ldots \boxtimes \chi^{\lam_k}$. Hence $[\chi_M,\varphi]=1$.

\smallskip
(b) To show that $\varphi$ splits over $\AAA(k,l)$, it suffices to show that $\varphi(x)=0$ for all $x \in M \smallsetminus \AAA_n$. 
We can write $x=yh$, where $y=(y_1, \ldots,y_l) \in N$ and $h \in \SSS_l$. Since
the character $\chi^{\lam_l}$ of the $M/N$-module $B$ splits over $\AAA_l$, we see that $\varphi(x) =0$ if 
$h \notin \AAA_l$. Also, the actions of $h$ on the $k$ $l$-subsets of $\{1,2, \ldots,n\}$ are identical and $2 \nmid k$, so $h \notin \AAA_l$ means exactly
that $h$ is odd as an element of $\SSS_n$. It remains to consider the case $h \in \AAA_l$, so $y$ is odd. 
Conjugating $x$ by an element in $\SSS_l$, we may assume that 
$$h = (1,2,\ldots,a_1)(a_1+1, \ldots,a_1+a_2) \ldots (a_1+\ldots+a_{s-1}+1, \ldots,a_1 + \ldots+a_{s-1}+a_s)$$
is a disjoint product of $s$ cycles of length $a_1, a_2, \ldots,a_s$, where $l=\sum^s_{i=1}a_i$. Note that $\SSS_l$ permutes the $l$ tensor
factors $\wedge^{(k-1)/2}(V_{k,i})$, $1 \leq i \leq l$, of $A$ naturally. It follows from the character formula for tensor induced modules 
\cite{GI} that the trace of $x=yh$ on $A$ is the product of
the $s$ traces of $z_1:=y_1 \ldots y_{a_1}$ on $\wedge^{(k-1)/2}(V_{k,1})$, $z_2:=y_{a_1+1} \ldots y_{a_1+a_2}$ on 
$\wedge^{(k-1)/2}(V_{k,a_1+1})$, $\ldots$, $z_s:=y_{a_1+\ldots+a_{s-1}+1} \ldots y_{a_1 + \ldots+a_{s-1}+a_s}$ on 
$\wedge^{(k-1)/2}(V_{k,a_1+\ldots+a_{s-1}+1})$. Since $y$ is odd, at least one of $z_1, \ldots,z_s$ is odd, and hence its trace is zero
as $\chi^{\lam_k}$ splits over $\AAA_k$. Thus $\varphi(x)=0$ in this case as well.

\smallskip
(c) It is clear that $\chi$ splits over $\AAA_n$, and the two irreducible constituents $\varphi^\pm$ each has field of values
$\FF:=\QQ(\sqrt{(-1)^{(n-1)/2}n})$, see \cite[Theorem 2.5.13]{JK}.
Similarly, each $\wedge^{(k-1)/2}(V_{k,i})$, $1 \leq i \leq l$, splits over $\AAA_{k|i}$, the derived 
subgroup of the $i^{\mathrm {th}}$ factor $\SSS_{k|i}$ of $N$, as the sum of two simple submodules, with characters $\theta^+_i$ and $\theta^-_i$ which are fused by $\SSS_{k|i}$. Note that $N_0:= \AAA_{k|1} \times \AAA_{k|2} \times \ldots \times \AAA_{k|l} \lhd M$. Furthermore, 
$\chi^{\lam_k} \boxtimes \chi^{\lam_k} \boxtimes \ldots \boxtimes \chi^{\lam_k}$ is the only $N$-character that lies above 
$\theta^{+}_1 \boxtimes \theta^{+}_2 \boxtimes \ldots \boxtimes \theta^{+}_l$. By (a), we may write 
$\chi_{\AAA_n}=\theta^++\theta^-$, such that $(\theta^+)_M$ contains $\varphi^+$ and 
$\theta^+_1 \boxtimes \theta^+_2 \boxtimes \ldots \boxtimes \theta^+_l$ is a constituent of $(\varphi^+)_{N_0}$. Moreover,
we now have $[(\theta^\eps)_{\AAA(k,l)},\varphi^\eps]=1$ for $\eps=\pm$.

To prove the remaining assertions, conjugating by an element in $M \smallsetminus \AAA_n$, it suffices to prove them for $\eps=+$.
Suppose $\sigma \in \c G$. As $\chi$ and $\varphi$ is rational, $\sigma$ acts on $\{\theta^+,\theta^-\}$ and on
$\{\varphi^+,\varphi^-\}$. Also, as $\chi^{\lam_k}$ is rational, $\sigma$ sends  
$\theta^{+}_1 \boxtimes \theta^{+}_2 \boxtimes \ldots \boxtimes \theta^{+}_l$ to some
$\theta^{\pm}_1 \boxtimes \theta^{+}_2 \boxtimes \ldots \boxtimes \theta^{\pm}_l$, which then lies under $\tw \sigma (\varphi^+)$, and 
the latter lies under $\tw \sigma(\theta^+)$. The arguments in the preceding paragraph imply that $\sigma$ fixes $\theta^+$ if and only if
$\sigma$ fixes $\varphi^+$.

\smallskip
(d) First we note that $\AAA(k,l)$ is self-normalizing (in fact maximal) and intravariant in $\AAA_n$. Next, $\theta^\eps$ and 
$\varphi^\eps$ have the same stabilizer in $\Gamma \times \c G$, by Lemma \ref{easy2}(i) if $n$ is a square, and 
by Lemma \ref{easy2}(iii) and the arguments in (c), which show 
$$(\theta^\eps)^\Gamma = \{\theta^+,\theta^-\} = (\theta^\eps)^{\c G} \mbox{ and } 
    (\varphi^\eps)^\Gamma =  \{\varphi^+,\varphi^-\}= (\varphi^\eps)^{\c G},$$
otherwise. Hence we are done by Theorem \ref{easysituations}(a). 
\end{proof}

\begin{thm}\label{an-ord}
Let $X = \AAA_n$ with $n \geq 5$ and $n \neq 6$. Then any irreducible character $\theta \in \Irr(X)$ satisfies the inductive Feit condition.
\end{thm}

\begin{proof}
By Proposition \ref{res-an}, we need only consider the case where $\theta$ is an irreducible constituent of some character
$\chi=\chi^\lam$ of $\SSS_n$, with $\lam \vdash n$ being symmetric, so that $\chi_X = \theta^++\theta^-$, $\theta=\theta+$,
and $\Aut(X)_\theta=\Inn(X)$.

\smallskip
(a) First suppose that $\lam=\lam_n$ is the symmetric hook, whence $2 \nmid n$. If $n$ is composite in addition, then we are done by 
Proposition \ref{res-wr}. Suppose $n=p$ is a prime. Then we take 
$$U = \NB_X(P) \cong C_p \rtimes C_{(p-1)/2},$$ 
where 
$P \in \Syl_p(X)$. In this case, $\QQ(\theta) = \QQ(\sqrt{(-1)^{(p-1)/2}p})$ by \cite[Theorem 2.5.13]{JK}, and $\Aut(X)_\theta=\Inn(X)$. 
Choosing $\mu \in \Irr(U)$ of degree $(p-1)/2$, we get $\QQ(\mu)=\QQ(\theta)$ and $\Gamma_\theta=\Gamma_\mu = U$. Applying
Lemma \ref{easy2}(iii) we get $(\Gamma \times \c G)_\theta=(\Gamma \times \c G)_\mu$, and hence we are done by 
Theorem \ref{easysituations}(b). 

\smallskip
(b) It remains to consider the case $\lam$ is symmetric, but not a hook. This means $Y(\lambda)$ has the main diagonal with $s \geq 2$
nodes, with corresponding hook lengths $n_1 > n_2 > \ldots > n_s \geq 1$. Let $\alpha \vdash m:=n-n_s$ be the symmetric partition
whose Young diagram is obtained from $Y(\lam)$ by removing the $s^{\mathrm {th}}$ hook on the main diagonal, and 
$\beta = \lam_{n_s}$ the symmetric hook of $k:=n_s$. Using the Littlewood--Richardson rule \cite[Corollary 2.8.14]{JK}, we see that 
\begin{equation}\label{eq:an30}
  [\chi_{\SSS_a \times \SSS_b},\chi^\al \boxtimes \chi^\beta]=1.
\end{equation}  
Slightly abusing the notation, we have 
$$(\chi^\al)_{\AAA_m} = \al^+ + \al^-,$$
where $\al^\pm \in \Irr(\AAA_a)$ have field of values $\QQ(\sqrt{(-1)^{(m-s+1)/2}n_1 \ldots n_{s-1}})$, and are fused by some $2$-cycle $t \in \SSS_m$.

Assume furthermore that $k > 1$ (hence $k \geq 3$). Then 
$$(\chi^\beta)_{\AAA_k} = \beta^+ + \beta^-,$$
where $\beta^\pm \in \Irr(\AAA_k)$ have field of values $\QQ(\sqrt{(-1)^{(k-1)/2}n_s})$, and are fused by some $t$-cycle $t' \in \SSS_k$. Now we set
$$U:= (\SSS_m \times \SSS_k) \cap \AAA_n = (\AAA_m \times \AAA_k) \rtimes \langle tt' \rangle,$$
and note that $U$ is self-normalizing (in fact maximal) and intravariant in $X$. We may label $\al^\pm$ in such a way that 
$\theta_{U}$ contains $\mu \in \Irr(U)$ with
$$\mu_{\AAA_m \times \AAA_k} = \al^+ \boxtimes \beta^+ + \al^- \boxtimes \beta^-,~\mu=0 \mbox{ outside of }\AAA_m \times \AAA_k.$$ 
Conjugating by $t$ we see that $(\theta^-)_{U}$ contains $\mu^- \in \Irr(U)$ with
$$(\mu^-)_{\AAA_m \times \AAA_k} = \al^+ \boxtimes \beta^- + \al^- \boxtimes \beta^+.$$ 
It follows from \eqref{eq:an30} that $[\theta_U,\mu]=1$. Next, choosing $x \in \AAA_m$ of cycle type $(n_1, \ldots,n_{s-1})$ and 
a $k$-cycle $y \in \AAA_k$ we obtain from \cite[Theorem 2.5.13]{JK} that
$$\begin{aligned}4\mu(xy) & = \bigl( (-1)^{(m-s+1)/2}+\sqrt{(-1)^{(m-s+1)/2}n_1 \ldots n_{s-1}} \bigr)\bigl( (-1)^{(k-1)/2}+ \sqrt{(-1)^{(k-1)/2}n_s} \bigr)\\
    & + \bigl( (-1)^{(m-s+1)/2}-\sqrt{(-1)^{(m-s+1)/2}n_1 \ldots n_{s-1}} \bigr)\bigl( (-1)^{(k-1)/2}- \sqrt{(-1)^{(k-1)/2}n_s} \bigr)\\
    & = 2 \bigl( (-1)^{(n-s)/2}+\sqrt{(-1)^{(n-s)/2}n_1 \ldots n_{s-1}n_s} \bigr) = 4 \theta(xy).\end{aligned}$$
It follows that 
$$\QQ(\theta)=\QQ(\theta(xy)) = \QQ(\mu(xy)) \subseteq \QQ(\mu).$$
We already mentioned that $\mu$ vanishes outside of $\AAA_m \times \AAA_k$. Consider any $g = uv$ with $u \in \AAA_m$ and 
$v \in \AAA_k$. If $u$ is not conjugate to $x$ in $\SSS_m$ and $v$ is not conjugate to $y$ in $\SSS_k$, then $\mu(g) \in \ZZ$ by 
\cite[Theorem 2.5.13]{JK}. If $u$ is conjugate to $x$ in $\SSS_m$ and $v$ is not conjugate to $y$ in $\SSS_k$, then
$\beta^+(v)=\beta^-(v) \in \QQ$ and 
$$\mu(g) = (\al^+(u)+\al^-(u))\beta^+(v) = \pm \beta^+(v) \in \QQ.$$
Similarly, if $u$ is not conjugate to $x$ in $\SSS_m$ and $v$ is conjugate to $y$ in $\SSS_k$, then
$\al^+(u)=\al^-(u) \in \QQ$ and 
$$\mu(g) = (\beta^+(v)+\beta^-(v))\al^+(u) = \pm \al^+(u) \in \QQ.$$
A calculation similar to the above determination of $\mu(xy)$ shows that if $u$ is conjugate to $x$ in $\SSS_m$ and $v$ is conjugate to $y$ in $\SSS_k$,
then 
$$2\mu(g) =  (-1)^{(n-s)/2} \pm \sqrt{(-1)^{(n-s)/2}n_1 \ldots n_{s-1}n_s}.$$
It follows that $\mu(g) \in \QQ(\theta)$ for all $g \in \AAA_m \times \AAA_k$, and hence $\QQ(\theta)=\QQ(\mu)$. 
Clearly, $\Gamma_\theta=\Gamma_\mu = U$. Hence
$\theta$ and $\mu$ have the same stabilizer in $\Gamma \times \c G$ by Lemma \ref{easy2}(i), (iii). It then follows from Theorem 
\ref{easysituations}(b) that $\theta$ satisfies the inductive Feit condition.

\smallskip
(c) Finally, assume that $k=1$. In this case, we can take $U=\AAA_{n-1}$ and $\mu \in \{\al^+,\al^-\}$ a constituent of $\theta_U$. Arguing as
above, we see that $\QQ(\theta)=\QQ(\mu)$ and $\Gamma_\theta=\Gamma_\mu$, and are again done.
\end{proof}

\section{Alternating groups: Spin characters}
In this section we complete the proof of Theorem C, by considering the spin characters of the double covers $2\AAA_n$. We begin by recalling some more notation.

\subsection{Partitions and symmetric functions}
We denote by $\Par(n)$ the set of partitions of $n$ and by $\unlhd$ the {\em dominance order} on partitions, see \cite[I.1]{Macdonald}. Writing a partition $\lambda$ in the form $\lambda=(1^{m_1},2^{m_2},\dots)$, we define the integer
$$
z_\lambda:=m_1!\, 1^{m_1}\, m_2!\, 2^{m_2}\,\cdots.
$$
Given $\mu\in\Par(m)$ and $\nu\in\Par(n)$ we denote by $\mu\sqcup\nu$ the partition of $m+n$ whose parts are the multiset union of the parts of $\mu$ and $\nu$.

For a partition $\lambda$ we denote by $\ell(\lambda)$ the number of non-zero parts of $\lambda$. A partition $\lambda\in\Par(n)$ is called {\em even} (resp. {\em odd}) if $n-\ell(\lambda)$ is even (resp. odd). We denote by $\Par_\0(n)$ (resp. $\Par_\1(n)$) the set of all even (resp. odd) partitions of $n$.
We denote by $\OP(n)$ the set of all $\lambda\in\Par(n)$ such that all non-zero parts of $\lambda$ are odd. Note that $\OP(n)\subseteq \Par_\0(n)$. 
We denote by $\DP(n)$ the set of all {\em strict partitions} of $n$, i.e. partitions $\lambda\in\Par(n)$ whose non-zero parts are distinct. For $\lambda\in \DP(n)$, we have  $z_\lambda=\lambda_1\lambda_2\cdots\lambda_{\ell(\lambda)}$. 
Denote $\DP_\0(n):=\DP(n)\cap\Par_\0(n)$ and $\DP_\1(n):=\DP(n)\cap\Par_\1(n)$.

We review some necessary facts from the theory of symmetric functions, referring the reader to \cite[\S\S5,6]{Stembridge} and \cite[III.8]{Macdonald} for details. Let  $\La=\bigoplus_{n\geq 0}\La^n$ be the graded algebra of  symmetric functions over $\Q$ in the variables $x_1,x_2,\dots$. For a partition $\lambda$, we denote $x^\lambda:=x_1^{\lambda_1}x_2^{\lambda_2}\cdots$. For $r=1,2,\dots,$ we have the $r$th power-sum symmetric functions $p_r\in\La$. The symmetric functions $q_n\in\La^n$ are defined from 
$$
\sum_{n\geq 0}q_nt^n=\prod_{j\geq 1}\frac{1+x_jt}{1-x_jt},
$$
For a partition $\lambda$, we denote $p_\lambda:=p_{\lambda_1}p_{\lambda_2}\cdots $ and $q_\lambda:=q_{\lambda_1}q_{\lambda_2}\cdots $.

Let $\Om=\bigoplus_{n\geq 0}\Om^n$ be the graded subalgebra of $\La$ generated by $p_1,p_3,p_5,\dots$. Then $q_n\in \Om^n$ for all $n$, see \cite[(5.4)]{Stembridge}. Moreover, 
$\{p_\lambda\mid \lambda\in\OP(n)\}$ and $\{q_\lambda\mid \lambda\in\OP(n)\}$ are bases of $\Om^n$. 
The bilinear form  
$[\cdot,\cdot]$  on $\Om^n$ is defined from $[p_\lambda,p_\mu]=z_\lambda 2^{-\ell(\lambda)}\delta_{\lambda,\mu}$ for all $\lambda,\mu\in\OP(n)$. 

For $\lambda\in\DP(n)$, the {\em Schur's $Q$-function} $Q_\lambda\in \Om^n$ can be defined as 
$$
Q_\lambda=\sum_{\mu\in\Par(n)}K'_{\lambda,\mu}x^\mu
$$ 
where $K'_{\lambda,\mu}$ is the number of shifted tableaux of shape $\lambda$ and content $\mu$, see \cite[(6.3)]{Stembridge}. 
By \cite[Corollary 6.2(b)]{Stembridge}, $\{Q_\lambda\mid \lambda\in\OP(n)\}$ is a basis of $\Om^n$. By \cite[Lemma 6.3]{Stembridge}, we have 
\begin{equation}\label{EK'}
K'_{\lambda,\lambda}=2^{\ell(\lambda)}\quad\text{and}\quad \text{$K'_{\lambda,\mu}=0$ unless $\lambda\unrhd\mu$}.
\end{equation}
By \cite[(6.6)]{Stembridge}, we have
\begin{equation}\label{EqQ}
q_\mu=\sum_{\lambda\in\DP(n)}K'_{\lambda,\mu}2^{-\ell(\lambda)}Q_\lambda.
\end{equation}
For $\mu\in\DP(m),\,\nu\in\DP(n),\,\lambda\in\DP(m+n)$ define the numbers $f^\lambda_{\mu,\nu}$ from 
\begin{equation}\label{EF}
Q_\mu Q_\nu=\sum_{\lambda\in\DP(m+n)}2^{\ell(\mu)+\ell(\nu)-\ell(\lambda)}f^\lambda_{\mu,\nu}Q_\lambda.
\end{equation}

\begin{lem} \label{LQFun}
If $\mu\in\DP(m)$, $\nu\in\DP(n)$ and $\lambda\in\DP(m+n)$, then $f^\lambda_{\mu,\nu}=0$ unless $\lambda\unrhd \mu\sqcup\nu$, and $f^{\mu\sqcup\nu}_{\mu,\nu}=1$ if $\mu\sqcup\nu$ is a strict partition. 
\end{lem}
\begin{proof}
By definition, for any $\al\in\Par(m)$ and $\beta\in\Par(n)$, we have $q_\al q_\beta=q_{\al\sqcup\beta}$. 
By (\ref{EK'}) and (\ref{EqQ}), we can write 
$$
Q_\mu=q_\mu+\sum_{\al\rhd\mu}a_\al q_\al\quad\text{and}\quad 
Q_\nu=q_\nu+\sum_{\beta\rhd\nu}b_\beta q_\beta
$$
Hence, interpreting $a_\mu$ and $a_\nu$ as $1$, we obtain
$$
Q_\mu Q_\nu=q_{\mu\sqcup\nu}+\sum_{\substack{\al\unrhd\mu,\,\beta\unrhd\nu\\(\al,\beta)\neq(\mu,\nu)}}a_\al a_\beta \,q_{\al\sqcup\beta}
=q_{\mu\sqcup\nu}+\sum_{\gamma\rhd\mu\sqcup\nu}c_\gamma q_{\gamma}
$$
for some scalars $c_\gamma$. 
By (\ref{EK'}) and (\ref{EqQ}) again, we deduce that the last expression can be written as a linear combination of the $Q_\lambda$ with $\lambda\unrhd \mu\sqcup\nu$, and the coefficient of $Q_{\mu\sqcup\nu}$ is $1$ if  $\mu\sqcup\nu$ is a strict partition. The lemma now follows from (\ref{EF}) using $\ell(\mu\sqcup\nu)=\ell(\mu)+\ell(\nu)$. 
\end{proof}

\subsection{Symmetric and alternating groups}
As before, let $\sfS_n$ be the symmetric group on $n$ letters with simple generators $\sfs_j:=(j,j+1)$, $j=1,\dots,n-1$, and $\sfA_n\leq \sfS_n$ be the corresponding alternating group. 

For $\lambda\in\Par(n)$ we denote by $\sfC^\lambda$ the conjugacy class of $\sfS_n$ consisting of the elements of cycle-type $\lambda$. 
The centralizer of any $\sfs\in \sfC^\lambda$ in $\sfS_n$ has order $z_\lambda$. 
The elements of $\sfC^\lambda$ are even if and only if $\lambda$ is even. 
We have the canonical representative 
$$
\sfs^\lambda:=(\sfs_1\cdots\sfs_{\lambda_1-1})(\sfs_{\lambda_1+1}\cdots\sfs_{\lambda_1+\lambda_2-1})\cdots 
(\sfs_{\lambda_1+\dots+\lambda_{\ell(\lambda)-1}+1}\cdots\sfs_{n-1})\in\sfC^\lambda. 
$$

If $r_1,\dots,r_k\in\Z_{>0}$ and $r_1+\dots+r_k=n$, we have the standard parabolic subgroup 
$$
\sfS_{r_1,\dots,r_k}\cong \sfS_{r_1}\times\dots\times \sfS_{r_k}\leq \sfS_n.
$$
We canonically embed each $\sfS_{r_j}$ into $\sfS_{r_1,\dots,r_k}$. Denote
$$
\sfA_{r_1,\dots,r_k}:=\sfA_n\cap \sfS_{r_1,\dots,r_k}. 
$$
Note that $\sfA_{r_1,\dots,r_k}\leq \sfS_{r_1,\dots,r_k}$ is an index $2$ subgroup if $\sfS_{r_1,\dots,r_k}$ is non-trivial.

\subsection{Double covers of symmetric and alternating groups} 

We consider a Schur's double cover of $\sfS_n$ defined as the group $\TS_n$ given by generators  
$
\ts_1,\dots,\ts_{n-1},\tz$ and relations 
$$
\ts_j\tz=\tz\ts_j,\ \ \tz^2=1,\ \ \ts_j^2=\tz,\ \ (\ts_i\ts_{i+1})^3=\tz,\ \ \ts_k\ts_{l}=\tz\ts_l\ts_{k}
$$
for all $1\leq j<n$, $1\leq i<n-1$, and $1\leq k,l<n$ with $|k-l|>1$.

There is a surjective homomorphism $\TS_n\to\sfS_n,\ \ts_j\mapsto \sfs_j$ whose kernel of order $2$ is generated by $\tz$. For $\ts\in\sfS_n$, we denote by $|\ts|$ the image of $\ts$ under this homomorphism. We have the Schur's double cover $
\TA_n=\{\ts\in\TS_n\mid|\ts|\in\sfA_n\}$ of $\sfA_n$.

Let $\lambda\in\Par(n)$. We have a subset
$
\TiC^\lambda:=\{\ts\in\TS_n\mid|\ts|\in\sfC^\lambda\}\subseteq\TS_n,
$
and a canonical representative
$$
\ts^\lambda:=(\ts_1\cdots\ts_{\lambda_1-1})(\ts_{\lambda_1+1}\cdots\ts_{\lambda_1+\lambda_2-1})\cdots 
(\ts_{\lambda_1+\dots+\lambda_{\ell(\lambda)-1}+1}\cdots\ts_{n-1})\in\TiC^\lambda.
$$

The following two lemmas were proved by Schur \cite{Schur}, see also \cite[Theorems 3.8,\,3.9]{HH} and \cite[Theorems 2.1,\,2.7]{Stembridge}.

\begin{lem} 
Let $\lambda\in\Par(n)$. Then one of the following happens:
\begin{enumerate}
\item[{\rm (i)}] $\lambda\not\in \OP(n)\sqcup\DP_\1(n)$ and\, $\TiC^\lambda$ is one conjugacy class of\, $\TS_n$;
\item[{\rm (ii)}] $\lambda\in \OP(n)\sqcup\DP_\1(n)$ and\, $\TiC^\lambda$ splits into two conjugacy classes 
of\, $\TS_n$ with representatives\, $\ts^\lambda
$ and\, $\tz\ts^\lambda
$. 
\end{enumerate}
Moreover, the conjugacy classes in (i) and (ii) exhaust all the conjugacy classes of $\TS_n$.
\end{lem}

\begin{lem} 
Let $\lambda\in\Par_\0(n)$. Then one of the following happens:
\begin{enumerate}
\item[{\rm (i)}] $\lambda\not\in \DP_\0(n)\cup\OP(n)$ and\, $\TiC^\lambda$ is one conjugacy class of\, $\TA_n$;
\item[{\rm (ii)}] $\lambda\in \DP_\0(n)\cup\OP(n)$ and\, $\TiC^\lambda$ splits into two conjugacy classes of\, $\TA_n$ with representatives\, $\ts^\lambda$ and\, $\tz\ts^\lambda$. 
\end{enumerate}
Moreover, the conjugacy classes in (i) and (ii) exhaust all the conjugacy classes of $\TA_n$.
\end{lem}


Let $\tz\in G\leq \TS_n$. Then $\Irr(G)=\Irr^-(G)\sqcup\Irr^+(G)$ where $\Irr^\pm(G)$ denotes the irreducible characters corresponding to the modules on which $\tz$ acts as $\pm\id$. 
The following two theorems were proved by Schur \cite{Schur}, see also \cite[Theorems 8.6,\,8.7]{HH} and \cite[Theorem 7.1 and Corollary 7.5]{Stembridge}. 

\begin{thm} \label{TIrrS}
We have\,
$\Irr^-(\TS_n)=\{\phi^\lambda\mid\lambda\in\DP_\0(n)\}\sqcup\{\phi^\lambda_\pm\mid \lambda\in\DP_1(n)\}$ where 
\begin{align*}
\phi^\lambda(\ts^\mu)
&=
\left\{
\begin{array}{ll}
2^{(\ell(\mu)-\ell(\lambda))/2}[Q_\lambda,p_\mu] &\hbox{if $\mu\in\OP(n)$,}\\
0 &\hbox{otherwise,}
\end{array}
\right.
\qquad\ \ \ \qquad(\lambda\in\DP_\0(n)),
\\
\phi^\lambda_\pm(\ts^\mu)
&=
\left\{
\begin{array}{ll}
2^{(\ell(\mu)-\ell(\lambda)-1)/2}[Q_\lambda,p_\mu] &\hbox{if $\mu\in\OP(n)$,}\\
\pm i^{(n-\ell(\lambda)+1)/2}\sqrt{z_\lambda/2} &\hbox{if $\mu=\lambda$,}
\\
0 &\hbox{otherwise}
\end{array}
\right.
\qquad\qquad(\lambda\in\DP_\1(n)).
\end{align*}
\end{thm}

\begin{thm} \label{TIrrA}
We have\,
$\Irr^-(\TA_n)=\{\psi^\lambda_\pm\mid\lambda\in\DP_\0(n)\}\sqcup\{\psi^\lambda_\pm\mid \lambda\in\DP_1(n)\}$ 
where 
\begin{align*}
\psi^\lambda_\pm(\ts^\mu)
&=
\left\{
\begin{array}{ll}
2^{(\ell(\mu)-\ell(\lambda)-2)/2}\,[Q_\lambda,p_\mu] &\hbox{if $\mu\in\OP(n)$,}\\
\pm i^{(n-\ell(\lambda))/2}\sqrt{z_\lambda}/2 &\hbox{if $\mu=\lambda$,}
\\
0 &\hbox{otherwise}
\end{array}
\right.
\qquad\qquad(\lambda\in\DP_\0(n)),
\\
\psi^\lambda(\ts^\mu)
&=
\left\{
\begin{array}{ll}
2^{(\ell(\mu)-\ell(\lambda)-1)/2}\,[Q_\lambda,p_\mu] &\hbox{if $\mu\in\OP(n)$,}\\
0 &\hbox{otherwise}
\end{array}
\right.
\qquad\, \qquad(\lambda\in\DP_\1(n)).
\end{align*}
Moreover, $\phi^\lambda{\downarrow}_{\TA_n}=\psi^\lambda_+ + \psi^\lambda_-$ for\, $\lambda\in \DP_\0(n)$ and\, $\phi^\lambda_\pm{\downarrow}_{\TA_n}=\psi^\lambda$ for\, $\lambda\in \DP_\1(n)$.
\end{thm}

\begin{rem} \label{RBasicSpin}
{\rm 
Note that the irreducible characters corresponding to the case where $\lambda=(n)$ in Theorems~\ref{TIrrS},\,\ref{TIrrA} are the {\em basic spin characters}, see e.g. \cite[Theorem 3.3]{Stembridge}.  
}
\end{rem}

Note that the powers of $2$ appearing in the theorems are integers and $[Q_\lambda,p_\mu]\in\Q$, so the theorems immediately imply the known result for the fields of character values:

\begin{cor} \label{CCharVal} 
We have:
\begin{enumerate}
\item[{\rm (i)}] If $\lambda\in\DP_\0(n)$ then\, 
$\Q(\phi^\lambda)=\Q$ and\, $\Q(\psi^\lambda_\pm)=\Q\big(\sqrt{(-1)^{(n-\ell(\lambda))/2}z_\lambda}\big)$. 

\item[{\rm (ii)}] If $\lambda\in\DP_\1(n)$ then\, 
$\Q(\phi^\lambda_\pm)=\Q\big(\sqrt{(-1)^{(n-\ell(\lambda)+1)/2}z_\lambda/2}\,\,\big)$ and\, $\Q(\psi^\lambda)=\Q$. 
\end{enumerate}
\end{cor}

For $m,n\in\Z_{>0}$, we now consider the subgroups 
$$
\TS_{m,n}:=\{\ts\in\TS_{m+n}\mid|\ts|\in\sfS_{m,n}\}< \TS_{m+n}
\quad \text{and} \quad  
\TA_{m,n}:=
\TS_{m,n}\cap \TA_{m+n}< \TA_{m+n}.
$$
Note that we can canonically embed $\TS_{m}$ and $\TS_n$ into $\TS_{m,n}$. 

Denote
\begin{align*}
\DP_\0(m,n)&:=(\DP_\0(m)\times \DP_\0(n))\,\sqcup\, (\DP_\1(m)\times \DP_\1(n)),
\\
\DP_\1(m,n)&:=(\DP_\0(m)\times \DP_\1(n))\,\sqcup\, (\DP_\1(m)\times \DP_\0(n)).
\end{align*}

The following result follows from \cite[Proposition 4.2,\,Theorem 4.3]{Stembridge}, see \cite[Lemma 6.2]{Livesey}. It comes from Stembridge's construction of  reduced Clifford product of irreducible characters (which is a counterpart of the operation $\circledast$ for supermodules, see for example \cite[\S12.2]{KBook}). 

\begin{thm} \label{TMNS}
We have
$$\Irr^-(\TS_{m,n})=\{\phi^{\mu,\nu}\mid(\mu,\nu)\in\DP_\0(m,n)\}\sqcup\{\phi^{\mu,\nu}_\pm\mid (\mu,\nu)\in\DP_1(n)\},$$ 
where, for $x\in\TS_m$ and\, $y\in\TS_n$, the following hold:
\begin{enumerate}
\item[{\rm (i)}] If $\mu\in\DP_\0(m)$ and $\nu\in\DP_\0(n)$ then 
$\phi^{\mu,\nu}(xy)=\phi^\mu(x)\phi^\nu(y)$. 
\item[{\rm (ii)}] If $\mu\in\DP_\1(m)$ and $\nu\in\DP_\1(n)$ then 
$$
\phi^{\mu,\nu}(xy)=
\left\{
\begin{array}{ll}
2\psi^\mu(x)\psi^\nu(y) &\hbox{if $x\in \TA_m$ and $y\in \TA_n$,}\\
0 &\hbox{otherwise.}
\end{array}
\right.
$$
\item[{\rm (iii)}] If $\mu\in\DP_\0(m)$ and $\nu\in\DP_\1(n)$ then 
$$
\phi^{\mu,\nu}_\pm(xy)=
\left\{
\begin{array}{ll}
\psi^\mu_+(x)\phi^\nu_\pm(y)+\psi^\mu_-(x)\phi^\nu_\mp(y) &\hbox{if $x\in \TA_m$,}\\
0 &\hbox{otherwise.}
\end{array}
\right.
$$
\item[{\rm (iv)}] If $\mu\in\DP_\1(m)$ and $\nu\in\DP_\0(n)$ then 
$$
\phi^{\mu,\nu}_\pm(xy)=
\left\{
\begin{array}{ll}
\phi^\mu_\pm(x)\psi^\nu_+(y)+\phi^\mu_\mp(x)\psi^\nu_-(y) &\hbox{if $y\in \TA_m$,}\\
0 &\hbox{otherwise.}
\end{array}
\right.
$$
\end{enumerate}
\end{thm}

The following result follows from \cite[Lemma 6.3]{Livesey}. 

\begin{thm} \label{TMNA} 
We have
$$\Irr^-(\TA_{m,n})=\{\psi^{\mu,\nu}_\pm\mid(\mu,\nu)\in\DP_\0(m,n)\}\sqcup\{\psi^{\mu,\nu}\mid (\mu,\nu)\in\DP_1(n)\},$$ 
where, for $x\in\TS_m$ and\, $y\in\TS_n$ with $xy\in\TA_{m,n}$, the following hold:
\begin{enumerate}
\item[{\rm (i)}] If $\mu\in\DP_\0(m)$ and $\nu\in\DP_\0(n)$ then 
$$
\psi^{\mu,\nu}_\pm(xy)=
\left\{
\begin{array}{ll}
\psi^\mu_+(x)\psi^\nu_\pm(y)+\psi^\mu_-(x)\psi^\nu_\mp(y) &\hbox{if $x\in \TA_m$ and $y\in \TA_n$,}\\
0 &\hbox{otherwise.}
\end{array}
\right.
$$

\item[{\rm (ii)}] If $\mu\in\DP_\1(m)$ and $\nu\in\DP_\1(n)$ then 
$$
\psi^{\mu,\nu}_\pm(xy)=
\left\{
\begin{array}{ll}
\psi^\mu(x)\psi^\nu(y) &\hbox{if $x\in \TA_m$ and $y\in \TA_n$,}\\
\pm i \,\phi^\mu_+(x)\phi^\nu_+(y) &\hbox{otherwise.}
\end{array}
\right.
$$
\item[{\rm (iii)}] If $\mu\in\DP_\0(m)$ and $\nu\in\DP_\1(n)$ then 
$$
\psi^{\mu,\nu}(xy)=
\left\{
\begin{array}{ll}
\phi^\mu(x)\psi^\nu(y) &\hbox{if $x\in \TA_m$ and $y\in \TA_n$,}\\
0 &\hbox{otherwise.}
\end{array}
\right.
$$
\item[{\rm (iv)}] If $\mu\in\DP_\1(m)$ and $\nu\in\DP_\0(n)$ then 
$$
\psi^{\mu,\nu}(xy)=
\left\{
\begin{array}{ll}
\psi^\mu(x)\phi^\nu(y) &\hbox{if $x\in \TA_m$ and $y\in \TA_n$,}\\
0 &\hbox{otherwise.}
\end{array}
\right.
$$
\end{enumerate}
Moreover, $\phi^{\mu,\nu}{\downarrow}_{\TA_{m,n}}=\psi^{\mu,\nu}_+ + \psi^{\mu,\nu}_-$ for $({\mu,\nu})\in \DP_\0(n)$ and $\phi^{\mu,\nu}_\pm{\downarrow}_{\TA_n}=\psi^{\mu,\nu}$ for $({\mu,\nu})\in \DP_\1(n)$.
\end{thm}

\begin{cor} \label{CCharValMN} 
We have:
\begin{enumerate}
\item[{\rm (i)}] If $(\mu,\nu)\in\DP_\0(m,n)$ then 
$$\Q(\phi^{\mu,\nu})=\Q\quad \text{and} \quad \Q(\psi^{\mu,\nu}_\pm)=\Q\big(\sqrt{(-1)^{(m+n-\ell(\mu)-\ell(\nu))/2}z_\mu z_\nu}\,\,\big).
$$ 
\item[{\rm (ii)}] If $(\mu,\nu)\in\DP_\1(n)$ then 
$$
\Q(\phi^{\mu,\nu}_\pm)=\Q\big(\sqrt{(-1)^{(m+n-\ell(\mu)-\ell(\nu)+1)/2}z_\mu z_\nu/2}\,\,\big)\quad \text{and}\quad \Q(\psi^{\mu,\nu})=\Q.
$$ 
\end{enumerate}
\end{cor}
\begin{proof}
This follows from Theorems~\ref{TIrrS},\,\ref{TIrrA},\,\ref{TMNS},\,\ref{TMNA}, upon noticing that the powers of $2$ appearing in the character values in Theorems~\ref{TIrrS},\,\ref{TIrrA} are integers and inner products $[Q_\lambda,p_\mu]$ appearing  in Theorems~\ref{TIrrS},\,\ref{TIrrA} belong to $\Q$. 
\end{proof}

\begin{cor} \label{CMNCharVal}
Let $\lambda\in\DP(m+n)$. Suppose that $\lambda=\mu\sqcup\nu$ for $\mu\in\DP(m)$ and $\nu\in\DP(n)$. 
\begin{enumerate}
\item[{\rm (i)}] If $\lambda\in\DP_\0(m+n)$ then $(\mu,\nu)\in\DP_\0(m,n)$, 
$$\Q(\phi^\lambda)=\Q(\phi^{\mu,\nu})\quad \text{and}\quad  
\Q(\psi^\lambda_+)=
\Q(\psi^\lambda_-)=\Q(\psi^{\mu,\nu}_+)=\Q(\psi^{\mu,\nu}_-).
$$
\item[{\rm (ii)}] If $\lambda\in\DP_\1(m+n)$ then $(\mu,\nu)\in\DP_\1(m,n)$, 
$$
\Q(\phi^\lambda_+)=
\Q(\phi^\lambda_-)=\Q(\phi^{\mu,\nu}_+)=\Q(\phi^{\mu,\nu}_-)
\quad \text{and}\quad  
\Q(\psi^\lambda)=\Q(\psi^{\mu,\nu}).
$$
\end{enumerate}
\end{cor}
\begin{proof}
As $\ell(\lambda)=\ell(\mu)+\ell(\nu)$ and $z_\lambda=z_\mu z_\nu$, this follows from Corollaries~\ref{CCharVal} and \ref{CMNCharVal}. 
\end{proof}

The following is a restatement of \cite[Theorems 8.1]{Stembridge}. In the case $\lambda=\mu\sqcup\nu$ in part (iv) we have used the additional character information coming from Theorems~\ref{TIrrS},\,\ref{TIrrA},\,\ref{TMNS} to resolve the uncertainty remaining in the exceptional case of \cite[Theorems 8.1]{Stembridge} (although this will not be needed in this paper). 

\begin{thm} \label{TBranchSpin}
Let\, $\mu\in\DP(m),\,\nu\in\DP(n)$ and\, $\lambda\in\DP(m+n)$. 
\begin{enumerate}
\item[{\rm (i)}] If $\lambda$ is even and $(\mu,\nu)$ is even then 
$$
[\phi^\lambda{\downarrow}_{\TS_{m,n}}:\phi^{\mu,\nu}]=2^{(\ell(\mu)+\ell(\nu)-\ell(\lambda))/2}f^\lambda_{\mu,\nu}.
$$

\item[{\rm (ii)}]  If $\lambda$ is even and $(\mu,\nu)$ is odd then 
$$
[\phi^\lambda{\downarrow}_{\TS_{m,n}}:\phi^{\mu,\nu}_\pm]=2^{(\ell(\mu)+\ell(\nu)-\ell(\lambda)-1)/2}f^\lambda_{\mu,\nu}.
$$

\item[{\rm (iii)}] If $\lambda$ is odd and $(\mu,\nu)$ is even then 
$$
[\phi^\lambda_\pm{\downarrow}_{\TS_{m,n}}:\phi^{\mu,\nu}]=2^{(\ell(\mu)+\ell(\nu)-\ell(\lambda)-1)/2}f^\lambda_{\mu,\nu}.
$$

\item[{\rm (iv)}] If $\lambda$ is odd and $(\mu,\nu)$ is odd then 
$$
[\phi^\lambda_\pm{\downarrow}_{\TS_{m,n}}:\phi^{\mu,\nu}_\mp]=
[\phi^\lambda_\pm{\downarrow}_{\TS_{m,n}}:\phi^{\mu,\nu}_\pm]=2^{(\ell(\mu)+\ell(\nu)-\ell(\lambda)-2)/2}f^\lambda_{\mu,\nu}.
$$
unless $\lambda=\mu\sqcup\nu$, in which case 
$$
[\phi^{\mu\sqcup\nu}_\pm{\downarrow}_{\TS_{m,n}}:\phi^{\mu,\nu}_\mp]=0\quad \text{and}\quad 
[\phi^{\mu\sqcup\nu}_\pm{\downarrow}_{\TS_{m,n}}:\phi^{\mu,\nu}_\pm]=1.$$
\end{enumerate}
\end{thm}

\begin{cor} \label{CABranch}
Let $\mu\in\DP(m)$, $\nu\in\DP(n)$. If $\lambda:=\mu\sqcup \nu$ is a strict partition then one of the following happens:
\begin{enumerate}
\item[{\rm (i)}] $\lambda$ and $(\mu,\nu)$ are odd and $[\psi^\lambda{\downarrow}_{\TA_{m,n}}:\psi^{\mu,\nu}]=1$; 
\item[{\rm (ii)}] $\lambda$ and $(\mu,\nu)$ are even and 
either\,\, $[\psi^\lambda_\pm{\downarrow}_{\TA_{m,n}}:\psi^{\mu,\nu}_\pm]=1$\, or\,\, $[\psi^\lambda_\pm{\downarrow}_{\TA_{m,n}}:\psi^{\mu,\nu}_\mp]=1$. 
\end{enumerate}
\end{cor}
\begin{proof}
Note that $(\mu,\nu)$ and $\lambda$ have the same parity. Now if they are both odd then $\phi^{\lambda}_\pm{\downarrow}_{\TA_{m+n}}=\psi^\lambda$ and $\phi^{\mu,\nu}_\pm{\downarrow}_{\TA_{m,n}}=\psi^{\mu,\nu}$, so by Theorem~\ref{TBranchSpin}(iv), we are in the case (i). If $(\mu,\nu)$ and $\lambda$ are both even, then $\phi^{\lambda}{\downarrow}_{\TA_{m+n}}=\psi^\lambda_++\psi^\lambda_-$ and $\phi^{\mu,\nu}{\downarrow}_{\TA_{m,n}}=\psi^{\mu,\nu}_++\psi^{\mu,\nu}_-$. Moreover, we have $f^\lambda_{\mu,\nu}=1$ by Lemma~\ref{LQFun}. Since $\ell(\lambda)=\ell(\mu)+\ell(\nu)$, it now follows from Theorem~\ref{TBranchSpin}(i) that $[\psi^\lambda_++\psi^\lambda_-){\downarrow}_{\TA_{m,n}}:\psi^{\mu,\nu}_\pm]=1$, and we are in the case (ii). 
\end{proof}

\begin{rem} 
{\rm 
Looking more carefully at the character values, one can resolve the uncertainty in part (ii) of Corollary~\ref{CABranch} and show that this case splits into two subcases:
\begin{enumerate}
\item[{\rm (iia)}] $\lambda,\,\mu,\, \nu$ are even, and\,  $[\psi^\lambda_\pm{\downarrow}_{\TA_{m,n}}:\psi^{\mu,\nu}_\pm]=1$. 
\item[{\rm (iib)}] $\lambda$ is even, $\mu,\,\nu$ are odd, and \,
$[\psi^\lambda_\pm{\downarrow}_{\TA_{m,n}}:\psi^{\mu,\nu}_\mp]=1$. 
\end{enumerate}
We are not going to need this information. 
}
\end{rem}

\begin{thm} 
\label{spin}
Let $\chi\in\Irr^-(\TA_n)$. If $\chi$ is not a basic spin character then there exist an integer\, $1\leq m<n$ and an irreducible character\, $\eta\in\Irr^-(\TA_{n-m,m})$ such that:
\begin{enumerate}
\item[{\rm (i)}] $m\neq n/2$; in particular the subgroup $\TA_{n-m,m}<\TS_n$ is self-normalizing; 
\item[{\rm (ii)}] $[\chi{\downarrow}_{\TA_{n-m,m}}:\eta]=1$;
\item[{\rm (iii)}] $\Q(\chi)=\Q(\eta)$;
\item[{\rm (iv)}] One of the following two conditions holds:
\begin{enumerate}
\item[{\rm (a)}] $\Q(\chi)=\Q$, $\chi$ extends to $\TS_n$, and $\eta$ extends to $\TS_{n-m,m}$;
\item[{\rm (b)}] $[\Q(\chi):\Q]=2$, 
$\chi^{\c G}=\chi^{\TS_{n-m,m}}$, $\eta^{\c G}=\eta^{\TS_{n-m,m}}$, $\chi$ does not extend to $\TS_n$, and $\eta$ does not extend to $\TS_{n-m,m}$. 
\end{enumerate}
\end{enumerate} 
In particular, $(\chi,\eta)$ satisfies the inductive Feit condition if $n \neq 6$.
\end{thm}
\begin{proof}
Let $\chi=\psi^\lambda_{(\pm)}$ for some $\lambda\in\DP(n)$, see Theorem~\ref{TIrrA}. In view of Remark~\ref{RBasicSpin}, we may assume that $\lambda\neq (n)$. 
In this case we can always write $\lambda=\mu\sqcup\nu$ for $\mu\in\DP(m)$ and $\nu\in\Par(n-m)$ such that $1<m<n$ and $m\neq n/2$. Note that $\lambda$ is even if and only if $(\mu,\nu)$ is even. 

Suppose first that $(\mu,\nu)$ is odd, so $\lambda$ is also odd. 
In this case $\chi=\psi^\lambda$. 
We take $\eta=\psi^{\mu,\nu}$. Then 
the condition (ii) follows from Corollary~\ref{CABranch}(i). For the conditions (iii) and (iv) it remains to note by Corollaries~\ref{CCharValMN},\ \ref{CMNCharVal} that $\Q(\chi)=\Q(\eta)=\Q$, and by Theorems~\ref{TIrrA},\,\ref{TMNA} that $\psi^\lambda$ extends to $\phi^\lambda$ and $\psi^{\mu,\nu}$ extends to $\phi^{\mu,\nu}$. 

Now let $(\mu,\nu)$ be even, so $\lambda$ is also even. 
In this case $\chi$ is of the form $\psi^\lambda_\pm$. 
Choose $\eta=\psi^{\mu,\nu}_+$ or $\psi^{\mu,\nu}_-$ appropriately so that $[\chi{\downarrow}_{\TA_{n-m,m}}:\eta]=1$ by Corollary~\ref{CABranch}(ii). This yields the condition 
(ii). The condition (iii) comes from Corollary~\ref{CMNCharVal}. 
For the condition (iv), note by Corollaries~\ref{CCharVal},\,\ref{CCharValMN} that $[\Q(\chi):\Q]=2$, $\chi^{\c G}=\{\phi^\lambda_+,\phi^\lambda_-\}=\chi^{\TS_{n-m,m}}$, and $\eta^{\c G}=\{\psi^{\mu,\nu}_+,\psi^{\mu,\nu}_-\}=\eta^{\TS_{n-m,m}}$. Moreover, by Theorems~\ref{TIrrA},\,\ref{TMNA}, $\psi^\lambda_\pm$ does not extend to $\TS_n$ and $\psi^{\mu,\nu}_\pm$ does not extend to $\TS_{n-m,m}$. 

For the last statement we can now apply Lemma~\ref{easy2} and Theorem \ref{easysituations}(a).
\end{proof}

Finally, we handle the basic spin characters $\chi=\psi^{(n)}_{(\pm)}$ of $\TA_n$.
 
\begin{thm}\label{basic-spin}
For any $n \geq 5$, $n \neq 6$, the basic spin characters $\chi=\psi^{(n)}_{(\pm)}$ of $\TA_n$ satisfy the inductive Feit condition.
\end{thm}

\begin{proof}
(a) First we consider the case $n = p$ is a prime. Let $c \in X:=\TA_n$ be the inverse image in $X$ of a $p$-cycle. Then 
$C:=\langle c \rangle$ is a Sylow $p$-subgroup of $X$, and so $U:=\NB_X(\langle c \rangle)$ satisfying Definition \ref{inductive}(i). 
Note that $U = (\ZB(X) \times C) \cdot C_{(p-1)/2}$, and $\NB_{\TS_n}(U) = (\ZB(X) \times C) \cdot C_{p}$.
Also, $\lambda=(n)$ is even, and 
$$\Q(\chi) = \Q(\chi(c)) = \Q\Bigl(\sqrt{(-1)^{(p-1)/2}p}\Bigr),~\phi^{(n)}\downarrow_{X}=\psi^{(n)}_+ + \psi^{(n)}_-,~\Q(\phi^{(n)})=\Q.$$
Furthermore, inducing faithful linear characters of $\ZB(X) \times C$ to $U$, we obtain two Galois-conjugate 
irreducible characters $\mu= \mu^+$ and $\mu^-$, both 
of degree $(p-1)/2$, with $\Q(\mu^\eps) = \Q(\chi)$ and $\mu^\eps$ agreeing with $\chi$ on $\ZB(X)$ for each $\eps=\pm$. Note that 
$\NB_{\TS_n}(U)$ fuses $\mu^+$ with $\mu^-$, so we have $\Gamma_\chi=\Gamma_\mu$. By Lemma \ref{easy2}(iii), 
$(\Gamma \times \c G)_\chi= (\Gamma \times \c G)_\mu$. Moreover, $\Aut(X)_\chi=\Inn(X)$ since $\TS_n$ fuses $\psi^{(n)}_+$ with 
$\psi^{(n)}_-$. Hence we are done by Theorem \ref{easysituations}(b).

\smallskip
(b) From now on we may assume that $n =ab$ with $a,b \in \Z_{\geq 2}$, $a \geq b$, and $(a,b) \neq (2,2)$. Then the full inverse image $U$ 
of $(\SSS_a \wr \SSS_b) \cap \AAA_n$ is a maximal subgroup of $X$ that satisfies Definition \ref{inductive}(i), and 
$\NB_{\TS_n}(U)$ is the full inverse image of $\SSS_a \wr \SSS_b$ in $\TS_n$.

Assume first that $n$ is odd. Then $\lambda=(n)$ is even, $U$ contains the inverse image $c$ in $\TA_n$ of an $n$-cycle, and 
$$\Q(\chi) = \Q(\chi(c)) = \Q\Bigl(\sqrt{(-1)^{(n-1)/2}n}\Bigr),~\phi^{(n)}\downarrow_{X}=\psi^{(n)}_+ + \psi^{(n)}_-,~\Q(\phi^{(n)})=\Q.$$
In particular, $\Aut(X)_\chi=\Inn(X)$. By Theorems 1 and 2 of \cite{KW}, $\phi^{(n)}$ restricts irreducibly to $\NB_{\TS_n}(U)$, and 
each restriction $\mu^\eps$ of $\psi^{(n)}_{\eps}$ to $U$, $\eps=\pm$, is also irreducible. It follows that $\NB_{\TS_n}(U)$ fuses 
$\mu=\mu^+$ with $\mu^-$, and hence $\Gamma_\chi=\Gamma_\mu$. Since $\mu=\chi\downarrow_U$ and $c \in U$, we also have
$$\Q(\chi)=\Q(\chi(c)) = \Q(\mu(c)) \subseteq \Q(\mu) \subseteq \Q(\chi),$$
and hence $\Q(\chi)=\Q(\mu)$. Now if $n$ is a square, then $\chi$ and $\mu$ are both rational. Suppose $n$ is not a square. Then 
$\Q(\chi)=\Q(\mu)$ has degree $2$ over $\Q$. Furthermore, $\c G$ fixes $\phi^{(n)}$ and so permutes the two constituents 
$\psi^{(n)}_+$ and $\psi^{(n)}_-$ of its restriction to $\TA_n$, showing that $\chi^{\c G}=\chi^\Gamma$. Restricting to $U$, we see
that $\c G$ also permutes the two constituents $\mu^+$ and $\mu^-$ of the restriction of $\phi^{(n)}$ to $\NB_{\TS_n}(U)$, showing 
that $\mu^{\c G}=\mu^\Gamma$. By Lemma \ref{easy2}(iii), $(\Gamma \times \c G)_\chi= (\Gamma \times \c G)_\mu$ in this case as well. 
Hence we are done by Theorem \ref{easysituations}(a).

Now assume that $n$ is even. Then $\lambda=(n)$ is odd, $\Q(\chi) = \Q$, and $\chi$ extends to the two characters 
$\phi^{(n)}_{\eps}$, $\eps = \pm$, of $\TS_n$.
By Theorems 1 and 2 of \cite{KW}, each $\phi^{(n)}_\eps$ restricts irreducibly to $\NB_{\TS_n}(U)$, and $\mu:=\chi\downarrow_U$ 
is also irreducible. It follows that $\Gamma_\chi=\Gamma_\mu$ and $\Q(\mu)=\Q$.
By Lemma \ref{easy2}(i) we have $(\Gamma \times \c G)_\chi= (\Gamma \times \c G)_\mu$, 
and hence we are done by Theorem \ref{easysituations}(a).
\end{proof}

Combining Theorems \ref{an-ord}, \ref{spin}, \ref{basic-spin} and Lemma \ref{69}, we arrive at Theorem C, which we restate

\begin{thm}\label{main-an}
For any $n \geq 5$, the simple group $\AAA_n$ satisfies the inductive Feit condition.
\end{thm}
 
\section{Simple groups of Lie type and proof of Theorem D}
 
\begin{thm}\label{sl22}
The simple groups $S=\SL_2(q)$ with $q=2^n$ and $n \geq 2$ satisfy the inductive Feit condition.
\end{thm} 
 
\begin{proof}
(a) Since $\SL_2(4) \cong \AAA_5$ was already handled in Theorem \ref{main-an},
we may assume $n \geq 3$, so that $X = S = \SL_2(q)$, $\Aut(X) = \langle X,\alpha \rangle \cong X \rtimes C_n$ for a field automorphism
$\alpha$, and let $\chi \in \Irr(X)$. 
We will use the notation for irreducible characters of $X$ as listed in \cite[Theorem 38.2]{Do}, with the further convention that 
the indices $i$ for $\chi_i$ of degree $q+1$ are viewed in $\Z/(q-1)\Z$, and the indices $j$ for $\theta_j$ of degree $q-1$ are viewed in $\Z/(q+1)\Z$.

First assume that $\chi(1)=1$, or $q$, or $q+1$ with $3|(q-1)$ and $\chi=\chi_{(q-1)/3}$, or 
$q-1$ with $3|(q+1)$ and $\chi=\theta_{(q+1)/3}$. Then $\chi$ is rational and $\Aut(X)$-invariant. Take $P \in \Syl_2(X)$
and $U=\NB_X(P)$. Also choose $\mu=1_U$ when $\chi=1_G$, and $\mu$ the unique character of degree $q-1$ otherwise.
In particular, $\mu$ is rational and stable under $\Gamma= \Aut(X)_U$; moreover, $[\chi_U,\mu]=1$. Hence $(X,\chi)$ 
satisfies the inductive Feit condition by Lemma \ref{easy2} and Theorem \ref{easysituations}(a).

\smallskip
(b) Next assume that $\chi=\chi_i$ with $i \not\equiv 0 \pmod{(q-1)/\gcd(3,q-1)}$. Then
$$\Q(\chi) = \Q(\rho^i+\rho^{-i}) = \Q(\cos\frac{2\pi i}{q-1}),$$
for a primitive $(q-1)^{\mathrm {th}}$-root of unity $\rho \in \C^\times$. Then we take $U=\NB_X(A) = A \rtimes \langle c \rangle$, the normalizer of the split torus 
$A := \langle a \rangle \cong C_{q-1}$, and $|c|=2$. We can choose $\alpha$ to be induced by the action of the field automorphism $x \mapsto x^2$ of
$\overline{\FF_q}$ on $X = \SL(\FF_q^2)$, so that $\alpha: a \mapsto a^2$, $c \mapsto c$. Note that
$$\Irr(U) = \{\lambda_t,1_U,\sgn \mid 1 \leq t \leq q/2-1 \},$$
where $\lambda_t: 1 \mapsto 2,~a^l \mapsto \rho^{tl}+\rho^{-tl},c \mapsto 0$ for $1 \leq l \leq q-2$, and $(1_A)^U = 1_U +\sgn$.
Then $\alpha$ normalizes $U$, and 
\begin{equation}\label{eq:sl221}
  \alpha: \chi_t \mapsto \chi_{2t},~\lambda_t \mapsto \lambda_{2t},
\end{equation}
if we again view the index $t$ of $\chi_t$ and $\lambda_t$ as in $\Z/(q-1)\Z$. On the other hand, if $\sigma \in \c G$ acts on $\Q(\rho)$ via 
$\rho \mapsto \rho^k$ for some $k$ (coprime to $q-1$), then   
\begin{equation}\label{eq:sl222}
  \sigma: \chi_t \mapsto \chi_{kt},~\lambda_t \mapsto \lambda_{kt},
\end{equation}
It is easy to check that
$$(\chi_i)_U = \sum^{q/2-1}_{t=1}\lambda_t + \lambda_i + 1_U,$$
in particular, $[\chi_U,\lambda_{2i}]=1$. Recall that $\chi_t = \chi_{t'}$ if and only if $t' \equiv \pm t \pmod{q-1}$.
Given \eqref{eq:sl221} and \eqref{eq:sl221}, we now see that the element 
$(\delta,\sigma) \in \Gamma \times \c G$, with $s \in \Z$ and $\delta \in \alpha^s U \subseteq \Gamma = \NB_{\Aut(X)}(U)$, fixes $\chi$
if and only if 
\begin{equation}\label{eq:sl223}
  q-1 \mbox{ divides }i(2^s k \pm 1).
\end{equation}    
Similarly, this element $(\delta,\sigma)$ fixes $\lambda_{2i}$
if and only if 
$$q-1 \mbox{ divides }2i(2^s k \pm 1),$$
which is equivalent to \eqref{eq:sl223} since $2|q$. Taking $\mu=\lambda_{2i}$, we have 
$(\Gamma \times \c G)_\chi=(\Gamma \times \c G)_\mu$, and hence we are done by Theorem \ref{easysituations}(a).

\smallskip
(c) Now assume that $\chi=\theta_j$ with $j \not\equiv 0 \pmod{(q+1)/\gcd(3,q+1)}$. Then
$$\Q(\chi) = \Q(\eta^j+\eta^{-j}) = \Q(\cos\frac{2\pi j}{q+1}),$$
for a primitive $(q+1)^{\mathrm {th}}$-root of unity $\eta \in \C^\times$. It is convenient to view $X$ as $\SU_2(q) = \SU(V)$, where 
$V = \FF_{q^2}^2$ is endowed with a standard Hermitian form.
Then we take $U=\NB_X(B) = B \rtimes \langle c \rangle$, the normalizer of the diagonal torus 
$B := \langle b \rangle \cong C_{q+1}$, and $|c|=2$. We can choose $\alpha$ to be induced by the action of the field automorphism $x \mapsto x^2$ of
$\overline{\FF_q}$ on $\SU(V)$, so that $\alpha: b \mapsto b^2$, $c \mapsto c$. Note that
$$\Irr(U) = \{\lambda_t,1_U,\sgn \mid 1 \leq t \leq q/2 \},$$
where $\lambda_t: 1 \mapsto 2,~b^l \mapsto \eta^{tl}+\eta^{-tl},c \mapsto 0$ for $1 \leq l \leq q$, and $(1_B)^U = 1_U +\sgn$.
Then $\alpha$ normalizes $U$, and 
\begin{equation}\label{eq:sl224}
  \alpha: \theta_t \mapsto \theta_{2t},~\lambda_t \mapsto \lambda_{2t},
\end{equation}
if we again view the index $t$ of $\theta_t$ and $\lambda_t$ as in $\Z/(q+1)\Z$. On the other hand, if $\sigma \in \c G$ acts on $\Q(\eta)$ via 
$\eta \mapsto \eta^k$ for some $k$ (coprime to $q+1$), then   
\begin{equation}\label{eq:sl225}
  \sigma: \chi_t \mapsto \chi_{kt},~\lambda_t \mapsto \lambda_{kt},
\end{equation}
It is easy to check that
$$(\theta_j)_U = \sum^{q/2}_{t=1,~t \neq j}\lambda_t + \sgn.$$
In particular, $[\chi_U,\lambda_{2j}]=1$ since $j$ is not divisible by $(q+1)/\gcd(3,q+1)$. Recall that $\theta_t = \theta_{t'}$ if and only if $t' \equiv \pm t \pmod{q+1}$.
Given \eqref{eq:sl224} and \eqref{eq:sl225}, we now see that the element 
$(\delta,\sigma) \in \Gamma \times \c G$, with $s \in \Z$ and $\delta \in \alpha^s U \subseteq \Gamma = \NB_{\Aut(X)}(U)$, fixes $\chi$
if and only if 
\begin{equation}\label{eq:sl226}
  q+1 \mbox{ divides }j(2^s k \pm 1).
\end{equation}    
Similarly, this element $(\delta,\sigma)$ fixes $\lambda_{2j}$
if and only if 
$$q+1 \mbox{ divides }2j(2^s k \pm 1),$$
which is equivalent to \eqref{eq:sl226} since $2|q$. Taking $\mu=\lambda_{2j}$, we have 
$(\Gamma \times \c G)_\chi=(\Gamma \times \c G)_\mu$, and hence we are done by Theorem \ref{easysituations}(a).
\end{proof} 

Let $p>2$ be any prime and $q=p^n \geq 5$. Fix a generator $\rho$ of $\FF_{q^2}^\times$, and a 
primitive $(q^2-1)^{\mathrm {th}}$ root of unity $\trho \in \C^\times$. Also set
\begin{equation}\label{eq:sl210}
  \eps = \rho^{(q-1)/2},~\teps = \trho^{(q-1)/2},~\varep = \rho^{(q+1)/2},~\tvarep = \trho^{(q+1)/2}.
\end{equation}
  
Suppose $q \equiv 1 \pmod{4}$. Then for $X = \SU_2(q)$ and $Y = \GU_2(q)$ we have  
$$Y=L \times \OB_{2'}(\ZB(Y)),\mbox{ where }L= \{ y \in Y \mid \det{y}= \pm 1\}.$$
We also consider the subgroup 
\begin{equation}\label{eq:sl211}
  \tL = \left\{x, \eps^{(q+1)/2}y \mid x \in X,~y \in L \smallsetminus X\right\}
\end{equation}  
of $\SL_2(q^2)$ which is isoclinic to $L$ (note $\tL$ is a group because $\eps^{q+1}=-1$), and in which every involution is central. 
In particular, $\ZB(X)=\ZB(L)=\ZB(\tL) = \langle z \rangle$, and $\tL/\ZB(X) \cong \PGU_2(q)$. 

\begin{lem}\label{sl21-sub1}
Suppose $q \equiv 1 \pmod{4}$ and $q > 9$. In the above notation, the following hold.
\begin{enumerate}[\rm(i)]
\item Let $\langle a \rangle$ be a (unique up to conjugacy) $C_{q-1}$-subgroup of $X=\SU_2(q)$. Then there are elements $\ta \in \tL$ of order $2(q-1)$ and 
$u \in X$ such that $U=\langle a,u \rangle$ is maximal in $X$, $\ta^{q-1}=u^2=z$, $u\ta u^{-1}=\ta^{-1}$, $\ta^2=a^{(q+1)/2}$,
and $\NB_{\tL}(U) = \langle \ta, u \rangle$.
\item Let $\langle b \rangle$ be a (unique up to conjugacy) $C_{q+1}$-subgroup of $X=\SU_2(q)$. Then there are elements $\tb \in \tL$ of order $2(q+1)$ and 
$u \in X$ such that $U=\langle b,u \rangle$ is maximal in $X$, $\tb^{q+1}=u^2=z$, $u\tb u^{-1}=\tb^{-1}$, $\tb^2=b^{(q+5)/2}$,
and $\NB_{\tL}(U) = \langle \tb, u \rangle$.
\end{enumerate}
\end{lem}

\begin{proof}
(i) Consider the space $V = \langle e,f \rangle_{\FF_{q^2}}$ with the Hermitian product $e \circ e = f \circ f = 0$ and $e \circ f=1$, and take $Y = \GU(V)$. Then
the elements 
$$\ta = \eps^{(q+1)/2}\begin{pmatrix}\varep & 0\\0 & -\varep^{-1}\end{pmatrix} \in \tL,~
     a=\begin{pmatrix}\varep^2 & 0\\0 & \varep^{-2}\end{pmatrix} \in X,~u =  \begin{pmatrix} 0 & \varep\\-\varep^{-1} & 0 \end{pmatrix} \in X$$ 
satisfy the required relations. Note that $U=\langle a,u \rangle$ is maximal in $X$ by \cite[Table 8.1]{BHR}, so
$\NB_X(U)=U$ and $\NB_{\tL}(U) = U \cdot 2  = \langle \ta,u \rangle$.

\smallskip
(ii) Consider the space $V = \langle e,f \rangle_{\FF_{q^2}}$ with the Hermitian product $e \circ e = f \circ f = 1$ and $e \circ f=0$, and take $Y = \GU(V)$. Then
the elements 
$$\tb = \eps^{(q+1)/2}\begin{pmatrix}\eps^2 & 0\\0 & -\eps^{-2}\end{pmatrix} \in \tL,~
     b = \begin{pmatrix} \eps^2 & 0\\0 & \eps^{-2} \end{pmatrix} \in X,~u =  \begin{pmatrix} 0 & 1\\-1 & 0 \end{pmatrix} \in X$$ 
satisfy the required relations. Note that $U=\langle b,u \rangle$ is maximal in $X$ by \cite[Table 8.1]{BHR}, so
$\NB_X(U)=U$ and $\NB_{\tL}(U) = U \cdot 2  = \langle \tb,u \rangle$.
\end{proof}

Now suppose that $q \equiv 3 \pmod{4}$. Then for $X = \SL_2(q)$ and $Y = \GL_2(q)$ we have  
$$Y=M \times \OB_{2'}(\ZB(Y)),\mbox{ where }M= \{ y \in Y \mid \det{y}= \pm 1\}.$$
We also consider the subgroup 
\begin{equation}\label{eq:sl212}
  \tM = \left\{x, \varep^{(q-1)/2}y \mid x \in X,~y \in M \smallsetminus X\right\}
\end{equation}  
of $\SL_2(q^2)$ which is isoclinic to $M$ (note $\tM$ is a group because $\varep^{q-1}=-1$), and in which every involution is central. 
In particular, $\ZB(X)=\ZB(M)=\ZB(\tM) = \langle z \rangle$, and 
$\tM/\ZB(X) \cong \PGL_2(q)$. 

\begin{lem}\label{sl21-sub2}
Suppose $q \equiv 3 \pmod{4}$ and $q > 3$. In the above notation, the following hold.
\begin{enumerate}[\rm(i)]
\item Let $\langle a \rangle$ be a (unique up to conjugacy) $C_{q-1}$-subgroup of $X=\SL_2(q)$. Then there are elements $\ta \in \tM$ of order $2(q-1)$ and 
$u \in X$ such that $U=\langle a,u \rangle$ is self-normalizing in $X$, $\ta^{q-1}=u^2=z$, $u\ta u^{-1}=\ta^{-1}$, $\ta^2=a^{(q+3)/2}$,
and $\NB_{\tM}(U) = \langle \ta, u \rangle$.
\item Let $\langle b \rangle$ be a (unique up to conjugacy) $C_{q+1}$-subgroup of $X=\SL_2(q)$. Then there are elements $\tb \in \tM$ of order $2(q+1)$ and 
$u \in X$ such that $U=\langle b,u \rangle$ is self-normalizing in $X$, $\tb^{q+1}=u^2=z$, $u\tb u^{-1}=\tb^{-1}$, $\tb^2=b^{(q+3)/2}$,
and $\NB_{\tM}(U) = \langle \tb, u \rangle$.
\end{enumerate}
\end{lem}

\begin{proof}
(i) Consider the space $V = \langle e,f \rangle_{\FF_{q}}$ and take $Y = \GL(V)$. Then
the elements 
$$\ta = \varep^{(q-1)/2}\begin{pmatrix}\varep^2 & 0\\0 & -\varep^{-2}\end{pmatrix} \in \tM,~
     a = \begin{pmatrix} \varep^2 & 0\\0 & \varep^{-2} \end{pmatrix} \in X,~u =  \begin{pmatrix} 0 & 1\\-1 & 0 \end{pmatrix} \in X$$ 
satisfy the required relations. Note that $U=\langle a,u \rangle$ is maximal in $X$ when $q>11$ by \cite[Table 8.1]{BHR}, whereas $\langle a \rangle$ is the unique 
$C_{q-1}$-subgroup of $U$ when $q=7,11$ \cite{Atlas}. Hence $U$ is self-normalizing in $X$,
and $\NB_{\tM}(U) = U \cdot 2  = \langle \ta,u \rangle$.

\smallskip
(ii) Consider $V=\FF_{q^2}$ as a $2$-dimensional $\FF_q$-module for $Y$. Then the Frobenius map $v:x \mapsto x^q$ yields 
an involution in $Y$ with determinant $-1$ that interchanges the eigenspaces for the eigenvalues $\eps$ and $\eps^q$ of 
the multiplication by $\eps$ on $V$, viewed as an element $m$ of order $2(q+1)$ and determinant $\eps^{q+1}=-1$ in $Y$. So 
the element $\tb=\varep^{(q-1)/2}m$ belongs to $\tM \smallsetminus X$, still of order $2(q+1)$. Taking $b=m^2$, we have 
$\tb^2=zb = b^{(q+3)/2}$. As mentioned before, $vmv=vmv^{-1}=m^q$, whence $(vm)^2=m^{q+1}=z$ and $\det{vm}=1$. 
It follows that $u=vm \in X$ satisfies the relations $u^2=z=\tb^{q+1}$ and $u\tb u^{-1}=\tb^{-1}$. Note that $U=\langle b,u \rangle$ is maximal in $X$ 
when $q>7$ by \cite[Table 8.1]{BHR}, whereas $\langle b \rangle$ is the unique 
$C_{q+1}$-subgroup of $U$ when $q=7$ \cite{Atlas}. Hence $U$ is self-normalizing in $X$, and $\NB_{\tM}(U) = U \cdot 2  = \langle \tb,u \rangle$.
\end{proof}

\begin{thm}\label{sl21}
Let $p > 2$ be any prime, $2 \nmid n$, and $q=p^n \geq 5$. Then the simple group $S=\PSL_2(q)$ satisfies the inductive Feit condition.
\end{thm} 
 
\begin{proof}
(a) Since $\PSL_2(5) \cong \AAA_5$, using Theorem \ref{main-an} we will assume $q =7$ or $q\geq 11$, so that $X = \SL_2(q)$, and 
$\Aut(X) = \langle S,\tau,\alpha\rangle$, where $\langle S,\tau\rangle \cong \PGL_2(q)$ and $\alpha$ is induced by the action of
the field automorphims $x \mapsto x^p$ of $\overline{\FF_q}$ on $\SL_2(\FF_q^2)$. Our subsequent proof will use the oddness of $n$ 
only in the treatment of the characters $\chi \in \Irr(X)$ of degree $q+1$.
We will use the notation for $\chi \in \Irr(X)$ as listed in \cite[Theorem 38.1]{Do}, with the further convention that 
the indices $i$ for $\chi_i$ of degree $q+1$ are viewed in $\Z/(q-1)\Z$, and the indices $j$ for $\theta_j$ of degree $q-1$ are viewed in $\Z/(q+1)\Z$.

First assume that $\chi(1)=1$ or $q$. Then $\chi$ is rational and $\Aut(X)$-invariant. Take $P \in \Syl_p(X)$
and $U=\NB_X(P)$, which is $\langle \tau,\alpha\rangle$-invariant. 
Also choose $\mu=1_U$, whence $\mu$ is rational and stable under $\Gamma=\Aut(X)_U$; moreover, $[\chi_U,\mu]=1$. 
Hence $(X,\chi)$ satisfies the inductive Feit condition by Lemma \ref{easy2} and Theorem \ref{easysituations}(a).

Next assume that $\chi(1) = (q \pm 1)/2$. Then $\Q(\chi) = \Q(\sqrt{(-1)^{(q-1)/2}q})$, and 
$\Aut(X)_\chi=\langle S,\alpha\rangle$. Taking the same $U=\NB_X(P)$, we note that $U$ has $q-1$ linear characters and $4$ irreducible characters
$\mu_0^\pm$, $\mu_1^\pm$ of degree $(q-1)/2$, each having $\Q(\chi)$ as field of values. Furthermore, $\tau$ fuses $\mu_0^+$ with
$\mu_0^-$, and $\mu_1^+$ with $\mu_1^-$, whereas $\alpha$ fixes each of $\mu_0^\pm$, $\mu_1^\pm$.
Moreover, $\ZB(X) \cong C_2$ is contained in the kernel of $\mu_0^\pm$, but not of $\mu_1^\pm$. Now there is a unique $\mu \in \{\mu_0^\pm,\mu_1^\pm\}$
such that $[\chi_U,\mu]=1$, and we have $\Gamma_\chi=\Gamma_\mu$. If in addition $2|n$, then $\chi$ and $\mu$ are rational, and so 
we are done by Lemma \ref{easy2}(i) and Theorem \ref{easysituations}(a). If $2\nmid n$, then $\chi$ and $\mu$ satisfy the hypothesis of Lemma 
\ref{easy2}(iii), and so we are done.

\smallskip
(b) The rest of the proof is to analyze the characters $\chi$ of degree $q \pm 1$. Here we assume that 
$$q \equiv 1 \pmod{4},$$
and view $X \cong \SU_2(q)$ as the derived subgroup of $Y:=\GU_2(q)$. Modify $\alpha$ to be induced by the action of the field automorphism
$x \mapsto x^p$ on $\SU(\FF_{q^2}^2)$. We keep the notation \eqref{eq:sl210}, and follow the notation of \cite{E} for conjugacy classes and irreducible characters in $Y$. 
We also consider the subgroup $\tL$ defined in \eqref{eq:sl211}; in particular, 
$\tL/\ZB(X) \cong \PGU_2(q)$ induces the subgroup $\langle S,\tau \rangle$ of $\Aut(X)$. 

\smallskip
(b1) First we consider the case $\chi=\theta_j$ (in the notation of \cite{Do}) of degree $q-1$, $1 \leq j \leq (q-1)/2$. This character takes value 
$(-1)^j(q-1)$ at the central involution $z$ of $X$ which belongs to class $C^{((q+1)/2)}_1$ in $Y$ \cite{E}. Furthermore, it takes
value $-1$ at $c \in C^{(0)}_2$, $0$ at $a^l \in C_4^{(l(q+1))}$ when $(q-1) \nmid 2l$, and $-(\teps^{2jm}+\teps^{-2jm})$ at 
$b^m \in C_3^{(m,-m)}$ when $(q+1) \nmid 2m$. Hence, $\chi$ extends to the character $\chi^{(t_0,u_0)}_{q-1}$ of $Y$ \cite{E}, where
$$t_0=j(q+3)/4,~u_0=j(q-1)/4,$$
whose kernel contains 
$$Z_1=\OB_{2'}(\ZB(Y)).$$ 
Thus we can view $\chi^{(t_0,u_0)}_{q-1}$ as a character of $Y/Z_1 \cong L$. 

Now, if $2|j$ then $\chi^{(t_0,u_0)}_{q-1}$ is real-valued on $L$ and trivial at $z \in \ZB(L)$. Hence we can view it (by inflation) as a real character
$\tchi$ of the group $\tL$ defined in \eqref{eq:sl211}.
However, if $2 \nmid j$, then $\chi^{(t_0,u_0)}_{q-1}$ takes values in $\RR \sqrt{-1}$ for any $y \in L$ with $\det{y}=-1$, and is faithful on $L$.
In this case, if $\Phi$ denotes a representation affording $\chi^{(t_0,u_0)}_{q-1}$, then 
$$\left\{\Phi(x), \sqrt{-1}\Phi(y) \mid x \in X,y \in L \smallsetminus X\right\}$$
(which is a group because $\Phi(z)=-\mathrm{Id}$) defines a representation of $\tL$, whose corresponding character $\tchi$ is real-valued. 
In  both cases we have 
$$\Q(\chi) = \Q(\teps^{2j}+\teps^{-2j}) = \Q(\cos \frac{2\pi j}{q+1}),~~\Q(\tchi)=\Q(\teps^j+\teps^{-j}) = \Q(\cos \frac{\pi j}{q+1}).$$
Note that $\chi$ is $\tau$-invariant; furthermore, $\alpha$ acts on the subset $\{\theta_t \mid (q+1) \nmid 2t \}$ of $\Irr(X)$ via 
$$\theta_{t} \mapsto \theta_{tp}.$$
 
On the local side, we take $U = \langle b,u \rangle$, 
the normalizer of the diagonal torus in $X$ defined in Lemma \ref{sl21-sub1}(ii), which is self-normalizing and 
$\langle \tau,\alpha \rangle$-invariant. Then we embed $U$ in the subgroup $\tUb = \langle \tb,u \rangle < \tL$ also defined in Lemma \ref{sl21-sub1}(ii)
and may assume that $\tau$ is induced by the conjugation by $\tilde b$.
Let $\lambda \in \Irr(\langle b \rangle)$ and 
and $\tilde\lambda \in \Irr(\langle \tb \rangle)$ be such that
$$\lambda: b \mapsto \teps^2,~\tilde\lambda: \tb \mapsto \teps.$$
Then the character $\mu=(\lambda^{j})^U$ is irreducible (as $(q+1) \nmid 2j$), and sends $b^m$ to $-(\teps^{2jm}+\teps^{-2jm})$ when 
$(q+1) \nmid 2m$; in particular,
$$\Q(\mu) = \Q(\teps^{2j}+\teps^{-2j}) =\QQ(\chi).$$
Since $\mu$ extends to the character $(\tilde\lambda^{(j(q+5)/2})^{\tUb}$ of $\tUb$, $\mu$ is $\tau$-invariant. Furthermore, $\alpha$ acts on (a subset of) 
$\Irr(U)$ via 
$$(\lambda^t)^U \mapsto (\lambda^{tp})^U.$$ 
Thus, for $\Gamma=\Aut(X)_U$, we have shown that the element 
$(\delta,\sigma) \in \Gamma \times \c G$, where $\delta \in \alpha^s \langle U,\tau \rangle$ with $s \in \ZZ$ and $\sigma$ induces 
the Galois automorphism $\teps^2 \mapsto \teps^{2k}$ with $k$ coprime to $q+1$, fixes $\chi$, respectively $\mu$,
if and only if $q+1$ divides $j(p^s k \pm 1)$. As a consequence,
$$(\Gamma \times \c G)_\chi=(\Gamma \times \c G)_\mu.$$
Furthermore, $\chi_U + (\lambda^{-j})^U$ takes value $0$ at any $b^m$ with $(q+1) \nmid 2m$ and at $y$, and $(-1)^j(q+1)$ at $z$, hence it is equal to
$\beta^U$, where $\beta \in \Irr(\langle z \rangle)$ sends $z$ to $(-1)^j$. Note that
$$\beta^{\langle b \rangle} = \sum_{0 \leq t \leq q,~2|(t-j)}\lambda^t$$
contains both $\lambda^j$ and $\lambda^{-j}$. It follows that $[\chi_U,\mu]=1$, and so $(X,\chi)$ satisfies the inductive Feit condition by 
Theorem \ref{easysituations}(a).

\smallskip
(b2) Now we consider the case $\chi=\chi_j$ (in the notation of \cite{Do}) of degree $q+1$, $1 \leq j \leq (q-3)/2$, and assume in addition that 
\begin{equation}\label{eq:sl213}
  2 \nmid n.
\end{equation}  
This character takes value 
$(-1)^j(q-1)$ at $z$, $1$ at $c \in C^{(0)}_2$, $\tvarep^{2jl}+\tvarep^{-2jl}$ at $a^l \in C_4^{(l(q+1))}$ when $(q-1) \nmid 2l$, and $0$ at 
$b^m \in C_3^{(m,-m)}$ when $(q+1) \nmid 2m$. Hence, $\chi$ extends to the character $\chi^{(t_0)}_{q+1}$ of $Y$ \cite{E}, where
$$t_0 = \left\{ \begin{array}{ll}j(q+1)/2, & 2|j,\\ j(q+1)/2 + (q^2-1)/4, &2 \nmid j, \end{array}\right.$$
whose kernel again contains $Z_1$. Thus we can view $\chi^{(t_0)}_{q+1}$ as a character of $Y/Z_1 \cong L$. 

Now, if $2|j$ then $\chi^{(t_0)}_{q+1}$ is real-valued on $L$ and trivial at $z \in \ZB(L)$. Hence we can view it (by inflation) as a real character
$\chi^*$ of the group $\tL$ defined in \eqref{eq:sl211}.
However, if $2 \nmid j$, then $\chi^{(t_0)}_{q+1}$ takes values in $\RR \sqrt{-1}$ for any $y \in L$ with $\det{y}=-1$, and is faithful on $L$.
In this case, if $\Psi$ denotes a representation affording $\chi^{(t_0)}_{q+1}$, then 
$$\left\{\Psi(x), \sqrt{-1}\Psi(y) \mid x \in X,y \in L \smallsetminus X\right\}$$
(which is again a group because $\Psi(z)=-\mathrm{Id}$) defines a representation of $\tL$, whose corresponding character $\chi^*$ is real-valued. 
In  both cases we have 
\begin{equation}\label{eq:sl214}
  \Q(\chi) = \Q(\tvarep^{2j}+\tvarep^{-2j}) = \Q(\cos \frac{2\pi j}{q-1}),~~\Q(\chi^*)=\Q(\tvarep^j+\tvarep^{-j}) = \Q(\cos \frac{\pi j}{q-1}).
\end{equation}  
Note that $\chi$ is $\tau$-invariant; furthermore, $\alpha$ acts on the subset $\{\chi_t \mid (q-1) \nmid 2t \}$ of $\Irr(X)$ via 
$$\chi_{t} \mapsto \chi_{tp}.$$
 
On the local side, we take $U = \langle a,u \rangle$ as defined in Lemma \ref{sl21-sub1}(i),   
which is self-normalizing and $\langle \tau,\alpha \rangle$-invariant. Then we embed $U$ in the subgroup
$\tUb = \langle \ta,u \rangle < \tL$ as defined in Lemma \ref{sl21-sub1}(i)
and may assume that $\tau$ is induced by the conjugation by $\tilde a$. Let $\xi \in \Irr(\langle a \rangle)$ and 
$\tilde\xi \in \Irr(\langle \ta \rangle)$ be such that
$$\xi: a \mapsto \tvarep^2,~\tilde\xi: \ta \mapsto \varep.$$
Then the character $\mu=(\xi^{j})^U$ is irreducible (as $(q-1) \nmid 2j$), and sends $a^l$ to $\tvarep^{2jl}+\tvarep^{-2jl}$ when 
$(q-1) \nmid 2l$; in particular,
\begin{equation}\label{eq:sl215}
  \Q(\mu) = \Q(\tvarep^{2j}+\tvarep^{-2j}) =\QQ(\chi).
\end{equation}  
Note that $\mu$ extends to the character $\mu^*=(\tilde\xi^{j(q+1)/2})^{\tUb}$ of $\tUb$ with
\begin{equation}\label{eq:sl216}
  \Q(\mu^\ast) = \Q(\tvarep^{j(q+1)/2}+\tvarep^{-j(q+1)/2}) =\Q(\cos \frac{\pi j(q+1)/2}{q-1})= \QQ(\chi^*)
\end{equation}  
(since $(q+1)/2$ is coprime to $q-1$); in particular, 
$\mu$ is $\tau$-invariant. Furthermore, $\alpha$ acts on (a subset of) $\Irr(U)$ via 
$$(\xi^t)^U \mapsto (\xi^{tp})^U.$$ 
Thus, for $\Gamma=\Aut(X)_U$, we have shown  that the element 
$(\delta,\sigma) \in \Gamma \times \c G$, where $\delta \in \alpha^s \langle U,\tau \rangle$ with $s \in \ZZ$ and $\sigma$ induces 
the Galois automorphism $\tvarep^2 \mapsto \tvarep^{2k}$ with $k$ coprime to $q+1$, fixes $\chi$, respectively $\mu$,
if and only if $q-1$ divides $j(p^s k \pm 1)$. As a consequence,
$$(\Gamma \times \c G)_\chi=(\Gamma \times \c G)_\mu.$$
We also note that the restrictions of $\chi^*$ and $\mu^*$ to $\CB_{\tUb}(X) = \langle z \rangle$ are both multiples of 
the linear character $\beta$ sending $z$ to $(-1)^j$. 
(But unlike the case of (b1), $[\chi_U,\mu]=3$; in fact, $\chi_U$ does not have multiplicity-one irreducible constituent with $\Q(\chi)$ as its 
field of values.)

So far we have not used the assumption \eqref{eq:sl213}. Now we invoke  it and define $\Sigma = \tL \rtimes \langle \alpha^2 \rangle$ which
contains $X$ as a normal subgroup and induces the entire $\Aut(X)$ while acting on $X$. Then $\Delta = \NB_\Sigma(U) = \tUb \rtimes \langle \alpha^2 \rangle$
induces $\Gamma_{\chi^{\c G}} = \Gamma$. Since $[\Sigma:\tL]=n=[\Delta:\tUb]$ is odd, we have $\OB^{2'}(\Sigma_\chi)=\tL$ and 
$\OB^{2'}(\Delta_\mu)=\tUb$. If $\nu$ denotes the unique linear character of order $2$ of $\tL$, then $\chi^*$ and $\nu\chi^*$ are the $2$ extensions of 
$\chi$ to $\tL$. Now, the odd-order group $\Sigma_\chi/\tL$ acts on the set $\{\chi^*,\nu\chi^*\}$, hence trivially. We have shown that $\chi^*$ is a 
$\Sigma_\chi$-invariant real extension of $\chi$ to $\OB^{2'}(\Sigma_\chi)$, and similarly $\mu^*$ is a 
$\Delta_\mu$-invariant real extension of $\mu$ to $\OB^{2'}(\Delta_\mu)$. Next, consider any $(\delta,\sigma) \in (\Gamma \times \c G)_\chi$, 
where $\delta \in \alpha^{2s} \langle U,\tau \rangle$ with $s \in \ZZ$. Then $(\delta,\sigma)$ acts on the set
$\{\chi^*,\nu\chi^*\}$. As $2 \nmid n$, this action is the same as the action of $(\delta^n,\sigma^n)$, which is turn is the same as of 
$\sigma^n$ (because $\delta^n \in \tUb$ fixes $\chi^*$). Using \eqref{eq:sl214} and the fact that $\sigma^n$ fixes $\chi$, we see that
$\sigma^n$ can be viewed as an element $\sigma^* \in {\rm Gal}(\Q(\chi^*)/\Q(\chi)) \leq C_2$, and $\sigma^n$ sends $\chi^*$ to 
$\chi^*$ if $\sigma^* = \mathrm{Id}$ and to $\nu\chi^*$ if $\sigma^* \neq \mathrm{Id}$. Using \eqref{eq:sl215} and \eqref{eq:sl216}, and applying the same arguments 
to $\mu^*$ (after identifying $\Delta_\mu/U = \Sigma_\chi/X$), we see that  $(\delta,\sigma)$ acts on the set
$\{\mu^*,\nu\mu^*\}$, with the same action as of $\sigma^n$, and $\sigma^n$ sends $\mu^*$ to 
$\mu^*$ if $\sigma^* = \mathrm{Id}$ and to $\nu\mu^*$ if $\sigma^* \neq \mathrm{Id}$. It follows from Theorem \ref{easysituations}(e) that 
$(X,\chi)$ satisfies the inductive Feit condition.

\smallskip
(c) Here we assume that 
$$q \equiv 3 \pmod{4},$$
and view $X \cong \SL_2(q)$ as the derived subgroup of $Y:=\GL_2(q)$. Choose $\alpha$ to be induced by the action of the field automorphism
$x \mapsto x^p$ on $\SL(\FF_{q}^2)$. We keep the notation \eqref{eq:sl210}, and follow the notation of \cite{E} for conjugacy classes and irreducible characters in $Y$
(after applying the Ennola duality $q \mapsto -q$).
We also consider the subgroup $\tM$ defined in \eqref{eq:sl212}; in particular, 
$\tM/\ZB(X) \cong \PGL_2(q)$ induces the subgroup $\langle S,\tau \rangle$ of $\Aut(X)$. 

\smallskip
(c1) First we consider the case $\chi=\theta_j$ (in the notation of \cite{Do}) of degree $q-1$, $1 \leq j \leq (q-1)/2$. This character takes value 
$(-1)^j(q-1)$ at the central involution $z$ of $X$ which belongs to class $C^{((q-1)/2)}_1$ in $Y$. Furthermore, it takes
value $-1$ at $c \in C^{(0)}_2$, $0$ at $a^l \in C_3^{(l,-l})$ when $(q-1) \nmid 2l$, and $-(\teps^{2jm}+\teps^{-2jm})$ at 
$b^m \in C_3^{((q-1)m)}$ when $(q+1) \nmid 2m$. Hence, $\chi$ extends to the character $\chi^{(t_0)}_{q-1}$ of $Y$, where
$$t_0 = \left\{ \begin{array}{ll}j(q-1)/2, & 2|j,\\ j(q-1)/2 + (q^2-1)/4, &2 \nmid j, \end{array}\right.$$
whose kernel contains 
$$Z_1=\OB_{2'}(\ZB(Y)).$$ 
Thus we can view $\chi^{(t_0)}_{q-1}$ as a character of $Y/Z_1 \cong M$. 

Now, if $2|j$ then $\chi^{(t_0)}_{q-1}$ is real-valued on $M$ and trivial at $z \in \ZB(M)$. Hence we can view it (by inflation) as a real character
$\tchi$ of the group $\tM$ defined in \eqref{eq:sl212}.
However, if $2 \nmid j$, then $\chi^{(t_0)}_{q-1}$ takes values in $\RR \sqrt{-1}$ for any $y \in M$ with $\det{y}=-1$, and is faithful on $M$.
In this case, if $\Phi$ denotes a representation affording $\chi^{(t_0)}_{q-1}$, then 
$$\left\{\Phi(x), \sqrt{-1}\Phi(y) \mid x \in X,y \in M\smallsetminus X\right\}$$
(which is a group because $\Phi(z)=-\mathrm{Id}$) defines a representation of $\tM$, whose corresponding character $\tchi$ is real-valued. 
In  both cases we have 
$$\Q(\chi) = \Q(\teps^{2j}+\teps^{-2j}) = \Q(\cos \frac{2\pi j}{q+1}),~~\Q(\tchi)=\Q(\teps^j+\teps^{-j}) = \Q(\cos \frac{\pi j}{q+1}).$$
Note that $\chi$ is $\tau$-invariant; furthermore, $\alpha$ acts on the subset $\{\theta_t \mid (q+1) \nmid 2t \}$ of $\Irr(X)$ via 
$$\theta_{t} \mapsto \theta_{tp}.$$
 
On the local side, we take $U = \langle b,u \rangle$, 
as defined in Lemma \ref{sl21-sub2}(ii), which is self-normalizing and 
$\langle \tau,\alpha \rangle$-invariant. Then we embed $U$ in the subgroup $\tUb = \langle \tb,u \rangle < \tL$ also defined in Lemma \ref{sl21-sub2}(ii)
and may assume that $\tau$ is induced by the conjugation by $\tilde b$.
Let $\lambda \in \Irr(\langle b \rangle)$ and 
and $\tilde\lambda \in \Irr(\langle \tb \rangle)$ be such that
$$\lambda: b \mapsto \teps^2,~\tilde\lambda: \tb \mapsto \teps.$$
Then the character $\mu=(\lambda^{j})^U$ is irreducible (as $(q+1) \nmid 2j$), and sends $b^m$ to $-(\teps^{2jm}+\teps^{-2jm})$ when 
$(q+1) \nmid 2m$; in particular,
$$\Q(\mu) = \Q(\teps^{2j}+\teps^{-2j}) =\QQ(\chi).$$
Since $\mu$ extends to the character $(\tilde\lambda^{(j(q+3)/2})^{\tUb}$ of $\tUb$, $\mu$ is $\tau$-invariant. Furthermore, $\alpha$ acts on (a subset of) 
$\Irr(U)$ via 
$$(\lambda^t)^U \mapsto (\lambda^{tp})^U.$$ 
Thus, for $\Gamma=\Aut(X)_U$, we have shown that the element 
$(\delta,\sigma) \in \Gamma \times \c G$, where $\delta \in \alpha^s \langle U,\tau \rangle$ with $s \in \ZZ$ and $\sigma$ induces 
the Galois automorphism $\teps^2 \mapsto \teps^{2k}$ with $k$ coprime to $q+1$, fixes $\chi$, respectively $\mu$,
if and only if $q+1$ divides $j(p^s k \pm 1)$. As a consequence,
$$(\Gamma \times \c G)_\chi=(\Gamma \times \c G)_\mu.$$
Furthermore, $\chi_U + (\lambda^{-j})^U$ takes value $0$ at any $b^m$ with $(q+1) \nmid 2m$ and at $y$, and $(-1)^j(q+1)$ at $z$, hence it is equal to
$\beta^U$, where $\beta \in \Irr(\langle z \rangle)$ sends $z$ to $(-1)^j$. Note that
$$\beta^{\langle b \rangle} = \sum_{0 \leq t \leq q,~2|(t-j)}\lambda^t$$
contains both $\lambda^j$ and $\lambda^{-j}$. It follows that $[\chi_U,\mu]=1$, and so $(X,\chi)$ satisfies the inductive Feit condition by 
Theorem \ref{easysituations}(a).

\smallskip
(c2) Now we consider the case $\chi=\chi_j$ (in the notation of \cite{Do}) of degree $q+1$, $1 \leq j \leq (q-3)/2$, and assume \eqref{eq:sl213} in addition.
This character takes value 
$(-1)^j(q-1)$ at $z$, $1$ at $c \in C^{(0)}_2$, $\tvarep^{2jl}+\tvarep^{-2jl}$ at $a^l \in C_3^{(l,-l)}$ when $(q-1) \nmid 2l$, and $0$ at 
$b^m \in C_3^{((q-1)m)}$ when $(q+1) \nmid 2m$. Hence, $\chi$ extends to the character $\chi^{(t_0)}_{q+1}$ of $Y$, where
$$t_0=j(q+1)/4,~u_0=j(q-3)/4,$$
whose kernel again contains $Z_1$. Thus we can view $\chi^{(t_0,u_0)}_{q+1}$ as a character of $Y/Z_1 \cong M$. 

Now, if $2|j$ then $\chi^{(t_0,u_0)}_{q+1}$ is real-valued on $M$ and trivial at $z \in \ZB(M)$. Hence we can view it (by inflation) as a real character
$\chi^*$ of the group $\tM$ defined in \eqref{eq:sl212}.
However, if $2 \nmid j$, then $\chi^{(t_0,u_0)}_{q+1}$ takes values in $\RR \sqrt{-1}$ for any $y \in M$ with $\det{y}=-1$, and is faithful on $L$.
In this case, if $\Psi$ denotes a representation affording $\chi^{(t_0,u_0)}_{q+1}$, then 
$$\left\{\Psi(x), \sqrt{-1}\Psi(y) \mid x \in X,y \in M \smallsetminus X\right\}$$
(which is again a group because $\Psi(z)=-\mathrm{Id}$) defines a representation of $\tM$, whose corresponding character $\chi^*$ is real-valued. 
In  both cases we have 
\begin{equation}\label{eq:sl214a}
  \Q(\chi) = \Q(\tvarep^{2j}+\tvarep^{-2j}) = \Q(\cos \frac{2\pi j}{q-1}),~~\Q(\chi^*)=\Q(\tvarep^j+\tvarep^{-j}) = \Q(\cos \frac{\pi j}{q-1}).
\end{equation}  
Note that $\chi$ is $\tau$-invariant; furthermore, $\alpha$ acts on the subset $\{\chi_t \mid (q-1) \nmid 2t \}$ of $\Irr(X)$ via 
$$\chi_{t} \mapsto \chi_{tp}.$$
 
On the local side, we take $U = \langle a,u \rangle$ as defined in Lemma \ref{sl21-sub2}(i),   
which is self-normalizing and $\langle \tau,\alpha \rangle$-invariant. Then we embed $U$ in the subgroup
$\tUb = \langle \ta,u \rangle < \tL$ as defined in Lemma \ref{sl21-sub2}(i) and may assume that $\tau$ is induced by the conjugation by $\tilde a$. 
Let $\xi \in \Irr(\langle a \rangle)$ and 
$\tilde\xi \in \Irr(\langle \ta \rangle)$ be such that
$$\xi: a \mapsto \tvarep^2,~\tilde\xi: \ta \mapsto \varep.$$
Then the character $\mu=(\xi^{j})^U$ is irreducible (as $(q-1) \nmid 2j$), and sends $a^l$ to $\tvarep^{2jl}+\tvarep^{-2jl}$ when 
$(q-1) \nmid 2l$; in particular,
\begin{equation}\label{eq:sl215a}
  \Q(\mu) = \Q(\tvarep^{2j}+\tvarep^{-2j}) =\QQ(\chi).
\end{equation}  
Note that $\mu$ extends to the character $\mu^*=(\tilde\xi^{j(q+3)/2})^{\tUb}$ of $\tUb$ with
\begin{equation}\label{eq:sl216a}
  \Q(\mu^\ast) = \Q(\tvarep^{j(q+3)/2}+\tvarep^{-j(q+3)/2}) =\Q(\cos \frac{\pi j(q+3)/2}{q-1})= \QQ(\chi^*)
\end{equation}  
(since $(q+3)/2$ is coprime to $q-1$); in particular, 
$\mu$ is $\tau$-invariant. Furthermore, $\alpha$ acts on (a subset of) $\Irr(U)$ via 
$$(\xi^t)^U \mapsto (\xi^{tp})^U.$$ 
Thus, for $\Gamma=\Aut(X)_U$, we have shown  that the element 
$(\delta,\sigma) \in \Gamma \times \c G$, where $\delta \in \alpha^s \langle U,\tau \rangle$ with $s \in \ZZ$ and $\sigma$ induces 
the Galois automorphism $\tvarep^2 \mapsto \tvarep^{2k}$ with $k$ coprime to $q+1$, fixes $\chi$, respectively $\mu$,
if and only if $q-1$ divides $j(p^s k \pm 1)$. As a consequence,
$$(\Gamma \times \c G)_\chi=(\Gamma \times \c G)_\mu.$$
We also note that the restrictions of $\chi^*$ and $\mu^*$ to $\CB_{\tUb}(X) = \langle z \rangle$ are both multiples of 
the linear character $\beta$ sending $z$ to $(-1)^j$. 
Again, unlike the case of (c1), $[\chi_U,\mu]=3$; in fact, $\chi_U$ does not have multiplicity-one irreducible constituent with $\Q(\chi)$ as its 
field of values.

So far we have not used the assumption \eqref{eq:sl213}. Now we invoke it and define $\Sigma = \tM \rtimes \langle \alpha \rangle$ which
contains $X$ as a normal subgroup and induces the entire $\Aut(X)$ while acting on $X$. Then $\Delta = \NB_\Sigma(U) = \tUb \rtimes \langle \alpha \rangle$
induces $\Gamma_{\chi^{\c G}} = \Gamma$. Since $[\Sigma:\tM]=n=[\Delta:\tUb]$ is odd, we have $\OB^{2'}(\Sigma_\chi)=\tM$ and 
$\OB^{2'}(\Delta_\mu)=\tUb$. If $\nu$ denotes the unique linear character of order $2$ of $\tM$, then $\chi^*$ and $\nu\chi^*$ are the $2$ extensions of 
$\chi$ to $\tL$. Arguing as in (b2), we see that $\chi^*$ is a 
$\Sigma_\chi$-invariant real extension of $\chi$ to $\OB^{2'}(\Sigma_\chi)$, and similarly $\mu^*$ is a 
$\Delta_\mu$-invariant real extension of $\mu$ to $\OB^{2'}(\Delta_\mu)$. Next, consider any $(\delta,\sigma) \in (\Gamma \times \c G)_\chi$, 
where $\delta \in \alpha^{s} \langle U,\tau \rangle$ with $s \in \ZZ$. The same arguments as in (b2) show that the action of $(\delta,\sigma)$ on the set
$\{\chi^*,\nu\chi^*\}$ is the same as its action on the set $\{\mu^*,\nu\mu^*\}$. It follows from Theorem \ref{easysituations}(e) that 
$(X,\chi)$ satisfies the inductive Feit condition.
\end{proof} 

Next we study some general situation.

\begin{pro}\label{regular}
Let $\GC$ be a connected reductive group over a field of characteristic $p$, and let $X:=\GC^F$ for a Steinberg endomorphism 
$F:\GC \to \GC$. Let $\TC$ be an $F$-stable maximal torus of $\GC$ be such that $T:=\TC^F$ contains a Sylow $r$-subgroup 
$R$ of $G$ for some prime $r \neq p$, and $R$ contains a regular semisimple element $s$. 

\begin{enumerate}[\rm(a)]
\item Then $U:=\NB_X(R)=\NB_X(T) = \NB_X(\TC)$. In particular, $U$ satisfies condition {\rm \ref{inductive}(i)}.
\item Assume in addition that $\chi=\eps R^\GC_\TC(\theta)$ is irreducible for some $\eps=\pm$ and $\theta \in \Irr(T)$. Then 
$\mu:=\theta^U \in \Irr(U)$. Moreover, if every automorphism of $X$ extends to a bijective morphism of $\GC$, then 
$(\Gamma \times \c G)_\chi = (\Gamma \times \c G)_\mu$ for $\Gamma = \NB_{\Aut(X)}(U)$.
\end{enumerate}
\end{pro}

\begin{proof}
(a) Since $R = \OB_r(T)$, we have $\NB_X(\TC) \leq \NB_X(T) \leq \NB_X(R)$. Also, 
$$\TC \leq \CB_\GC(R)^\circ \leq \CB_\GC(s)^\circ = \TC,$$ so 
$\TC=\CB_\GC(R)^\circ$ is normalized by $\NB_X(R)$, i.e. $\NB_X(R) \leq \NB_X(\TC)$.

\smallskip
(b) By \cite[Corollary 9.3.1]{DM}, for any $F$-stable maximal torus $\TC'$ of $\GC$ and $\theta' \in \Irr((\TC')^F)$ we have 
\begin{equation}\label{eq:r10}
  [R^\GC_\TC(\theta),R^\GC_{\TC'}(\theta')] = \frac{1}{|T|}|\{ x \in \GC^F \mid \tw x\TC=\TC',~\tw x \theta = \theta'\}|.
\end{equation}  
Hence, the irreducibility of $\chi$ implies that $\theta$ has trivial stabilizer in $\NB_X(\TC)/T = U/T$. It follows from Clifford's theorem
that $\mu = \theta^U$ is irreducible. 

Now assume that every automorphism of $X$ extends to a bijective morphism of $\GC$. Then as in (a) we have 
$\NB_{\Aut(X)}(\TC) \leq \NB_{\Aut(X)}(T) \leq \NB_{\Aut(X)}(R)$. Furthermore, $\TC=\CB_\GC(R)^\circ$ is normalized by $\NB_{\Aut(X)}(R)$, so 
we obtain that $\NB_{\Aut(X)}(R)=\NB_{\Aut(X)}(T)=\NB_{\Aut(X)}(\TC)$. Next, any $x \in \NB_{\Aut(X)}(U)$ normalizes $R = \OB_r(U)$, and any 
$y \in \NB_{\Aut(X)}(R)$ normalizes $U = \NB_X(R)$. It follows that
\begin{equation}\label{eq:r11}
  \Gamma = \NB_{\Aut(X)}(U) = \NB_{\Aut(X)}(R)=\NB_{\Aut(X)}(T)=\NB_{\Aut(X)}(\TC).
\end{equation}  

Next we consider any pair $(n,\sigma^{-1}) \in (\Gamma \times \c G)$. Applying \eqref{eq:r10} to the pairs 
$(\TC,\tw \sigma\theta)$ and $(\tw n \TC,\tw n \theta)$,
and using the fact that both $R^\GC_{\tw n \TC}(\tw n\theta)$ and $R^\GC_{\TC}(\tw \sigma\theta)$ are irreducible up to sign in this situation, we see that
$(n,\sigma^{-1})$ fixes $\chi$ if and only if there is some $m \in X$ such that $\tw m \TC = \tw n \TC$ and $\tw m (\tw \sigma \theta) = \tw n \theta$. But
$\tw n \TC=\TC$ by \eqref{eq:r11}, so in fact $m \in \NB_X(\TC)=U$.

Suppose that the pair $(n,\sigma^{-1})$ fixes $\mu$, i.e. $\tw n \mu = \tw\sigma \mu$. By \eqref{eq:r11}, $n$ stabilizes $T$. Furthermore,
$\mu_T$ is the sum over the $U$-orbit of $\theta$. It follows that there is some 
$m \in U$ such that $\tw n \theta = \tw m (\tw \sigma \theta)$. Conversely, if there is some $m \in U$ such that $\tw n \theta = \tw m (\tw \sigma\theta)$, 
then induction to $U$ shows that $\tw n \mu = \tw\sigma \mu$, i.e.  $(n,\sigma^{-1})$ fixes $\mu$. 
We have therefore shown that $(\Gamma \times \c G)_\chi = (\Gamma \times \c G)_\mu$.
\end{proof}

\begin{thm}\label{ree}
Let $q=3^{2f+1}$ with $f \in \Z_{\geq 1}$. Then the simple group $S=\tw2 G_2(q)$ satisfies the inductive Feit condition.
\end{thm} 

\begin{proof}
Since $S$ has trivial Schur multiplier, we have $X=S$. 

\smallskip
(a) First we consider the ten characters $\xi_i$, $1 \leq i \leq 10$ of $X$, in the notation of \cite{W}. 
Let $P \in \Syl_3(X)$ and let $L:=\NB_X(P)$. As usual, if $\chi=\xi_1=1_X$, then we take
$(U,\mu)=(L,1_L)$. Suppose $\chi=\xi_2$, of degree $q^2-q+1$. By \cite[Lemma 1.3]{LM}, $\chi_L$ contains the unique character
$\Lambda_q$ of $L$ of degree $q^2-q$, with multiplicity $1$. Since both $\chi$ and $\mu$ are $\Gamma$-invariant and rational-valued, we are done by 
Theorem \ref{easysituations}(a).
Similarly, the Steinberg character $\xi_3$ (of degree $q^3$)  contains the unique character $\Lambda_1$ of $L$ of degree $q-1$ with multiplicity $1$.
We can also check by direct computation using \cite{W} that $\xi_4$, the only character of $X$ of degree $q(q^2-q+1)$, upon 
restriction to $L$ contains $I^-_q$, the only linear character of order $2$
of $L$, with multiplicity $1$. So we are done again in this case.

Next, $X$ has exactly two characters $\xi_5$, $\xi_7 = \overline{\xi_5}$ of degree $m(q-1)(q+3m+1)/2$, where $m:=3^f$.  By \cite[Lemma 1.3]{LM},
if $\chi=\xi_5$ or $\xi_7$ then $\chi_L$ has a unique irreducible constituent $\mu$ of degree $m(q-1)/2$, with multiplicity 1, and 
$\QQ(\chi)=\QQ(\mu)=\QQ(\sqrt{-3})$.
Since the odd-order group $\Gamma/U$ acts on $\{ \chi,\overline\chi\}$ and $\{\mu,\overline\mu\}$, we have $\Gamma_\chi=\Gamma=\Gamma_\mu$.  
Hence $(\Gamma \times \c G)_\chi = (\Gamma \times \c G)_\mu$ by Lemma \ref{easy2}(iv), and so we are done by Theorem \ref{easysituations}(a).
The same arguments apply to the pair $\xi_6$, $\xi_8 = \overline{\xi_5}$ of degree $m(q-1)(q-3m+1)/2$.

Using the values given in \cite{W} of $\xi_9$ and $\xi_{10}=\overline{\xi_9}$, the only two characters of degree $m(q^2-1)$ of $X$, one readily checks that
$$(\xi_9)_L = m\Lambda_q + \Lambda_m,~(\xi_{10})_L = m\Lambda_q + \overline\Lambda_m,$$
where $\Lambda_m$ and $\overline\Lambda_m$ are the only two characters of degree $m(q-1)$ of $L$. As $\Q(\xi_9)=\Q(\Lambda_m)=\Q(\sqrt{-3})$,
we can finish as above.

\smallskip
(b) All the remaining irreducible characters of $X$ are Deligne--Lusztig characters up to sign. Indeed, $\eta_r$ and $\eta'_r$ of degree $q^3+1$ in \cite{W}
correspond to the maximal torus $T_1$ of order $q-1$; $\eta_t$ and $\eta'_t$ of degree $(q-1)(q^2-q+1)$ 
correspond to the maximal torus $T_2$ of order $q+1$; $\eta^+_i$ of degree $(q^2-1)(q-3m+1)$ corresponds to $T_3$ of order 
$q+3m+1$; and $\eta^-_i$ of degree $(q^2-1)(q+3m+1)$ corresponds to $T_4$ of order 
$q-3m+1$, where $T_{1,2,3,4}$ are as listed in \cite[D.2]{H}. 
Each of these characters $\chi$ is the semisimple character $\chi_s$ labeled by a semisimple element $s$ in the dual group which is 
isomorphic to $X$. As $s$ is real, $\chi=\chi_s$ is also real-valued. Furthermore, since $q \geq 27$, we can check that each $T_i$, $1 \leq i \leq 4$,
contains a Sylow $r$-subgroup $R$ of $X$, for a suitable prime $r \neq 2,3$, and $R$ contains a regular semisimple element $t$ (indeed with
$\CB_X(t)=T$). We also note that the outer automorphisms of $X$ are field automorphisms and hence extend to bijective morphisms of the simple
algebraic group $\GC=G_2$ if we view $X=\GC^F$. 
Applying Proposition \ref{regular}, we can choose $(U,\mu)$ with $U = \NB_X(R)$ and $(\Gamma \times \c G)_\chi=(\Gamma \times \c G)_\mu$;
in particular, $\mu$ is real-valued as so is $\chi$. Finally, as $|\Gamma/U|$ is odd, taking $\Sigma = \Aut(X) \cong X \cdot C_{2f+1}$, we are done 
by Theorem \ref{easysituations}(f).
\end{proof}

\begin{proof}[Proof of Theorem D]
By Corollary B, it suffices to prove that every non-abelian simple group with abelian Sylow $2$-subgroups satisfy the inductive Feit condition.
Using Walter's classification \cite{Wa}, and Theorems \ref{spo}, \ref{sl22},
\ref{sl21}, and \ref{ree}, we are done.
\end{proof}


\section{Conjecture E}
In this final section, we explore some suitable generalizations of Feit's conjecture.

\begin{lem}
Suppose that $G$ is a finite group and $\chi \in \irr G$. Then the following are equivalent:

\begin{enumerate}[\rm(a)]

\item
There is a subgroup $H$ and a linear irreducible constituent $\lambda$ of $\chi_H$ such that $c(\chi)$ divides $o(\lambda)$.

\item
There is a cyclic subgroup $H$ and a linear irreducible constituent $\lambda$ of $\chi_H$ such that $c(\chi)$ divides $o(\lambda)$. 

\item
There is a cyclic subgroup $H$ and a linear irreducible constituent $\lambda$ of $\chi_H$ such that $c(\chi)=o(\lambda)$.

\end{enumerate}
\end{lem}

\begin{proof}
It suffices to show that (a) implies (c).
 Suppose that there is an irreducible linear constituent $\mu$ of $\chi_H$ such that $c(\chi)=n$ divides $o(\mu)=m$. Let $K=\ker\mu$
and let $h \in H$ such that $\langle h \rangle K=H$. Let $D=\langle h\rangle$ and notice that $D/D\cap K$ has order $m$. Now, if $D\cap K \sbs C \le D$
and $|C:D\cap K|$ has order $n$, it follows that $\lambda_C$ is an irreducible constituent of $\chi_C$ with $o(\lambda_C)=c(\chi)$.
\end{proof}

\begin{lem}\label{roots}
Suppose that $G$ is a cyclic group of odd order, and let $u,v \in G$. Then there exists an integer $k$ such that
$(k,o(v))=1$ and $o(u)$ divides the order of $v^ku$. Hence, if $\lambda, \zeta$ are odd order linear characters of a cyclic group $C$,
 then there is a Galois conjugate $\lambda^\sigma$ such that $o(\zeta)$ divides $o(\lambda^\sigma \zeta)$
\end{lem}

\begin{proof}
Suppose that $|G|=p_1^{f_1} \cdots p_n^{f_n}$. Suppose that for each $i$, there is $k_i$ coprime to $p_i$ such that
$o(u_{p_i})$ divides the order of $(v_{p_i})^{k_i}u_{p_i}$. By the Chinese Remainder Theorem, let $k$ be an integer
such that $k \equiv k_i$ mod $(p_i)^{f_i}$.  Then $v_{p_1}^{k_1} \cdots v_{p_n}^{k_n}=v^k$, and therefore we may assume that $G$
has $p$-power order.  We may also assume that $\langle u, v \rangle=G$. 
Suppose that there is no $k$ coprime to $p$ such that $\langle u\rangle $ is not contained in $\langle v^k u\rangle$.
Since $G$ is a cyclic $p$-group, it follows that $\langle v u\rangle \le \langle u\rangle$. Thus $G=\langle u \rangle$.
Write $v=u^t$ for some $t$. We want to show that there is $(k,p)=1$ such that  $tk+1$ is coprime to $p$.
If $p$ divides $t$, then we take $k=1$. If $p$ does not divide $t$, then let $k$ be such that $kt \equiv 1$ mod $p$. Then
$kt +1 \equiv 2$ mod $p$, and therefore $kt+1$ is coprime to $p$ since $p$ is odd.
For the final part, consider the cyclic group $G$ of linear characters of $C/C_2$, where $C_2$ is the Sylow 2-subgroup of $C$.
\end{proof}

%
%
%
%
 \medskip
 
An essential tool in the proof of Conjecture E for solvable groups is, again, the use of Isaacs' character $\psi$ associated to a
fully ramified section, which we already used in Theorem \ref{isaacs}. We need extra information on Isaacs character $\psi$. For the reader's convenience, we review some of the basic ingredients.
Recall that if $N \triangleleft G$, $\theta \in \irr N$ is $G$-invariant
and  $\theta^G=e\chi$ for some $\chi \in \irr G$, we say  in this case that
$\theta$ (or $\chi$) is fully ramified in $G$. Under these condition, it is easy to prove that $e^2=|G:N|$.
Suppose now that $L,K$ are normal subgroups of $G$, $K/L$ abelian, and $\varphi \in \irr L$ is $G$-invariant with
$\varphi^K=e\theta$, that is, $\varphi$ is fully ramified in $K$.  In Isaacs terminology, this is
a character five $(G,K,L,\theta, \varphi)$. (See Section 3 of \cite{Is73}.)
Whenever we have a character five, the main goal is to find a {\sl good} complement $U$ of $K/L$ in $G$
and a relationship between $\irr{G|\theta}$ and $\irr{U|\varphi}$. This can be successfully done if
$G/L$ has odd order or if $(|G/K|, |K/L|)=1$, which will be our case. 
Notice that in the latter, the existence of a unique conjugacy class of
complements $U$ of $K/L$ in $G$ is guaranteed by the
Schur--Zassenhaus theorem.

\begin{lem}\label{fully}
Suppose that $L, K$ are normal subgroups of $G$, with $L\le K$, such that
$K/L$ is an abelian $p$-group and $G/K$ is
a $p'$-group. Let $\varphi \in \irr L$ be $G$-invariant, $\theta \in \irr K$ such that
$\varphi^K=e\theta$, where $e^2=|K:L|$. Let $U/L$ be a complement of $K/L$ in $G/L$, and
let $\psi \in {\rm Char}(U/L)$ be the Isaacs canonical character associated with $(G,K,L,\theta, \varphi)$.
Let $u \in U$ and suppose that $p$ is odd.  
Then $\psi_{\langle uL\rangle}$ has a linear irreducible constituent $\epsilon$ with $\epsilon^2=1$. 
\end{lem}

\begin{proof} 
Let $E=K/L$ which carries a complex valued multiplicative form by Section 2 of \cite{Is73}.
 Using Theorem 4.7 of \cite{Is73}, we may assume that  $(G,K,U,L,\theta, \varphi)$ is a constellation.
 (See Definition 4.5 of \cite{Is73}).
Using Theorem 6.3 of \cite{Is73}, we obtain that $\psi$ is the canonical character associated with the constellation.
We may assume that $U/L=\langle uL \rangle$. 
By Lemma 5.6 of \cite{Is73}, we have $V:=\ker\psi=\cent{U}{K/L}$.
Arguing by induction on $|K:L|$, and using Lemma 5.5 of \cite{Is73},  we may assume that 
$\cent{K/L}x=1$ for all $1 \ne xV \in U/V$.  By Theorem 5.7 of \cite{Is73}, we have $\psi(x)=\pm 1$
for all $1 \ne xV \in U/V$. Problem 4.9 in \cite{N} now implies
$$\bar\psi=f\rho + \epsilon \lambda \, $$
where $\bar\psi$ is the character $\psi$ viewed as a character of $U/V$,
$f \ge 0$ is an integer,  $\rho$ is the regular character of $U/V$, $\lambda \in \irr{U/V}$, and $\epsilon\in\{\pm1\}$.
If $1 \ne xV \in U/V$, by evaluating, we deduce that $\lambda(x)=\pm 1$. Thus $\lambda^2=1$. 
Hence, if $\psi$ does not contain any linear character $\mu$ with $\mu^2=1$, we deduce that
$f=1$, $\lambda=1$, $\epsilon=-1$ and $|U:V|=m$ is odd.   
Hence $p^{n}=m-1$, and $p=2$.
\end{proof}

\begin{thm}\label{9.4}
If $G$ is solvable and $\chi \in \irr G$,
then there exists a cyclic subgroup $H$ of $G$ and an irreducible
constituent $\lambda \in \irr H$ of $\chi_H$ such that $c(\chi)$ divides $o(\lambda)$.
\end{thm}

\begin{proof}
We argue by induction on $|G|$. If $\chi=\psi^G$, where
$\psi \in \irr H$ and $H<G$, then $c(\chi)$ divides $c(\psi)$, and we are done by induction.
 Hence, we may assume that $\chi$
is primitive. We also may assume that $\chi$ is faithful.

If $G$ is nilpotent, then $\chi$ is monomial, and therefore linear. Since $\chi$ is faithful, $G$ is cyclic, and we are done. 
Let $K$ be the smallest normal subgroup of $G$ such that $G/K$ is nilpotent,
let $K/L$ be a $p$-chief factor of $G$, and let $U=LC$, where $C$ is a Carter subgroup
of $G$. Then $KU=G$ and $K\cap U=L$. 

Since $\chi$ is primitive, we can write $\chi_K=f\theta$ and $\theta_L=e\varphi$, where $\theta \in \irr K$ and $\varphi \in \irr L$.
By the Going Down Theorem 6.18 of \cite{Is}, either $e=1$ or $e^2=|K/L|$.
If $e=1$ then we are done by Lemma \ref{res}.
So we may assume that $\varphi^K=e\theta$, where $e^2=|K:L|$.
 
Let $P/L$ be a Sylow $p$-subgroup of $G/L$. Then $P\nor G$, since $G/K$ is nilpotent.
Now, $K/L \cap  \zent {P/L} >1$, and therefore, $K/L \sbs \zent{P/L}$. This implies that $R=P\cap U \nor U$ and
is normalized by $K$, so $R\nor G$. 
 
Write $\chi_P=v\tau$ and $\chi_R=w\mu$. Using again the Going Down Theorem and Lemma \ref{res}, 
we may assume that $\tau$ and $\mu$ are fully ramified,
but now we have that $G/P$ and $P/R$ are coprime. Assume first that $p$ is odd.
Then we are in the situation of Lemma \ref{fully}.
Now we use Theorem \ref{isaacs}. Let $\psi$ be the associated Isaacs character
and write $\chi_U=\psi \chi^\prime$, where $c(\chi^\prime)=c(\chi)=c$.
By induction, there is a cyclic subgroup $V \sbs U$ and a linear  irreducible constituent $\mu$
of $\chi^\prime_V$ such that $c(\chi)=c=o(\mu)$. 
Now, by Lemma \ref{fully}, the character $\psi_{VR}$ has a linear constituent $\epsilon \in \irr{VR/R}$
with $\epsilon^2=1$. Altogether, we have
$$\chi_{VR}=\psi_{VR} \chi^\prime_{VR}=\epsilon \chi^\prime_{VR} + \Delta \, ,$$
where $\Delta$ is some character or zero. Hence, $\chi_V$ has the linear constituent $\epsilon\mu$
and $\Q(\chi) \sbs \Q(\mu)=\Q(\epsilon_V\mu)$.  We are done in this case.

So we may assume that $p=2$.  Hence, $U/R$ has odd order.   Write $\mu=\delta\zeta$, where $o(\delta)=c_2$ is a 2-power and $o(\zeta)=c_{2'}$ is odd. Write $V=\langle u \rangle$.
Since $\psi$ is rational valued,
we can write $\psi_{VR}=e_1O(\lambda_1) + \cdots + e_tO(\lambda_t)$,
where $\lambda_i$ are linear (odd order) characters of $VR/R$ and $O(\lambda_i)$ is the sum of
the Galois orbit of $\lambda_i$.  
By Lemma \ref{roots}, there is a Galois conjugate $((\lambda_1)_V)^\sigma$ such that the order of
$((\lambda_1)_V)^\sigma\zeta$ is an odd multiple of $c_{2'}$. Hence, the order of $\delta ((\lambda_1)_V)^\sigma\zeta$ 
is divisible by $c$ and it is an irreducible constituent of $\chi_V$. 
\end{proof}

\begin{cor}\label{overab}
Suppose that $A \nor G$ is abelian, and $G$ is solvable. Suppose that $\chi \in \irr G$. Then
there exists $A \le H \le G$ and $\lambda \in \irr H$ linear, such that $H/A$ is cyclic, $[\chi_H, \lambda] \ne 0$ and $\Q(\chi) \sbs \Q(\lambda)$.
\end{cor}

\begin{proof}
We argue by induction on $|G:A|$. The case $|G:A|=1$ is trivial.
Let $\mu \in \irr A$ be under $\chi$, and let $\tau \in \irr{G_\mu|\mu}$ be the Clifford correspondent of $\chi$. 
Assume first that $G_\mu<G$. Then there exists $A \le H \le G_\mu$ and $\lambda \in \irr H$ linear, such that $H/A$ is cyclic, $[\tau_H, \lambda] \ne 0$ and $\Q(\tau) \sbs \Q(\lambda)$. Since $\Q(\chi) \sbs \Q(\tau)$, we are done in this case.
So we may assume that $\chi_A=\mu$. We may also assume that $\mu$ is faithful so $A \sbs \zent G$.
By the previous theorem there is $H$ cyclic such that $[\chi_H, \lambda] \ne 0$ and $\Q(\chi) \sbs \Q(\lambda)$.  Now, let $\delta \in \irr{AH}$ be an irreducible constituent of $\chi_{AH}$ over $\lambda$.
Then $\delta_H=\lambda$ and $\delta_A$ is irreducible. Also, $\Q(\lambda) \sbs \Q(\delta)$ and we are done. 
\end{proof}

Corollary \ref{overab} is not true if $A$ is not abelian. If $G={\tt SmallGroup}(32,42)$, then
$\chi$ has an irreducible character of degree 2, with $\Q(\chi)=\Q_8$, a normal subgroup $N={\sf Q}_8$
such that $G/N=C_2 \times C_2$, $\chi_N=\theta \in \irr N$ but no subgroup $H/N\le G/N$ of order 2 contains an 
irreducible constituent of $\chi_H$ containing $\Q_8$ in its field of values.
This fact made the proof of our main theorem much more difficult.
Notice that it would make sense for our purposes to ask if Corollary \ref{overab} holds for the weaker conclusion 
that $c(\chi)$ divides $c(\lambda)$. Again this is false, but the counterexample is more difficult to find.
If $U={\sf C}_{60} \times {\sf C}_2 \times {\sf C}_2$, then there is a subgroup $P$ of order 16
of the Sylow 2-subgroup of the automorphism group of $U$ such that the 
semidirect product $G:=U\rtimes P$ has an irreducible character $\chi$
with field of values ${\rm NF}(60,[1,23,47,49])$ (in \cite{GAP} notation), thus $c(\chi)=60$. 
Moreover, $G$ has a normal subgroup $N$ of such that $G/N\cong {\sf C}_2 \times {\sf C}_2$,
$\chi_N=\theta$ is irreducible and rational valued, but the restriction of the proper subgroups
between $N$ and $G$ have conductors 12, 15 and 20.
 
\medskip
In nilpotent groups, Theorem \ref{9.4} can be improved dramatically. We  prove next that we can find a linear $\lambda \in \irr H$ such that $c(\chi)$ divides $o(\lambda)$
and $[\chi_H, \lambda]=1$. This is no longer true in general: for instance in $G=6.{\sf A}_7$ there does not even exist $(H, \lambda)$
such that $[\chi_H, \lambda]=1$ and $\lambda$ is linear. It is not clear if this holds or not in solvable groups.

\begin{thm}\label{induced}
Suppose that $G$ is a $p$-group and $\chi \in \irr G$ has conductor $c(\chi)=p^e>1$. 
Then there exists a linear character $\lambda$ such that $\lambda^G=\chi$ and $o(\lambda)=p^e$.
\end{thm}
 
\begin{proof}
We argue by induction on $|G|$. We may assume that $\chi$ is faithful.

Let $A\nor G$ be an arbitrary abelian normal subgroup of $G$. We claim that we may assume that $A$ is cyclic.
Let $\lambda \in \irr A$ be an irreducible constituent of the restriction $\chi_A$.
Let $T=G_\lambda$ be the stabilizer of $\lambda$ in $G$, and let $\psi \in \irr{T|\lambda}$ be the Clifford
correspondent of $\chi$ over $\lambda$. Further, let 
$T^*= \{ g \in G |\lambda^g=\lambda^\sigma$ for some $\sigma \in {\rm Gal}(\Q(\lambda)/\Q)\}$ and set 
$\eta:=\psi^{T^*}$. Then $c(\eta)=c(\chi)$. By induction, we may assume that $T^*=G$.
But then $\ker\lambda \nor G$ and, since $\chi$ is faithful, this implies that $\ker\lambda=1$.
Hence $A$ is cyclic. 

So, we may assume that every normal abelian subgroup of $G$ is cyclic. By \cite[Theorem~4.10, Chapter 5]{G}
this implies that $G$ is cyclic or $p=2$ and $G$ is quaternion, dihedral or semidihedral.
So we may assume that $p=2$. Suppose that $G$ is dihedral, semidihedral or quaternion  of order $2^n$, where $n>3$.
Since $\chi$ is faithful, we easily check that $c(\chi)=2^{n-1}$.
Also, $\chi$ is induced from a faithful linear character of the cyclic subgroup of order $2^{n-1}$.
If $n=3$, then the non-linear characters of $G$ are rational valued.
\end{proof}

Notice that Theorem \ref{induced} is not true if $c(\chi)=1$, as $G={\sf Q}_8$ shows us.

\medskip
Several other reasonable generalizations of Feit's conjecture simply do not hold.
For instance, an easy observation is that if  $G$ is nilpotent and $\chi \in \irr G$, 
then there is $g \in G$ such that $c(\chi)=c(\chi(g))$.
This fact is not true in general and  the smallest example is ${\tt SmallGroup}(960,730)$. 
However, if $G$ is a sporadic simple group, one can observe for every $\chi\in \irr G$ the astonishing
equality $\{ c(\chi(g))\mid g \in G\} = \{1, c(\chi)\}$.

If $x \in G$, let us define $\Q(x)=\Q(\chi(x) \, |\, \chi \in \irr G)$. 
Another plausible generalization of Feit's conjecture would be the following:
{\it Given $\chi \in \irr G$, is there $x \in G$ such that $\Q(\chi) \sbs \Q(x)$}? Since $\Q(x) \sbs \Q_{o(x)}$, this would imply Feit's conjecture.
However, this question has a negative answer, as shown by ${\tt SmallGroup}(32,15)$. It might be interesting to notice that we have not found
a counterexample if $G$ is a simple group. (In fact, computations suggest that except in $Fi_{24}'$, the actions of $\c G$ on $\irr G$
and the set ${\rm cl}(G)$ of conjugacy classes of a simple group $G$ are permutation isomorphic.)

\medskip
We have not been able to decide if the following relative version of Feit's conjecture holds:
{\it Suppose that $G$ is a finite group, $N\nor G$, $\chi \in \irr G$, and $\theta \in \irr N$
is an irreducible constituent of $\chi_N$. Is there a subgroup $H/N\le G/N$ and an irreducible constituent
$\hat\theta \in \irr H$ of $\chi_H$ such that $\hat\theta_N=\theta$  and $\Q(\chi) \sbs \Q(\hat\theta)$}?
 As mentioned above, if we further impose that $H/N$ is cyclic, then this is false for $G={\tt SmallGroup}(32,42)$. Perhaps some of the above plausible generalizations hold true if $\chi$ is assumed to be non-linear and primitive.
 
\medskip

We finally remark that although the modular analogue of Feit's conjecture  for irreducible Brauer characters of $p$-solvable groups  follows from Feit's conjecture (since every   
$\varphi \in \ibr G$ has a canonical lifting $\chi \in \irr G$ and therefore with $\Q(\chi)=\Q(\varphi)$), it does not hold in general. For instance, $2 \cdot {\sf A}_6$ has $3$-modular irreducible characters with conductor $40$, and no elements of that order.

\end{document}